\documentclass[11pt]{article}

\usepackage{amsmath}

\usepackage{latexsym}
\usepackage{amssymb}
\usepackage{amsfonts}\large
\usepackage{amsthm}
\usepackage{fontenc}
\usepackage{fullpage}
\usepackage{enumerate}
\usepackage{esint}
\usepackage{bm}
\usepackage{hyperref}
\usepackage{titling}

\hypersetup{colorlinks=true}
\usepackage{xcolor}
\usepackage[normalem]{ulem}
\usepackage{marginnote}
\usepackage{enumerate}

\DeclareMathOperator{\dive}{div}

\DeclareMathOperator{\e}{\varepsilon}
\DeclareMathOperator{\diam}{diam}
\DeclareMathOperator{\dist}{dist}

\DeclareMathOperator{\loc}{loc}
\DeclareMathOperator{\bR}{\mathbb R}
\DeclareMathOperator{\M}{M}
\DeclareMathOperator{\C}{C}
\DeclareMathOperator{\Tr}{Tr}

\providecommand{\ip}[1]{\langle#1\rangle}

\newtheorem{thm}{Theorem}[section]
\newtheorem{lemma}[thm]{Lemma}
\newtheorem{prop}[thm]{Proposition}

\newtheorem{dfn}[thm]{Definition}

\theoremstyle{definition}
\newtheorem{remark}[thm]{Remark}

\numberwithin{equation}{section}

\begin{document}
\title{Boundary value problems in Lipschitz domains \\ for equations with lower order coefficients}

\author{Georgios Sakellaris \thanks{Department of Mathematics, Universitat Aut{\`o}noma de Barcelona, Bellaterra 08193, Barcelona, Spain\newline
\hspace*{21pt}\textit{Email address}: \texttt{gsakellaris@mat.uab.cat} \newline
\hspace*{20pt}The author has received funding from the European Union's Horizon 2020 research and innovation programme under Marie Sk{\l}odowska-Curie grant agreement No 665919, and is partially supported by MTM-2016-77635-P (MICINN, Spain) and 2017 SGR 395 (Generalitat de Catalunya).}}

\date{\vspace{-5ex}}
\maketitle
\begin{abstract}
We use the method of layer potentials to study the $R_2$ Regularity problem and the $D_2$ Dirichlet problem for second order elliptic equations of the form $\mathcal{L}u=0$, with lower order coefficients, in bounded Lipschitz domains. For $R_2$ we establish existence and uniqueness assuming that $\mathcal{L}$ is of the form $\mathcal{L}u=-\text{div}(A\nabla u+bu)+c\nabla u+du$, where the matrix $A$ is uniformly elliptic and H{\"o}lder continuous, $b$ is H{\"o}lder continuous, and $c,d$ belong to Lebesgue classes and they satisfy either the condition $d\geq\dive b$, or $d\geq\dive c$ in the sense of distributions. In particular, $A$ is not assumed to be symmetric, and there is no smallness assumption on the norms of the lower order coefficients. We also show existence and uniqueness for $D_2$ for the adjoint equations $\mathcal{L}^tu=0$.
\end{abstract}

\section{Introduction}
In this paper we are interested with solvability of the Dirichlet and Regularity boundary value problems in bounded Lipschitz domains $\Omega\subseteq\bR^n$, where $n\geq 3$, for operators of the form
\[
\mathcal{L}u=-\dive(A\nabla u+bu)+c\nabla u+du.
\]
We will assume that the matrix $A$ is uniformly elliptic: there exists a constant $\lambda>0$ such that, 
\begin{equation}\label{eq:el}
\ip{A(x)\xi,\xi}\geq\lambda|\xi|^2,\quad\forall x\in\Omega,\,\,\forall\xi\in\bR^n.
\end{equation}
We will also assume that $A$ is H{\"o}lder continuous: that is, for some $\alpha\in(0,1)$ and $\tau>0$,
\begin{equation}\label{eq:Hold}
|A(x)-A(y)|\leq \tau|x-y|^{\alpha}.
\end{equation}
The set of matrices $A$ defined in $\Omega$ which satisfy \eqref{eq:el} and \eqref{eq:Hold} will be denoted by $\M_{\Omega}(\lambda,\alpha,\tau)$. Note that we do not assume that $A$ is symmetric.

For the lower order coefficients, we assume that $b,c$ are functions with values in $\bR^n$, and $d$ is a real valued function. If $b=(b_1,\dots b_n)$ satisfies \eqref{eq:Hold} in $\Omega$, we write $b\in\C_{\Omega}(\alpha,\tau)$. Moreover, we will assume that $c,d$ belong to $L^p$ for some $p>n$. We will also assume that $d\geq\dive b$, or $d\geq\dive c$ in the sense of distributions.

For a Lipschitz domain $\Omega$, let $B_{\Omega}$ be a ball centered at a point in $\Omega$ with radius $10\diam(\Omega)$. The first main theorem of this paper is on solvability of the $R_2$ Regularity problem.

\begin{thm}\label{R_2Solvability}
Let $\Omega\subseteq\bR^n$ be a Lipschitz domain, and suppose that $A\in M_{B_{\Omega}}(\lambda,\alpha,\tau)$. Suppose also that $b\in \C_{B_{\Omega}}(\alpha,\tau)$, $c,d\in L^p(B_{\Omega})$ for some $p>n$, with either $d\geq\dive b$ or $d\geq\dive c$. Then, for every $f\in W^{1,2}(\partial\Omega)$, there exists a unique $u\in W^{1,2}(\Omega)$ such that $-\dive(A\nabla u+bu)+c\nabla u+du=0$ in $\Omega$, $(\nabla u)^*\in L^2(\partial\Omega)$ and $u\to f$ nontangentially, almost everywhere on $\partial\Omega$. Moreover,
\[
\|(\nabla u)^*\|_{L^2(\partial\Omega)}\leq C\|f\|_{W^{1,2}(\partial\Omega)},
\]
where $C$ depends on $n,p,\lambda,\alpha,\tau$,$\|A\|_{\infty}$,$\|b\|_{\infty}$,$\|c\|_p$,$\|d\|_p$, $|\Omega|$ and the Lipschitz character of $\Omega$.
\end{thm}

The definitions of nontangential convergence and the nontangential maximal function $(\nabla u)^*$ are the usual ones (appearing after \eqref{eq:Cone} and at \eqref{eq:MaxFun}, respectively). 

We will also show the next result on solvability of the $D_2$ Dirichlet problem, which is the second main theorem of this paper.

\begin{thm}\label{D_2Solvability}
Under the same assumptions as in Theorem~\ref{R_2Solvability}, for every $f\in L^2(\partial\Omega)$, there exists a unique $u\in W_{\loc}^{1,2}(\Omega)$ such that $-\dive(A\nabla u+cu)+b\nabla u+du=0$ in $\Omega$, $u^*\in L^2(\partial\Omega)$ and $u\to f$ nontangentially, almost everywhere on $\partial\Omega$. In addition,
\[
\|u^*\|_{L^2(\partial\Omega)}\leq C\|f\|_{L^2(\partial\Omega)},
\]
where $C$ depends on $n,p,\lambda,\alpha,\tau$,$\|A\|_{\infty},\|b\|_{\infty}$,$\|c\|_p$,$\|d\|_p$, $|\Omega|$ and the Lipschitz character of $\Omega$.
\end{thm}

A few remarks follow.

\begin{remark}
To show our theorems, we rely both on the techniques and the results in the paper \cite{KenigShen} (for their treatment of the small scale), but with a couple of modifications. First, we develop an analog of the estimate in Lemma 2.7 in \cite{KenigShen} by only comparing with the fundamental solution (Lemma~\ref{PointwiseDifDifOnGrad}), since we will then reduce to the results in \cite{KenigShen} (Lemma~\ref{GeneralPerturbation}). In addition, we consider the approximations in \eqref{eq:A^rho} ((7.3) in \cite{KenigShen}), but we treat the conditions $d\geq\dive c$ and $d\geq\dive b$ separately. In the first case we adapt the arguments in \cite{KenigShen} by first showing the Rellich estimate. On the other hand, for the second case we rely on the first case, treating $c$ as a perturbation of $0$ close to $\partial\Omega$. In both cases we first reduce to the case $d=0$ (using Lemma~\ref{DivOperator}) so as not to deal with the divergence $\dive b^{\rho}$ of the modifications $b^{\rho}$ in \eqref{eq:A^rho}, and we adapt the three-step approximation argument in \cite{KenigShen}.
\end{remark}

\begin{remark}
One of the main difficulties in showing our estimates is the lack of pointwise bounds for $\nabla_yG(x,y)$, where $G$ is Green's function for the operator in Theorem~\ref{R_2Solvability}. For this reason, we need to use $L^p$ and weak-$L^p$ estimates (as in Lemma~\ref{PointwiseGreen}) to obtain the analogs of the estimates in \cite{KenigShen}.
\end{remark}

\begin{remark}
The equations we consider are not scale-invariant, so we use the construction of Green's function $G$ for our operator in $B_{\Omega}$ from \cite{KimSak}, and then $G$ serves as a variant of the fundamental solution in order to define the single layer potential $\mathcal{S}f(q)=\int_{\partial\Omega}G(q,q')f(q')\,d\sigma(q')$. We then show that $\mathcal{S}:L^2(\partial\Omega)\to W^{1,2}(\partial\Omega)$ is invertible (Theorems~\ref{InvForc} and \ref{InvForb}), leading to Theorem~\ref{R_2Solvability}. Next, we consider the adjoint operator $\mathcal{S}^*:W^{-1,2}(\partial\Omega)\to L^2(\partial\Omega)$, and combined with invertibility of $\mathcal{S}$, we deduce Theorem~\ref{D_2Solvability} (for more on the connection between $\mathcal{S}^*$ and $D_2$, we refer to \cite{Hofmannduality}). We also note that we have not pursued the direction of showing our theorems when the coefficients are only defined in $\Omega$, which requires the construction of extensions $\tilde{b}$ of vector valued functions $b$ with $\dive b\leq d$, such that $\tilde{b}$ satisfy a similar property.
\end{remark}

\begin{remark}
The reason for the assumed regularity of the coefficients is that, under our assumptions, gradients of solutions to $\mathcal{L}u=0$ are locally H{\"o}lder continuous. In turn, this implies the bound $|\nabla_xG(x,y)|\leq C|x-y|^{1-n}$, where $G$ is Green's function for $\mathcal{L}$ in a ball $B$ (from \cite{KimSak}), which is a crucial assumption in showing that the kernel $G(x,y)$, restricted on $\partial\Omega$, is Calder{\'o}n-Zygmund, thus allowing us to show the $L^2(\partial\Omega)\to W^{1,2}(\partial\Omega)$ boundedness of the single layer potential operator.
\end{remark}

\begin{remark}
The condition $d\geq\dive b$ implies that the maximum principle holds for solutions $u$ to  $\mathcal{L}u=0$, where $\mathcal{L}$ is the operator in Theorem~\ref{R_2Solvability}, while the condition $d\geq\dive c$ implies the same for the operator $\mathcal{L}^t$. The reason for using those two conditions is that we are using Green's function $G$ which was constructed in \cite{KimSak}, and which was carried out under those conditions.
\end{remark}

\begin{remark}
After showing the Rellich estimate, it turns out that we need to bound a term of the form $\displaystyle\int_{\Omega}|c||\nabla u|^2$ (Lemma~\ref{GlobalRellich}). The main ingredient that allows us to estimate this term when $c\in L^p$, where $p>n$, is a higher integrability result on derivatives of solutions (Lemma~\ref{qImproved}) when the coefficients satisfy a special condition (Condition~\ref{eq:ACond}). This is done by showing that $\nabla u\in B_{\beta}^{2,2}(\Omega)$ for $\beta\in\left(0,\frac{1}{2}\right)$ (the Besov space, Lemma~\ref{ImprovedIntegrability}), which we deduce using Theorem 4.1 in \cite{JerisonKenigInhomogeneous}. 
\end{remark}

We also remark that the results presented here are generalizations of some results in the author's PhD thesis \cite{thesis}, in which it was assumed that $A$ was Lipschitz continuous, and either $b\in L^{\infty}$ and $c=0$, $d=0$, or $c\in L^{\infty}$ and $b=0$, $d=0$.

The method of layer potentials for boundary value problems in Lipschitz domains was used by Verchota in \cite{Verchota}, who studied the Dirichlet and Regularity problems for the Laplacian in Lipschitz domains, based on the $L^p$-boundedness of the Cauchy integrals on Lipschitz curves \cite{Coifman}. The Rellich-type estimate was used by Jerison and Kenig \cite{JerisonKenigIdentity}, who also treated the Dirichlet and the Regularity problem for the Laplacian (\cite{JerisonKenigNonsmooth}, \cite{JerisonKenigN2R2}). The related literature on the connection between layer potentials and boundary value problems is vast and we do not intend to review it here; we refer to the introduction of \cite{HofmannAnalyticity} for some further main developments in the area. 

Assume that $\Omega$ is the unit ball. In the case that $A$ is symmetric, elliptic and bounded, it is shown in \cite{JerisonKenigNonsmooth} that $D_2$ for $-\dive(A\nabla u)$ is solvable if $A$ is independent of the radial direction (or, assuming that $A$ is continuous and the modulus of continuity along some transversal direction satisfies a square Dini condition, from \cite{FabesJerisonKenigNecessary}). The same independence guarantees that $R_2$ is solvable as well, from \cite{KenigPipherNeumann}. However, if $A$ is non-symmetric, independence of $A$ from some transversal direction to the boundary does not suffice for solvability, as it is shown in (3.2) of \cite{KenigKochPipherToro} and the appendix of \cite{KenigRule} for $D_2$ and $R_2$, respectively. For positive results for $D_2$ and $R_2$ when $A$ is non necessarily symmetric, we refer to \cite{HofmannAnalyticity}, \cite{NguyenPaper} and the references therein.

Towards the direction of homogenization of elliptic boundary value problems for equations we refer to \cite{KenigShenHomogenization}, and to \cite{KenigShen} for elliptic systems (as well as their introductions for more references). A more recent work which allows lower order coefficients is also the paper \cite{XuZhaoZhou}, in which the authors assume that $b,c$ are H{\"o}lder continuous, $d$ is bounded, and the bilinear form for $\mathcal{L}$ is coercive. Solvability results for operators $\mathcal{L}_{\e}$ in homogenization subsume the analogous results on solvability for $\mathcal{L}=\mathcal{L}_1$, but we believe that the problems we consider in this paper have not been treated before.

A summary of this paper now follows. In Section 2 we discuss the preliminaries for studying the problems we are considering, showing some lemmas about Lipschitz domains. In Section 3 we show various a priori estimates, including local regularity of derivatives of solutions to $\mathcal{L}u=0$, a bound of solid integrals of solutions by surface integrals, the Rellich estimate, and a global integrability result on gradients of solutions. We remark that we bypass the assumption of symmetry of $A$ in order to deduce the Rellich estimate, using and integration by parts argument from \cite{PipherDrifts}, which reduces our equation to an equation with symmetric principal part and a drift. In Section 4 we deal with various estimates on Green's function when we let the coefficients vary, leading to a comparison between differences of gradients of Green's functions with fundamental solutions. In Section 5 we study the single layer potential operator $\mathcal{S}$ and its adjoint, and we show that they solve the Regularity and the Dirichlet problem, respectively. We also establish invertibility of $\mathcal{S}$ under a specific assumption on the coefficients, relying on the Rellich estimate from Section 3. In Section 6 we turn to the $T$-Rellich property, which we show that it is equivalent to invertibility of $\mathcal{S}:L^2(\partial\Omega)\to W^{1,2}(\partial\Omega)$, and using suitable perturbations we show that $\mathcal{S}$ is invertible when $d=0$ and either $\dive c\leq 0$ or $\dive b\leq 0$. Finally, in Section 7 we show that the $R_2$ Regularity problem for $\mathcal{L}$ and the $D_2$ Dirichlet problem for $\mathcal{L}^t$ are uniquely solvable, with suitable estimates.

\section{Preliminaries}
\subsection{Definitions}
For a domain $\Omega\subseteq\bR^n$, $C_c^{\infty}(\Omega)$ will be the space of infinitely differentiable functions that are compactly supported in $\Omega$. For $p>1$ we denote by $W^{1,p}(\Omega)$ the Sobolev space of functions $f\in L^p(\Omega)$ such that their distributional derivative also belongs to $L^p(\Omega)$. The space $W^{1,p}_{\loc}(\Omega)$ will be the space of functions $u$ such that $u\phi\in W^{1,p}(\Omega)$ for any $\phi\in C_c^{\infty}(\Omega)$. Finally, $W_0^{1,p}(\Omega)$ will be the closure of $C_c^{\infty}(\Omega)$ under the $W^{1,p}$ norm, and $W^{-1,p'}(\Omega)$ the dual space to $W_0^{1,p}(\Omega)$, where $p'$ is the conjugate exponent to $p$.

For a bounded domain $\Omega\subseteq\bR^n$, $A$ bounded, $b,c\in L^n(\Omega)$, $d\in L^{n/2}(\Omega)$ and $f,g\in L^1_{\loc}(\Omega)$, we say that $u\in W^{1,2}_{\loc}(\Omega)$ is a solution to the equation $-\dive(A\nabla u+bu)+c\nabla u+du=-\dive f+g$ in $\Omega$, if
\[
\int_{\Omega}A\nabla u\nabla\phi+b\nabla\phi\cdot u+c\nabla u\cdot\phi+du\phi=\int_{\Omega}f\nabla\phi+g\phi,
\]
for any $\phi\in C_c^{\infty}(\Omega)$.

For $\alpha\in(0,1)$, the space of bounded functions satisfying \eqref{eq:Hold} in a domain $\Omega$ will be denoted by $C^{\alpha}(\Omega)$. If $f\in C^{\alpha}(\Omega)$, we define
\[
\|f\|_{C^{0,\alpha}(\Omega)}=\sup\left\{\frac{|f(x)-f(y)|}{|x-y|^{\alpha}}\Big{|}x,y\in\Omega,x\neq y\right\},\quad \|f\|_{C^{\alpha}(\Omega)}=\|f\|_{L^{\infty}(\Omega)}+\|f\|_{C^{0,\alpha}(\Omega)}.
\]

For $b\in L^n(\Omega)$ and $d\in L^{n/2}(\Omega)$, the assumption $d\geq\dive b$ in the sense of distributions in $\Omega$ is interpreted as follows: for all $\phi\in C_c^{\infty}(\Omega)$ with $\phi\geq 0$, we have $\displaystyle\int_{\Omega}d\phi+b\nabla\phi\geq 0$.

Finally, if $\mathcal{L}u=-\dive(A\nabla u+bu)+c\nabla u+du$ is an operator, its adjoint operator will be $\mathcal{L}^tu=-\dive(A^t\nabla u+cu)+b\nabla u+du$.

\subsection{Lipschitz domains}
Let $\Omega\subseteq\bR^n$ be a bounded set, where $n\geq 3$. We say that $\Omega$ is a \textit{Lipschitz domain}, if for each $q\in\partial\Omega$ there exists a neighborhood $U$ of $q$ and a Lipschitz function $\phi_U:\bR^{n-1}\to\bR$, such that, after translation and rotation,
\[
U\cap\Omega=\{(x',t):t>\phi_U(x')\}\cap\Omega.
\]

In order to quantify the statements that will follow, we will need the following definition from \cite[pp. 5]{KenigShen}.

\begin{dfn}\label{LipDom}
We say that $\Omega\in\Pi(M,N)$ for some $M>0$ and $N>10$, if there exists $r_{\Omega}>0$ and $q_i\in\partial\Omega$ for $i=1,\dots N$, such that $\partial\Omega\subseteq\bigcup_{i=1}^NB_{r_{\Omega}}(q_i)$ and for each $i$ there exists a coordinate system so that $q_i=(0,0)$ and
\begin{equation}\label{eq:LipCov}
B_{C_Mr_{\Omega}}(q_i)\cap\Omega=B_{C_Mr_{\Omega}}(q_i)\cap\{(x',x_n)\in\bR^n\big{|}x'\in\bR^{n-1},x_n>\psi_i(x')\},
\end{equation}
where $\psi_i:\bR^{n-1}\to\bR$ is a Lipschitz function, $\psi_i(0)=0$, $\|\nabla\psi_i\|_{\infty}\leq M$, and $C_M=10(M+1)$.
	
We say that a constant $C$ depends on the Lipschitz character of $\Omega$, if $\Omega\in\Pi(M,N)$ and the constant can be made uniform for any Lipschitz domain in $\Pi(M,N)$.
\end{dfn}

We now show the next lemma for $r_{\Omega}$.

\begin{lemma}\label{rOmegaIsGood}
Let $\Omega\subseteq\bR^n$ be a Lipschitz domain. Then $r_{\Omega}$ is bounded above and below by constants that depend on $n,|\Omega|$, and the Lipschitz character of $\Omega$.
\end{lemma}
\begin{proof}
Since the balls $B_{r_{\Omega}}(q_i)$ cover $\partial\Omega$, we obtain that
\begin{equation}\label{eq:sigmaMeasure}
\sigma(\partial\Omega)\leq\sigma\left(\bigcup_{i=1}^NB_{r_{\Omega}}(q_i)\cap\partial\Omega\right) \leq \sum_{i=1}^N\sigma\left(B_{r_{\Omega}}(q_i)\cap\partial\Omega\right)\leq NCr_{\Omega}^{n-1},
\end{equation}
where $C$ only depends on $n$ and $M$. From the isoperimetric inequality, $\sigma(\partial\Omega)$ is bounded below by a constant that depends on $|\Omega|$ and $n$; therefore, $r_{\Omega}$ is bounded below by a constant depending only on $n,|\Omega|$ and the Lipschitz character of $\Omega$.

For the opposite inequality, consider the coordinate system for $q_1=(0,0)$ and let $y=(0,Mr_{\Omega})$. Define also $\delta_M=\frac{M}{2(M+1)}$. Then,  $B_{\delta_Mr_{\Omega}}(y)\subseteq B_{C_Mr_{\Omega}}(q_1)$. Moreover, for all $(x',x_n)\in B_{\delta_Mr_{\Omega}}(y)$,
\[
|x'|\leq \left|(x',x_n)-y\right|\leq\delta_Mr_{\Omega},\qquad x_n\geq Mr_{\Omega}-|x_n-Mr_{\Omega}|\geq Mr_{\Omega}-\left|(x',x_n)-y\right|\geq\left(M-\delta_M\right)r_{\Omega}.
\]
Therefore, since $\|\nabla\psi_1\|_{\infty}\leq M$,
\[
\psi_1(x')\leq|\psi_1(x')-\psi_1(0)|\leq M|x'|\leq M\delta_Mr_{\Omega}< (M-\delta_M)r_{\Omega}\leq x_n,
\]
where we used the definition of $\delta_M$ in the fourth inequality. Therefore, $B_{\delta_Mr_{\Omega}}(y)$ is a subset of the set in the right hand side of \eqref{eq:LipCov} for $i=1$, hence $B_{\delta_Mr_{\Omega}}(y)\subseteq B_{C_Mr_{\Omega}}(q_1)\cap\Omega\subseteq\Omega$. Therefore, $C_n\left(\delta_Mr_{\Omega}\right)^n=|B_{\delta_Mr_{\Omega}}|\leq|\Omega|$, which shows the reverse inequality on $r_{\Omega}$.
\end{proof}

For a Lipschitz domain $\Omega$, $\delta(x)$ will denote the distance from $x$ to $\partial\Omega$. We now define the parts of a Lipschitz domain that are close to and far from the boundary: for $\sigma<r_{\Omega}$, set
\begin{equation}\label{eq:CloseFar}
\Omega_{\sigma}=\{x\in\Omega:\delta(x)\leq\sigma\},\quad \Omega^{\sigma}=\{x\in\Omega:\delta(x)>\sigma\}.
\end{equation}
We then have the next lemma.

\begin{lemma}\label{SmallArea}
Let $\Omega$ be a Lipschitz domain with Lipschitz character $(M,N)$, and for $\sigma<r_{\Omega}$, define
\[
\Omega_{\sigma}^i=\{(x',x_n)\in\bR^n\big{|}|x'|<2r_{\Omega},\,\psi_i(x')<x_n<\psi_i(x')+(M+2)\sigma\},
\]
in the coordinate system of $B_{C_Mr_{\Omega}}(q_i)$, from Definition~\ref{LipDom}. Then $\displaystyle\Omega_{\sigma}\subseteq\bigcup_{i=1}^N\Omega_{\sigma}^i$. In addition, $|\Omega_{\sigma}|\leq C\sigma$, where $C$ depends on $n,|\Omega|$ and the Lipschitz character of $\Omega$.
\end{lemma}
\begin{proof}
Suppose that $x\in\Omega_{\sigma}$, then there exists $q\in\partial\Omega$ such that $|x-q|=\delta(x)\leq\sigma$. From Definition~\ref{LipDom}, there exists $i$ such that $|q-q_i|\leq r_{\Omega}$, therefore $|x-q_i|<2r_{\Omega}\leq C_Mr_{\Omega}$. Then, $x\in B_{C_Mr_{\Omega}}(q_i)\cap\Omega$ and $q\in B_{C_Mr_{\Omega}}(q_i)\cap\partial\Omega$. Write $x=(x',x_n)$, $q=(q',\psi_i(q'))$ where $\|\nabla\psi_i\|_{\infty}\leq M$. Then, we estimate
\begin{align*}
x_n-\psi_i(x')&=|(x',x_n)-(x',\psi_i(x'))|\leq |x-q|+|(q',\psi_i(q'))-(x',\psi_i(x'))|\\
&\leq |x-q|+|q'-x'|\sqrt{1+M^2}\leq\left(\sqrt{1+M^2}+1\right)|x-q|<(M+2)\sigma.
\end{align*}
Moreover, $|x'|\leq |x|=|x-q_i|<2r_{\Omega}$, hence $x\in\Omega_{\sigma}^i$, which shows the first claim. Note also that
\[
|\Omega_{\sigma}^i|=\int_{B_{2r_{\Omega}}^{n-1}}\int_{\psi_i(x')}^{\psi_i(x')+(M+2)\sigma}\,dx_n\,dx'\leq Cr_{\Omega}^{n-1}\sigma=C\sigma,
\]
where $C$ depends on $n,|\Omega|$ and the Lipschitz character of $\Omega$, from Lemma~\ref{rOmegaIsGood}. Therefore,
\[
|\Omega_{\sigma}|\leq\left|\bigcup_{i=1}^N\Omega_{\sigma}^i\right|\leq\sum_{i=1}^N|\Omega_{\sigma}^i|\leq CN\sigma=C\sigma,
\]
which completes the proof.
\end{proof}

For any point $q\in\partial\Omega$, we define the nontangential region
\begin{equation}\label{eq:Cone}
\Gamma(q)=\{x\in\Omega\big{|}|x-q|\leq 10(M+1)\delta(x)\}.
\end{equation}
For a function $u\in W^{1,2}_{\loc}(\Omega)$ and $f\in L^2(\partial\Omega)$, we say that $u$ converges to $f$ nontangentially, almost everywhere on $\partial\Omega$, if $u(x)\to f(q)$ for almost every $q\in\partial\Omega$ as $x\to q$ and $x\in\Gamma(q)$.

\begin{lemma}\label{LiesAbove}
Let $\Omega\subseteq\bR^n$ be a Lipschitz domain, and suppose that $x\in B_{2r_{\Omega}}(q_i)\cap\Omega$ for some $i$. In the coordinate system for $B_{r_{\Omega}}(q_i)$, if $x=(x',x_n)$ and $q_x=(x',\psi_i(x'))$, then $B_{\delta(x)/2}(x)\subseteq \Gamma(q_x)$.
\end{lemma}
\begin{proof}
Set $\delta=\delta(x)$, and note first that $B_{\delta/2}(x)\subseteq\Omega$. Since $|x-q_i|<2r_{\Omega}$, we obtain that $\delta<2r_{\Omega}$, so there exists $q\in B_{4r_{\Omega}}(q_i)\cap\partial\Omega$ with $q=(q',\psi_i(q'))$ such that $\delta=|x-q|$. Then, for $y=(y',y_n)\in B_{\delta/2}(x)$,
\begin{align*}
|y-q_x|&\leq |y-x|+|x-(x',\psi_i(x'))|\leq\frac{\delta}{2}+(x_n-\psi_i(x'))\leq\frac{\delta}{2}+|x_n-\psi_i(q')|+|\psi_i(q')-\psi_i(x')|\\
&\leq\frac{\delta}{2}+|x_n-\psi_i(q')|+M|x'-q'|\leq\frac{\delta}{2}+(M+1)|x-q|\leq (M+2)\delta,
\end{align*}
which completes the proof.
\end{proof}
For a function $u$ defined in $\Omega$, we define the nontangential maximal function and the truncated nontangential maximal function
\begin{equation}\label{eq:MaxFun}
u^*(q)=\sup_{x\in\Gamma(q)}|u(x)|,\quad u^*_{\e}(q)=\sup_{x\in\Gamma(q),\delta(x)<\e}|u(x)|.
\end{equation}

For a point $q\in\partial\Omega$, suppose that $q\in B_{r_{\Omega}}(q_i)$ for some $i\in\{1,\dots N\}$ as in Definition~\ref{LipDom}, and let $r<r_{\Omega}$. Assume that, in the coordinate system for $B_{r_{\Omega}}(q)$, $q=\psi_i(q')$. We then define, similarly to (5.4) in \cite{KenigShen},
\begin{align}\label{eq:DT}
\begin{split}
&\Delta_r(q)=\{(x',\psi_i(x'))\in\bR^n\big{|}|x'-q'|<r\},\\
&T_r(q)=\{(x',x_n)\in\bR^n\big{|}|x'-q'|<r,\,\psi_i(x')<x_n<\psi_i(x')+r\}.
\end{split}
\end{align}
Note then that, from Definition~\ref{LipDom}, $T_r(q)\subseteq\Omega$.

We now show the next analog of Lemma~\ref{rOmegaIsGood}, for the diameter of a Lipschitz domain.

\begin{lemma}\label{DiameterBound}
Let $\Omega$ be a Lipschitz domain. Then, $\diam(\Omega)$ is bounded above and below by constants that depend on $n,|\Omega|$ and the Lipschitz character of $\Omega$.
\end{lemma}
\begin{proof}
Set $R=\diam(\Omega)$. Since $\Omega\subseteq B_{2R}(x)$ for any $x\in\Omega$, we obtain $|\Omega|\leq C_nR^n$. For the reverse inequality, let $p_1,p_2\in\partial\Omega$ such that $|p_1-p_2|=R$. Suppose that $p_1=(p_1',\psi_{i_1}(p_1'))\in B_{r_{\Omega}}(q_{i_1})$ and $p_2=(p_2',\psi_{i_2}(p_2'))\in B_{r_{\Omega}}(q_{i_2})$, in the coordinate systems for $B_{r_{\Omega}}(q_{i_1})$, $B_{r_{\Omega}}(q_{i_2})$, respectively, and set $y_1=(p_1',\psi_{i_1}(p_1')+r_{\Omega}/2)\in B_{C_Mr_{\Omega}}(q_{i_1})$, $y_2=(p_2',\psi_{i_2}(p_2')+r_{\Omega}/2)\in B_{C_Mr_{\Omega}}(q_{i_2})$. From the proof of Lemma~\ref{LiesAbove}, $\delta(y_1)\geq\frac{1}{M+2}|y_1-p_1|=C_Mr_{\Omega}$ and $\delta(y_2)\geq C_Mr_{\Omega}$, for $C_M=\frac{1}{2(M+2)}$. We now set
\begin{equation*}
\e=\frac{C_M}{2}r_{\Omega},\quad k=\left\lfloor\log_2\left(\frac{2R+2r_{\Omega}}{C_Mr_{\Omega}}\right)\right\rfloor+1,
\end{equation*}
where $\lfloor\cdot\rfloor$ denotes the integer part function. Then $\delta(y_1),\delta(y_2)>\e$, and $|y_1-y_2|<2^k\e$. Hence, from (3.4) in \cite{JerisonKenigBoundary}, we can connect $y_1,y_2$ with a Harnack chain of balls $B_1,\dots B_{Ck}$ in $\Omega$, with $y_1\in B_1$ and $y_2\in B_{Ck}$, where $C$ depends on the Lipschitz character of $\Omega$.

If $s_j$ is the radius of $B_j$ and $z_j$ is the center of $B_j$, then $s_j=C_n|B_j|^{1/n}\leq C_n|\Omega|^{1/n}$. Moreover, $\log h\leq\sqrt{h}$ for all $h>0$ and $2\log 2>1$, hence $\log_2h=\frac{\log h}{\log 2}\leq 2\sqrt{h}$. Hence,
\begin{align*}
R&=|p_1-p_2|\leq r_{\Omega}+|y_1-z_1|+|y_2-z_{Ck}|+\sum_{j=2}^{Ck}|z_j-z_{j-1}|\leq r_{\Omega}+2\sum_{j=1}^{Ck}s_j\leq r_{\Omega}+2Ck\cdot C_n|\Omega|^{1/n}\\
&\leq C+C\log_2(CR+C)\leq C+2C\sqrt{CR+C}\leq 3C\sqrt{C}+2C\sqrt{CR},
\end{align*}
where $C>16$ depends on $n,|\Omega|$ and the Lipschitz character of $\Omega$, from the definition of $k$ and Lemma~\ref{rOmegaIsGood}. If we assume that $R\geq C^4$, then $\sqrt{R}=\frac{R}{\sqrt{R}}\leq\frac{R}{C^2}$, hence
\begin{equation*}
R\leq 3C\sqrt{C}+2C\sqrt{CR}\leq 3C\sqrt{C}+2\frac{R}{\sqrt{C}}\leq 3C\sqrt{C}+\frac{R}{2}\quad\Rightarrow\quad R\leq 6C\sqrt{C},
\end{equation*}
which contradicts the assumption $R\geq C^4$. Therefore $R<C^4$, which completes the proof.
\end{proof}
We also show the next bound.

\begin{lemma}\label{BoundxInsideDelta}
Let $\Omega\subseteq\mathbb R^n$ be a Lipschitz domain and $\delta\in(0,1)$. Then, for any $x\in\Omega$,
\[
\int_{\partial\Omega}|x-q|^{1-n+\delta}\,d\sigma(q)\leq C,
\]
where $C$ depends on $n,\delta,|\Omega|$ and the Lipschitz character of $\Omega$. 
\end{lemma}
\begin{proof}
If $\delta(x)\geq r_{\Omega}$, then $|x-q|\geq r_{\Omega}$ for any $q\in\partial\Omega$, so $\displaystyle\int_{\partial\Omega}|x-q|^{1-n+\delta}\,d\sigma(q)\leq r_{\Omega}^{1-n+\delta}\sigma(\partial\Omega)\leq C$,
where $C$ depends on $n,\delta,|\Omega|$ and the Lipschitz character of $\Omega$, from \eqref{eq:sigmaMeasure} and Lemma~\ref{rOmegaIsGood}.

If now $\delta(x)<r_{\Omega}$, then there exists $q_x\in\partial\Omega$ such that $|x-q_x|=\delta(x)<r_{\Omega}$. From Definition~\ref{LipDom}, there exists $i$ such that $q_x\in B_{r_{\Omega}}(q_i)$, and then $x\in B_{C_Mr_{\Omega}}(q_i)$. In the coordinate system for $B_{r_{\Omega}}(q_i)$ we denote $x=(x',x_n)$, $q_x=(x',\psi_i(x'))$, and we also denote any $q\in\Delta_{2r_{\Omega}}(q_x)$ by $q=(q',\psi_i(q'))$ for $q'\in B_{2r_{\Omega}}^{n-1}(x')=B_0$, an $n-1$ dimensional ball centered at $x'$, with radius $2r_{\Omega}$. Then, we have that $|x-q|\geq|x'-q'|$, hence
\begin{equation}\label{eq:InsideSmall}
\int_{\Delta_{2r_{\Omega}}(q_x)}|x-q|^{1-n+\delta}\,d\sigma(q)\leq \int_{\Delta_{2r_{\Omega}}(q_x)}|x'-q'|^{1-n+\delta}\,d\sigma(q)\leq C\int_{B_0}|x'-q'|^{1-n+\delta}\,dq'\leq C,
\end{equation}
where $C$ depends on $n,\delta,M$ and $r_{\Omega}$. Moreover, if $q\in\partial\Omega\setminus \Delta_{2r_{\Omega}}(q_x)$, then we have that $|q-q_x|\geq 2r_{\Omega}$, hence $|x-q|\geq |q-q_x|-|q_x-x|\geq 2r_{\Omega}-\delta(x)>r_{\Omega}$,
therefore
\begin{equation}\label{eq:OutsideSmall}
\int_{\partial\Omega\setminus\Delta_{2r_{\Omega}}(q_x)}|x-q|^{1-n+\delta}\,d\sigma(q)\leq r_{\Omega}^{1-n+\delta}\sigma(\partial\Omega\setminus\Delta_{2r_{\Omega}}(q_x))\leq r_{\Omega}^{1-n+\delta}\sigma(\partial\Omega).
\end{equation}
Finally, we add \eqref{eq:InsideSmall} and \eqref{eq:OutsideSmall} and use \eqref{eq:sigmaMeasure} and Lemma~\ref{rOmegaIsGood} to complete the proof.
\end{proof}

\subsection{Sobolev spaces on the boundary}
We now turn to the definition of $W^{1,2}(\partial\Omega)$, which will be the space of boundary values for the Regularity problem $R_2$. The following is similar to Definition 1.7 in \cite{Verchota}.

\begin{dfn}
Let $\Omega$ be a Lipschitz domain. We say that $f\in W^{1,2}(\partial\Omega)$ if $f\in L^2(\partial\Omega)$ and if in each ball $B_i=B_{r_{\Omega}}(q_i)$ in Definition~\ref{LipDom}, there exist functions $g_1,\dots g_{n-1}\in L^2(B_i\cap\partial\Omega)$, so that, for every $h\in C_c^{\infty}(\bR^{n-1})$, $j=1,\dots n-1$ and $i=1,\dots N$,
\[
\int_{\mathbb R^{n-1}}h(x)g_j(x,\psi_i(x))\,dx=\int_{\mathbb R^{n-1}}\partial_jh(x)f(x,\psi_i(x))\,dx.
\]
\end{dfn}
In local coordinates, if $\nu(q)$ is the unit outer normal of $\partial\Omega$ at $q$, we then define (as in Definition 1.9 in \cite{Verchota}),
\[
-\nabla_Tf(q)=(g_1(q),\dots g_{n-1}(q),0)-\left<(g_1(q),\dots g_{n-1}(q),0),\nu(q)\right>\cdot\nu(q).
\]
Then $\nabla_Tf(q)$ is normal to $\nu(q)$ almost everywhere on $\partial\Omega$, and it is independent of the choice of coordinates. Moreover, if $f$ is $C^1$ in a neighborhood of $q$ in $\bR^n$ and $\nu(q)$ exists, we can show that $\nabla_Tf(q)=\nabla f(q)-\left<\nabla f(q),\nu(q)\right>\nu(q)$. We also define the norm
\[
\|f\|_{W^{1,2}(\partial\Omega)}=\|f\|_{L^2(\partial\Omega)}+\|\nabla_Tf\|_{L^2(\partial\Omega)}.
\]
Under this norm, $W^{1,2}(\partial\Omega)$ is a Hilbert space, with inner product
\[
\left<f,g\right>_{W^{1,2}(\partial\Omega)}=\int_{\partial\Omega}\left(f\cdot g+\nabla_Tf\cdot\nabla_Tg\right)\,d\sigma.
\]
We also consider the space $W^{-1,2}(\partial\Omega)$, which is the dual of $W^{1,2}(\partial\Omega)$. Then, the Riesz representation theorem shows that there exists an invertible operator
\begin{equation}\label{eq:Emb1}
R:W^{1,2}(\partial\Omega)\to W^{-1,2}(\partial\Omega),\quad \quad (Rf)g=\int_{\partial\Omega}(fg+\nabla_Tf\cdot\nabla_Tg)\,d\sigma ,
\end{equation}
for all $f,g\in W^{1,2}(\partial\Omega)$, and also $\|R\|_{W^{1,2}(\partial\Omega)\to W^{-1,2}(\partial\Omega)}=\|R^{-1}\|_{W^{-1,2}(\partial\Omega)\to W^{1,2}(\partial\Omega)}=1.$
In addition, we consider the operator
\begin{equation}\label{eq:Emb2}
E:L^2(\partial\Omega)\to W^{-1,2}(\partial\Omega):\quad \forall g\in W^{1,2}(\partial\Omega),\quad (Ef)g=\int_{\partial\Omega}fg\,d\sigma.
\end{equation}
Then, the image $E(L^2(\partial\Omega))$ is dense in $W^{-1,2}(\partial\Omega)$: if $G=Rg\in W^{-1,2}(\partial\Omega)$ for some $g\in W^{1,2}(\partial\Omega)$ with $\left<G,Ef\right>_{W^{-1,2},W^{-1,2}}=0$ for any $f\in L^2(\partial\Omega)$, then $\int_{\partial\Omega}fg\,d\sigma=0$ for any $f\in L^2(\partial\Omega)$, so $g\equiv 0$. This implies that $(E(L^2(\partial\Omega)))^{\perp}=\{0\}$ in $W^{-1,2}(\partial\Omega)$, therefore $\overline{E(L^2(\partial\Omega))}=W^{-1,2}(\partial\Omega)$.

\section{Estimates}
\subsection{A priori estimates}
We now turn to various a priori estimates for solutions to the equation $\mathcal{L}u=0$. We remark that similar estimates to the ones we will show appear in Section 2 in \cite{XuUniformSystems} under slightly stronger assumptions than ours.

We first show the Cacciopoli estimate.

\begin{lemma}\label{Cacciopoli}
Let $\Omega\subseteq\bR^n$ be a domain with $|\Omega|<1$. Let $A$ be bounded and elliptic in $\Omega$ with ellipticity $\lambda$, and let $b,c\in L^p(\Omega)$, $d\in L^{p/2}(\Omega)$ for some $p>n$. Assume also that $u\in W^{1,2}_{\loc}(\Omega)$ solves the equation $-\dive(A\nabla u+bu)+c\nabla u+du=-\dive f+g$ in $\Omega$, for some $f,g\in L^2(\Omega)$. Then, for any $\psi\in C_c^{\infty}(\Omega)$,
\[
\int_{\Omega}|\psi\nabla u|^2\leq C\int_{\Omega}\left(|\psi|^2+|\nabla\psi|^2\right)|u|^2+C\int_{\Omega}|f\psi|^2+C\int_{\Omega}|g\psi|^2,
\]
where $C$ depends on $n$,$\lambda,\|A\|_{\infty}$, $\|b\|_p,\|c\|_p$ and $\|d\|_{p/2}$.
\end{lemma}
\begin{proof}
Using $u\psi^2\in W_0^{1,2}(\Omega)$ as a test function we obtain that
\[
\int_{B_{2r}}A\nabla u\nabla(u\psi^2)+b\nabla(u\psi^2)\cdot u+c\nabla u\cdot u\psi^2+du^2\psi^2=\int_{\Omega}f\psi(\psi\nabla u+2u\nabla\psi)+gu\psi^2=I,
\]
where we estimate, for $\delta>0$,
\begin{equation}\label{eq:IIBound}
I\leq \frac{1}{4\delta}\|f\psi\|_2^2+\delta\|\psi\nabla u\|_2^2+\|f\psi\|_2^2+\|u\nabla\psi\|_2^2+\|g\psi\|_2^2+\|u\psi\|_2^2.
\end{equation}
Using the ellipticity of $A$ we then obtain
\begin{align}\nonumber
\lambda\int_{\Omega}|\psi\nabla u|^2&\leq -\int_{\Omega}\left(2A\nabla u\nabla\psi\cdot u\psi+(b+c)\nabla u\cdot u\psi^2+2b\nabla\psi\cdot u^2\psi+du^2\psi^2\right)+I\\
\nonumber
&\leq C\|u\nabla\psi\|_2\|\psi\nabla u\|_2+\int_{\Omega}\left(|b+c||\psi\nabla u||u\psi|+2|b||u\nabla\psi||u\psi|+|d||u^2\psi^2|\right)+I\\
\label{eq:Plug}
&\leq C\delta\|\psi\nabla u\|_2^2+\frac{C}{4\delta}\|u\nabla\psi\|_2^2+\int_{\Omega}|b+c||\psi\nabla u||u\psi|+2|b||u\nabla\psi||u\psi|+|d||u^2\psi^2|+I,
\end{align}
for any $\delta$, where $C$ depends on $\|A\|_{\infty}$.

As in (2.1) in \cite{KimSak}, for a function $f\geq 0$ and $t>0$, we denote $f^t(x)=f(x)$ if $f(x)\geq t$, and $f(x)=0$ if $f(x)<t$, and we also set $f_t=f-f^t$. Since $u\psi\in W_0^{1,2}(\Omega)$, using Sobolev's inequality and also (2.4) in \cite{KimSak}, we obtain that, for $t>1$,
\begin{align}\nonumber
\int_{\Omega}|b+c||\psi\nabla u||u\psi|&\leq t\int_{\Omega}|\psi\nabla u||u\psi|+\left\||b+c|^t\right\|_n\|\psi\nabla u\|_2\|u\psi\|_{2^*}\\
\nonumber
&\leq t\|\psi\nabla u\|_2\|u\psi\|_2+C_n\|b+c\|_p^{p/n}t^{1-p/n}\|\psi\nabla u\|_2\left(\|\psi\nabla u\|_2+\|u\nabla\psi\|_2\right)\\
\nonumber
&\leq t\delta\|\psi\nabla u\|_2^2+\frac{t}{4\delta}\|u\psi\|_2^2+Ct^{1-p/n}\|\psi\nabla u\|_2^2+C\|\psi\nabla u\|_2^2\|u\nabla\psi\|_2^2\\
\label{eq:Plug1}
&\leq\left(t\delta+Ct^{1-p/n}+C\delta\right)\|\psi\nabla u\|_2^2+\frac{t}{4\delta}\|u\psi\|_2^2+\frac{C}{4\delta}\|u\nabla\psi\|_2^2,
\end{align}
where $C$ depends on $n$, $\|b\|_p$ and $\|c\|_p$. Moreover, since $|\Omega|\leq 1$, $\|b\|_n\leq\|b\|_p$, therefore
\begin{equation}\label{eq:Plug2}
\int_{\Omega}|b||u\nabla\psi||u\psi|\leq\|b\|_n\|u\nabla\psi\|_2\|u\psi\|_{2^*}\leq  C\delta\|\psi\nabla u\|_2^2+\left(C+\frac{C}{4\delta}\right)\|u\nabla\psi\|_2^2,
\end{equation}
where $C$ depends on $n$ and $\|b\|_p$, and also, using (2.4) in \cite{KimSak},
\begin{align}\nonumber
\int_{\Omega}|d||u^2\psi^2|&\leq t\|u\psi\|_2^2+\left\||d|^t\right\|_{n/2}\|u\psi\|_{2^*}^2\leq t\|u\psi\|_2^2+C_n\|d\|_{p/2}^{p/n}t^{1-p/n}\left(\|\psi\nabla u\|_2+\|u\nabla\psi\|_2\right)^2\\
\label{eq:Plug3}
&\leq t\|u\psi\|_2^2+Ct^{1-p/n}\|\psi\nabla u\|_2^2+C\|u\nabla\psi\|_2^2.
\end{align}
where $C$ depends on $n$ and $\|d\|_{p/2}$. Plugging \eqref{eq:Plug1}, \eqref{eq:Plug2} and \eqref{eq:Plug3} in \eqref{eq:Plug}, and also using \eqref{eq:IIBound},
\begin{multline*}
\int_{\Omega}|\psi\nabla u|^2\leq \left(C+\frac{C}{4\delta}\right)\|u\nabla\psi\|_2^2+\left(1+t+\frac{t}{4\delta}\right)\|u\psi\|_2^2+\left(t\delta+Ct^{1-p/n}+C\delta\right)\|\psi\nabla u\|_2^2\\
+\left(\frac{1}{4\delta}+1\right)\|f\psi\|_2^2+\|g\psi\|_2^2,
\end{multline*}
where $C$ depends on $n$, $\lambda,\|A\|_{\infty}$,$\|b\|_p,\|c\|_p$ and $\|d\|_{p/2}$. We now choose $t>1$, depending on $\lambda$, $\|A\|_{\infty}$,$\|b\|_p,\|c\|_p,\|d\|_{p/2}$, such that $Ct^{1-p/n}<\frac{1}{4}$, and we choose $\delta>0$ such that $t\delta+C\delta<\frac{1}{4}$. Then,
\[
\int_{\Omega}|\psi\nabla u|^2\leq C\|u\nabla\psi\|_2^2+C\|u\psi\|_2^2+C\|f\psi\|_2+C\|g\psi\|_2^2,
\]
which completes the proof.
\end{proof}

We now turn to regularity of the derivatives of solutions to $\mathcal{L}u=0$. The next lemma will be the basis for a bootstrap argument.

\begin{lemma}\label{InductiveStep}
Let $\Omega\subseteq\mathbb R^n$ be a domain, and suppose that $B_2$ is a compactly supported ball in $\Omega$ with radius $2$. Let also $A\in\M_{\Omega}(\lambda,\alpha,\tau)$ and $b,c,d,f,g\in L^p(\Omega)$ for some $p>n$. Suppose also that $1<\beta<\gamma<2$, and $u\in W^{1,t}(B_{\gamma})$ is a solution to the equation
\[
-\dive(A\nabla u+bu)+c\nabla u+du=-\dive f+g
\]
in $B_2$, for some $t\in(1,n)$. Then,
\[
\|u\|_{W^{1,s^*}(B_{\beta})}\leq C\|u\|_{W^{1,t}(B_{\gamma})}+C\|f\|_{L^p(B_{\gamma})}+C\|g\|_{L^p(B_{\gamma})},
\]
where $\frac{1}{s^*}=\frac{1}{t}+\frac{1}{p}-\frac{1}{n}$, and $C$ depends on $n,p,\lambda$, $\alpha$, $\tau$, $\gamma-\beta$, $\|A\|_{L^{\infty}(B_2)}$, $\|b\|_{L^p(B_2)}$, $\|c\|_{L^p(B_2)}$ and $\|d\|_{L^p(B_2)}$.
\end{lemma}
\begin{proof}
At first assume, more generally, that $t>1$. Let $\phi$ be a smooth cutoff supported in $B_{\gamma}$, with $\phi=0$ in $B_{\gamma}\setminus B_{\frac{\gamma+\beta}{2}}$, and $\phi\equiv 1$ in $B_{\beta}$, with $|\nabla\phi|,|\nabla^2\phi|\leq C$ for $C$ depending on $\gamma-\beta$. We then compute
\begin{equation}\label{eq:uphi}
-\dive(A\nabla(u\phi))=\dive(bu\phi-A\nabla\phi\cdot u-f\phi)+(g-c\nabla u-du)\phi+(f-A\nabla u-bu)\nabla\phi.
\end{equation}
We now follow the lines of the proof of Proposition 2.13 in \cite{DongEscKim}: let $v_{\gamma}$ be the Newtonian potential of the function
\begin{equation}\label{eq:DefOfh}
h_{\gamma}=\left((g-c\nabla u-du)\phi+(f-A\nabla u-bu)\nabla\phi\right)\chi_{B_{\gamma}}.
\end{equation}
Set $s=\frac{pt}{p+t}$. Then, from the pointwise estimates on $\phi$ and $\nabla\phi$,
\begin{align}\label{eq:Lsforf}
\|h_{\gamma}\|_{L^s(\bR^n)}&\leq C\|c\nabla u\|_{L^s(B_{\gamma})}+C\|A\|_{\infty}\|\nabla u\|_{L^s(B_{\gamma})}+\|(d\phi+b\nabla\phi)u \|_{L^s(B_{\gamma})}+\|g\phi+f\nabla\phi\|_{L^s(B_{\gamma})}.
\end{align}
From H{\"o}lder's inequality, we estimate
\begin{equation}\label{eq:Lsforcu}
\|c\nabla u\|_{L^s(B_{\gamma})}\leq\|c\|_p\|\nabla u\|_{L^t(B_{\gamma})},\quad  \|(d\phi+b\nabla\phi)u\|_{L^s(B_{\gamma})}\leq C\left(\|b\|_p+\|d\|_p\right)\|u\|_{L^t(B_{\gamma})}.
\end{equation}
Moreover, since $s<t$ and $s<p$, we estimate
\begin{equation}\label{eq:Lsforu}
\|\nabla u\|_{L^s(B_{\gamma})}\leq C\|\nabla u\|_{L^t(B_{\gamma})},\quad \|g\phi+f\nabla\phi\|_{L^s(B_{\gamma})}\leq C\|f\|_{L^p(B_{\gamma})}+C\|g\|_{L^p(B_{\gamma})}.
\end{equation}
Plugging \eqref{eq:Lsforcu} and \eqref{eq:Lsforu} in \eqref{eq:Lsforf}, we obtain that
\begin{equation}\label{eq:Lsforf2}
\|h_{\gamma}\|_{L^s(\bR^n)}\leq C\|u\|_{W^{1,t}(B_{\gamma})}+C\|f\|_{L^p(B_{\gamma})}+C\|g\|_{L^p(B_{\gamma})},
\end{equation}
where $C$ depends on $n,p$ $\gamma-\beta$, $\|A\|_{\infty}$, $\|b\|_p$, $\|c\|_p$ and $\|d\|_p$. Then, from Theorem 9.9 in \cite{Gilbarg},
\begin{equation}\label{eq:NewtonianBound}
\|\nabla ^2v_{\gamma}\|_{L^s(\mathbb R^n)}\leq C\|h_{\gamma}\|_{L^s(\mathbb R^n)}\leq C\|u\|_{W^{1,t}(B_{\gamma})}+C\|f\|_{L^p(B_{\gamma})}+C\|g\|_{L^p(B_{\gamma})}.
\end{equation}
Assume now that $t\in(1,n)$, then $s<n$. Therefore, from Sobolev's inequality, and using \eqref{eq:NewtonianBound}, for $\frac{1}{s^*}=\frac{1}{s}-\frac{1}{n}=\frac{1}{p}+\frac{1}{t}-\frac{1}{n}$,
\begin{equation}\label{eq:vs^*}
\|\nabla v_{\gamma}\|_{L^{s^*}(\mathbb R^n)}\leq C\|\nabla ^2v_{\gamma}\|_{L^s(\mathbb R^n)}\leq C\|u\|_{W^{1,t}(B_{\gamma})}+C\|f\|_{L^p(B_{\gamma})}+C\|g\|_{L^p(B_{\gamma})}.
\end{equation}
Now, from \eqref{eq:uphi} we obtain that
\begin{equation}\label{eq:DefOfg'}
-\dive(A\nabla(u\phi))=\dive(bu\phi-A\nabla\phi\cdot u-f\phi+\nabla v_{\gamma})=\dive g'_{\gamma}.
\end{equation}
Since $u\phi$ and $u|\nabla\phi|$ vanish in a neighborhood of $\partial B_{\gamma}$ and also $s<t$, from the Sobolev inequality,
\begin{equation}\label{eq:L^s*foru}
\|u\phi\|_{L^{t^*}(B_{\gamma})}\leq C\|u\|_{W^{1,t}(B_{\gamma})},\quad\left\|u|\nabla\phi|\right\|_{L^{s^*}(B_{\gamma})}\leq \left\|\nabla\left(u|\nabla\phi|\right)\right\|_{L^s(B_{\gamma})}\leq C\|u\|_{W^{1,t}(B_{\gamma})},
\end{equation}
where $\frac{1}{t^*}=\frac{1}{t}-\frac{1}{n}$. Since $s^*<p$ and $\frac{s^*}{p}+\frac{s^*}{t^*}=1$, using \eqref{eq:vs^*} and the estimates in \eqref{eq:L^s*foru},
\begin{align*}
\|g'_{\gamma}\|_{L^{s^*}(B_{\gamma})}&\leq\|bu\phi\|_{L^{s^*}(B_{\gamma})}+\|A\|_{\infty}\left\|u|\nabla\phi|\right\|_{L^{s^*}(B_{\gamma})}+\|f\phi\|_{L^{s^*}(B_{\gamma})}+\|\nabla v_{\gamma}\|_{L^{s^*}(B_{\gamma})}\\
&\leq \|b\|_p\|u\phi\|_{L^{t^*}(B_{\gamma})}+ C\|u\|_{W^{1,t}(B_{\gamma})}+C\|f\|_{L^p(B_{\gamma})}+C\|g\|_{L^p(B_{\gamma})}\\
&\leq C\|u\|_{W^{1,t}(B_{\gamma})}+C\|f\|_{L^p(B_{\gamma})}+C\|g\|_{L^p(B_{\gamma})}
\end{align*}
Therefore, using Theorem 1 in \cite{AuscherQafsaoui}, we obtain that
\[
\|\nabla u\|_{L^{s^*}(B_{\beta})}\leq \|\nabla(u\phi)\|_{L^{s^*}(B_{\gamma})}\leq  C\|g'_{\gamma}\|_{L^{s^*}(B_{\gamma})}\leq C\|u\|_{W^{1,t}(B_{\gamma})}+C\|f\|_{L^p(B_{\gamma})}+C\|g\|_{L^p(B_{\gamma})},
\]
where $C$ depends on $n,p,\lambda$, $\alpha$, $\tau$, $\gamma-\beta$, $\|A\|_{\infty}$, $\|b\|_p$, $\|c\|_p$ and $\|d\|_p$.
Finally, we finish the proof by noting that $\|u\|_{L^{s^*}(B_{\beta})}\leq \|u\phi\|_{L^{s^*}(B_{\gamma})}$, and using that $s^*<t^*$ and the first estimate in \eqref{eq:L^s*foru}.
\end{proof}

We now obtain higher integrability for the derivatives of $u$.

\begin{prop}\label{LpReg}
Let $\Omega\subseteq\mathbb R^n$ be a domain with $|\Omega|<1$, and suppose that $A\in \M_{\Omega}(\lambda,\alpha,\tau)$, and $b,c,d,f,g\in L^p(\Omega)$ for some $p>n$. Assume also that the ball $B_{2r}$ is compactly supported in $\Omega$. Then, for any solution $u\in W^{1,2}_{\loc}(\Omega)$ of the equation
\[
\mathcal{L}u=-\dive(A\nabla u+bu)+c\nabla u+du=-\dive f+g
\]
in $\Omega$, we have that
\begin{align*}
\int_{B_r}|\nabla u|^p&\leq Cr^{n-p-pn/2}\|u\|_{L^2(B_{2r})}^p+C\|f\|_{L^p(B_{2r})}^p+Cr^p\|g\|_{L^p(B_{2r})}\\
&\leq\frac{C}{r^p}\int_{B_{2r}}|u|^p+C\int_{B_{2r}}|f|^p+Cr^p\int_{B_{2r}}|g|^p,
\end{align*}
where $C$ depends on $n,p$, $\lambda,\alpha,\tau$, $\|A\|_{L^{\infty}(B_{2r})}$, $\|b\|_{L^p(B_{2r})}$, $\|c\|_{L^p(B_{2r})}$ and $\|d\|_{L^p(B_{2r})}$.
\end{prop}
\begin{proof}
By scaling, it is enough to assume that $r=1$. Let $\delta=\frac{1}{n}-\frac{1}{p}\in\left(0,\frac{1}{3}\right)$, then there exists $N\in\mathbb N$ with $N\geq 1$ such that $N\delta<1-\frac{1}{n}+\frac{\delta}{2}\leq(N+1)\delta$. Set
\[
\frac{1}{t_i}=\frac{1}{n}+\left(N-i-\frac{1}{2}\right)\delta, \quad i=0,\dots N.
\]
Note then that $t_i>1$ for $i=0,\dots N$, and
\[
\frac{1}{t_{N-1}}>\frac{1}{n},\qquad\frac{1}{t_N}<\frac{1}{n},\qquad \frac{1}{t_0}=\frac{1}{n}-\frac{3\delta}{2}+(N+1)\delta\geq 1-\delta>\frac{1}{2},
\]
since $\delta<\frac{1}{3}$. Therefore, applying Lemma~\ref{InductiveStep} for $i=0,\dots N-1$ and suitable $\beta_i,\gamma_i$, we obtain
\begin{align*}
\|u\|_{W^{1,t_N}(B_{7/4})}&\leq C\|u\|_{W^{1,t_0}(B_{15/8})}+C\|f\|_{L^p(B_{15/8})}+C\|g\|_{L^p(B_{15/8})}\\
&\leq C\|u\|_{W^{1,2}(B_{15/8})}+C\|f\|_{L^{p}(B_2)}+C\|g\|_{L^p(B_2)},
\end{align*}
since $t_0<2$. Since $\|u\|_{W^{1,2}(B_{15/8})}\leq C\|u\|_{L^2(B_2)}+C\|f\|_{L^2(B_2)}+C\|g\|_{L^2(B_2)}$ from Lemma~\ref{Cacciopoli}, combining with H{\"o}lder's inequality we obtain that
\begin{equation}\label{eq:tN}
\|u\|_{W^{1,t_N}(B_{7/4})}\leq C\|u\|_{L^2(B_2)}+C\|f\|_{L^p(B_2)}+C\|g\|_{L^p(B_2)}.
\end{equation}
Note now that $t_N>n$, therefore Morrey's inequality shows that
\begin{equation}\label{eq:uInf}
\|u\|_{L^{\infty}(B_{7/4})}\leq \|u\|_{C^{1-n/t_N}(B_{7/4})}\leq C\|u\|_{W^{1,t_N}(B_{7/4})}\leq C\|u\|_{L^2(B_2)}+C\|f\|_{L^p(B_2)}+C\|g\|_{L^p(B_2)},
\end{equation}
where $C$ depends on $n,p,\lambda,\alpha,\tau$, $\|A\|_{\infty},\|b\|_p,\|c\|_p$ and $\|d\|_p$. Let now $\phi$ be a smooth cutoff that is supported in $B_{13/8}$ and it is equal to $1$ in $B_{3/2}$, set $h_{7/4}$ to be as in \eqref{eq:DefOfh} for $\gamma=\frac{7}{4}$ and let $v_{7/4}$ be the Newtonian potential of $h_{7/4}$. Then \eqref{eq:Lsforf2} for $s=\frac{pt_N}{p+t_N}$ shows that
\[
\|h_{7/4}\|_{L^s(\bR^n)}\leq C\|u\|_{W^{1,t_N}(B_{7/4})}+C\|f\|_{L^p(B_{7/4})}+C\|g\|_{L^p(B_{7/4})}.
\]
Since $t_N>n$, we obtain that $s>n$. Therefore Morrey's inequality, Theorem 9.9 in \cite{Gilbarg} and \eqref{eq:tN} show that
\[
\|\nabla v_{7/4}\|_{C^{0,1-n/s}(B_{7/4})}\leq C\|\nabla^2v_{7/4}\|_{L^s(\bR^n)}\leq C\|h_{7/4}\|_{L^s(\bR^n)}\leq  C\|u\|_{L^2(B)}+C\|f\|_{L^p(B_2)}+C\|g\|_{L^p(B_2)},
\]
hence, from the definition of the Newtonian potential,
\begin{equation}\label{eq:NablavInf}
\|\nabla v_{7/4}\|_{L^{\infty}(B_{7/4})}\leq \|\nabla v_{7/4}\|_{C^{1-n/s}(B_{7/4})}\leq C\|u\|_{L^2(B_2)}+C\|f\|_{L^p(B_2)}+C\|g\|_{L^p(B_2)}.
\end{equation}
Hence, if $g'_{7/4}$ is defined as in \eqref{eq:DefOfg'}, then $-\dive(A\nabla(u\phi))=\dive g'_{7/4}$ in $B_{7/4}$, and
\begin{align*}
\|g'_{7/4}\|_{L^p(B_{7/4})}&\leq \|bu\phi\|_{L^p(B_{7/4})}+\|A\nabla\phi\cdot u\|_{L^{p}(B_{7/4})}+\|f\phi\|_{L^p(B_{7/4})}+\|\nabla v_{7/4}\|_{L^p(B_{7/4})}\\
&\leq\|b\|_{L^{p}(B_{7/4})}\|u\|_{L^{\infty}(B_{7/4})}+C\|u\|_{L^{p}(B_{7/4})}+C\|f\|_{L^p(B_2)}+\|\nabla v_{7/4}\|_{L^p(B_{7/4})}\\
&\leq C\|u\|_{L^2(B_2)}+C\|f\|_{L^p(B_2)}+\|g\|_{L^p(B_2)},
\end{align*}
from \eqref{eq:uInf} and \eqref{eq:NablavInf}. Hence, Theorem 1 in \cite{AuscherQafsaoui} shows that
\begin{equation*}
\|\nabla u\|_{L^{p}(B_{3/2})}\leq\|\nabla(u\phi)\|_{L^{p}(B_{7/4})}\leq C\|g'_{7/4}\|_{L^p(B_{7/4})}\leq C\|u\|_{L^2(B_2)}+C\|f\|_{L^p(B_2)}+C\|g\|_{L^p(B_2)},
\end{equation*}
which shows the first part of the estimate. The second part follows from H{\"o}lder's inequality.
\end{proof}

If we assume that $b$ and $f$ are H{\"o}lder continuous, then we obtain that $\nabla u$ is H{\"o}lder continuous as well, as the next Proposition shows.

\begin{prop}\label{CAlphaReg}
Let $\Omega\subseteq\mathbb R^n$ be a domain with $|\Omega|<1$, and suppose that $A\in \M_{\Omega}(\lambda,\alpha,\tau)$, $b\in\C_{\Omega}(\alpha,\tau)$, and $c,d\in L^p(\Omega)$, for some $p>n$. Assume also that the ball $B_{2r}$ is compactly supported in $\Omega$, and $f\in C^{\beta}(B_{2r})$, $g\in L^p(B_{2r})$. Then, for any solution $u\in W^{1,2}_{\loc}(\Omega)$ of the equation
\[
\mathcal{L}u=-\dive(A\nabla u+bu)+c\nabla u+du=-\dive f+g
\]
in $\Omega$, we have that
\begin{multline}\label{eq:q:CAlphaReg2}
\frac{1}{r}\|u\|_{L^{\infty}(B_r)}+\|\nabla u\|_{L^{\infty}(B_r)}+r^{\gamma}\|\nabla u\|_{C^{0,\gamma}(B_r)}\leq\frac{C}{r}\left(\fint_{B_{2r}}u^2\right)^{1/2}\\
+C\|f\|_{L^{\infty}(B_{2r})}+Cr^{\beta}\|f\|_{C^{0,\beta}(B_{2r})}+Cr\left(\fint_{B_{2r}}|g|^p\right)^{1/p},
\end{multline}
where $\gamma=\min\left\{\alpha,\beta,1-\frac{n}{p}\right\}$, and $C$ depends on $n,p,\lambda, \alpha,\beta,\tau$, $\|A\|_{L^{\infty}(B_{2r})}$, $\|b\|_{L^{\infty}(B_{2r})}, \|c\|_{L^p(B_{2r})}$ and $\|d\|_{L^p(B_{2r})}$.
\end{prop}
\begin{proof}
By scaling, it suffices to show the inequalities for $r=1$. If $h_{7/4}$ and $v_{7/4}$ are as in the proof of Proposition~\ref{LpReg}, then the second estimates in \eqref{eq:uInf} and \eqref{eq:NablavInf} show that
\begin{equation}\label{eq:q_1q_2}
\|u\|_{C^{q_1}(B_{7/4})}+\|\nabla v_{7/4}\|_{C^{q_2}(B_{7/4})}\leq C\|u\|_{L^2(B_2)}+C\|f\|_{L^p(B_2)}+C\|g\|_{L^p(B_2)},
\end{equation}
for some $q_1,q_2$ that depend on $n$ and $p$. Note now that, from the definition of $v_{7/4}$ and \eqref{eq:DefOfg'}, $u$ solves the equation
\[
-\dive(A\nabla u)=\dive(bu-f+\nabla v_{7/4})=\dive g'
\]
in $B_{3/2}$. Then, if $\beta_0=\min\left\{q_1,q_2,\alpha,\beta\right\}$, and using also \eqref{eq:q_1q_2},
\begin{align*}
\|g'\|_{C^{\beta_0}(B_{3/2})}&\leq C\|b\|_{C^{\beta_0}}\|u\|_{C^{\beta_0}(B_{3/2})}+\|f\|_{C^{\beta}(B_{3/2})}+\|\nabla v_{7/4}\|_{C^{\beta_0}(B_{3/2})}\\
&\leq C\|u\|_{L^2(B_2)}+C\|f\|_{C^{\beta}(B_2)}+C\|g\|_{L^p(B_2)}.
\end{align*}
Using (2.1) and (2.2) in \cite{KenigShen}, we then obtain that
\begin{equation}\label{eq:5/4Estimate}
\|\nabla u\|_{C^{\beta_0}(B_{5/4})}\leq C\|u\|_{L^2(B_{3/2})}+C\|g'\|_{C^{\beta_0}(B_{3/2})}\leq C\|u\|_{L^2(B_2)}+C\|f\|_{C^{\beta}(B_2)}+C\|g\|_{L^p(B_2)},
\end{equation}
where $C$ depends on $n,p,\lambda,\alpha,\beta,\tau$, $\|A\|_{\infty}$, $\|b\|_{\infty}$, $\|c\|_p$ and $\|d\|_p$. Hence $u$ is Lipschitz in $B_{5/4}$. Note now that, from \eqref{eq:DefOfh} for $\gamma=\frac{5}{4}$ and $\phi$ that is equal to $1$ in $B_{9/8}$, and using also \eqref{eq:5/4Estimate},
\begin{align*}
\|h_{5/4}\|_{L^p(\bR^n)}&\leq C\|\nabla u\|_{L^{\infty}(B_{5/4})}+C\|u\|_{L^{\infty}(B_{5/4})}+C\|f\|_{L^{\infty}(B_{5/4})}+C\|g\|_{L^p(B_{5/4})}\\
&\leq C\|u\|_{L^2(B_2)}+C\|f\|_{C^{\beta}(B_2)}+C\|g\|_{L^p(B_2)},
\end{align*}
and if $v_{5/4}$ is the Newtonian potential of $h_{5/4}$, then using also Morrey's inequality,
\begin{align*}
\|\nabla v_{5/4}\|_{C^{0,1-n/p}(B_{5/4})}&\leq C\|\nabla^2v_{5/4}\|_{L^p(\bR^n)}\leq C\|h_{5/4}\|_{L^p(\bR^n)}\\
&\leq C\|u\|_{L^2(B_2)}+C\|f\|_{C^{\beta}(B_2)}+C\|g\|_{L^p(B_2)}.
\end{align*}
Since now $u$ solves the equation $-\dive(A\nabla u)=\dive(bu-f+\nabla v_{5/4})$ in $B_{9/8}$ from the definition of $\phi$, (2.1) and (2.2) in \cite{KenigShen} for $\gamma=\min\left\{\alpha,\beta,1-\frac{n}{p}\right\}$ show that
\begin{align*}
\|\nabla u\|_{C^{\gamma}(B_1)}&\leq C\|u\|_{L^2(B_{9/8})}+ C\|bu\|_{C^{0,\gamma}(B_{5/4})}+C\|f\|_{C^{0,\gamma}(B_{5/4})}+C\|\nabla v_{5/4}\|_{C^{0,\gamma}(B_{5/4})}\\
&\leq C\|u\|_{L^2(B_2)}+ C\|b\|_{C^{\alpha}}\|u\|_{C^1(B_{5/4})}+C\|f\|_{C^{\beta}(B_2)}+C\|g\|_{L^p(B_2)}\\
&\leq C\|u\|_{L^2(B_2)}+C\|f\|_{C^{\beta}(B_2)}+C\|g\|_{L^p(B_2)},
\end{align*}
where we used \eqref{eq:5/4Estimate} for the third estimate. Also, for any $x,y\in B_1$, $|u(x)|\leq |u(y)|+\|\nabla u\|_{L^{\infty}(B_1)}$, therefore
\[
\|u\|_{L^{\infty}(B_1)}\leq\fint_{B_1}|u|+\|\nabla u\|_{L^{\infty}(B_1)}\leq C\|u\|_{L^2(B_2)}+C\|f\|_{C^{\beta}(B_2)}+C\|g\|_{L^p(B_2)},
\]
which completes the proof.
\end{proof}

\subsection{A solid integral estimate}
In the following, we will need an estimate of solid integrals of solutions by surface integrals. In the absence of lower order terms, the quantity $\int_{\Omega}|\nabla u|^2$ can be estimated by the $L^2(\partial\Omega)$ norms of $u$ and $\partial_{\nu}u$, using (5.2) in \cite{KenigShen}. However, our operators are not necessarily coercive, so we will need to show an analogous estimate by reducing our case to the equation without lower order terms.

For a Lipschitz domain $\Omega$, let $\Tr:W^{1,2}(\Omega)\to L^2(\partial\Omega)$ be the trace operator. We then have the following lemma.

\begin{lemma}\label{TraceToNt}
Let $\Omega\subseteq\bR^n$ be a Lipschitz domain. Let also $u\in W^{1,2}(\Omega)\cap C(\Omega)$ and $f\in L^2(\partial\Omega)$, with $u\to f$ nontangentially, almost everywhere. Then, $\Tr u=f$ on $\partial\Omega$.
\end{lemma}
\begin{proof}
Set $T_i=T_{r_{\Omega}}(q_i)$ and $\Delta_i=\Delta_{r_{\Omega}}(q_i)$, from Definition~\ref{LipDom} and \eqref{eq:DT}. Let also $\Tr_i$ be the restriction of the trace operator $\Tr:W^{1,2}(\Omega)\to L^2(\partial\Omega)$ on $\Delta_i$, and define $u_m(x',x_n)=u\left(x',x_n+1/m\right)$ in the coordinate system of $B_{r_{\Omega}}(q_i)$, for $(x',x_n)\in T_i$ and $m$ sufficiently large. Then $u_m\to u$ in $W^{1,2}(T_i)$, therefore $\Tr_iu_m\to\Tr_i u$. Since $u\in C(\Omega)$, we obtain that $\Tr_iu_m=u_m$. Moreover, from nontangential convergence and Lemma~\ref{LiesAbove}, $u_m\to f$ almost everywhere on $\Delta_i$, hence $f|_{\Delta_i}=\Tr_i u$, which completes the proof.
\end{proof}

We now show the solid integral estimate.

\begin{prop}\label{SolidByBoundary}
Let $\Omega\subseteq\bR^n$ be a Lipschitz domain with $\diam(\Omega)<\frac{1}{8}$. Assume that $A\in\M_{\Omega}(\lambda,\alpha,\tau)$ and $b,c,d\in L^p(\Omega)$ for some $p>n$, with either $d\geq\dive b$ or $d\geq\dive c$. Suppose also that $u\in W_{{\rm loc}}^{1,2}(\Omega)$ solves the equation $-\dive(A\nabla u+bu)+c\nabla u+du=0$ in $\Omega$, and assume that $(\nabla u)^*\in L^2(\partial\Omega)$. If $u$ converges nontangentially, almost everywhere, to $u|_{\partial\Omega}\in W^{1,2}(\partial\Omega)$, then
\begin{equation}\label{eq:SolidByBoundary}
\int_{\Omega}|u|^2+\int_{\Omega}|\nabla u|^2\leq C \int_{\partial\Omega}|u|^2\,d\sigma+C\int_{\partial\Omega}|\nabla_Tu|^2\,d\sigma,
\end{equation}
where $C$ depends on $n,p,\lambda,\alpha,\tau,\|A\|_{\infty},\|b\|_p,\|c\|_p,\|d\|_p$ and the Lipschitz character of $\Omega$.
\end{prop}
\begin{proof}
As in Lemmas 2.1 and 2.2 in \cite{thesis}, we can construct $\tilde{A}\in\M_{\bR^n}(\lambda,\alpha,C(\tau+\|A\|_{\infty})$ that is 1-periodic in $\bR^n$ (as in (1.3) in \cite{KenigShen}) and extends $A$. We then consider the fundamental solution $\Gamma$ for the operator $\tilde{\mathcal{L}}u=-\dive(\tilde{A}\nabla u)$ in $\bR^n$, which exists from the argument after Lemma 2.1 in \cite{KenigShen}. We also define $\tilde{\mathcal{S}}f(x)=\int_{\partial\Omega}\Gamma(x,q)f(q)\,d\sigma(q)$ for $f\in L^2(\partial\Omega)$, which is the single layer potential operator in (4.1) in \cite{KenigShen}. Then, from the proof of Theorem 1.22 (page 182) in \cite{NguyenPaper}, $\tilde{\mathcal{S}}:L^2(\partial\Omega)\to W^{1,2}(\partial\Omega)$ is invertible. An inspection of the same proof shows that $\|\tilde{\mathcal{S}}^{-1}\|\leq C$, where $C$ depends on $n,\lambda,\alpha,\tau$,$\|A\|_{\infty}$ and the Lipschitz character of $\Omega$.

We now define
\begin{equation}\label{eq:FormulaForv}
v(x)=\int_{\partial\Omega}\Gamma(x,q)\tilde{\mathcal{S}}^{-1}u(q)\,d\sigma(q).
\end{equation}
From Theorems 3.1 and 4.7 in \cite{KenigShen}, $v$ solves $\tilde{\mathcal{L}}u=-\dive(\tilde{A}\nabla u)$ in $\Omega$, $v\to\tilde{\mathcal{S}}(\tilde{\mathcal{S}}^{-1}u)=u$ nontangentially, almost everywhere on $\partial\Omega$, and also $\|(\nabla v)^*\|_{L^2(\partial\Omega)}\leq C\|\tilde{\mathcal{S}}^{-1}u\|_{L^2(\partial\Omega)}$. Therefore, from the bound on $\|\tilde{\mathcal{S}}^{-1}\|$, we obtain that $\|(\nabla v)^*\|_{L^2(\partial\Omega)}\leq C\|u\|_{W^{1,2}(\partial\Omega)}$.

Define $w=u-v\in W^{1,2}(\Omega)$, then $w\in C(\Omega)$, from Proposition~\ref{CAlphaReg}. Moreover, $w$ converges to $0$ nontangentially, almost everywhere on $\partial\Omega$, therefore Lemma~\ref{TraceToNt} shows that $w\in W_0^{1,2}(\Omega)$. In addition, $w$ solves the equation $-\dive(A\nabla w+bw)+c\nabla w+dw=\dive(bv)-c\nabla v-dv=F\in W^{-1,2}(\Omega)$. To find the $W^{-1,2}(\Omega)$ norm of $F$, let $\phi\in C_c^{\infty}(\Omega)$. We then compute
\begin{align*}
|F\phi|&=\left|\int_{\Omega}bv\nabla\phi+c\nabla v\cdot\phi+dv\phi\right|\leq\|b\|_n\|v\|_{2^*}\|\nabla\phi\|_2+\left(\|c\|_n\|\nabla v\|_2+\|d\|_{n/2}\|v\|_{2^*}\right)\|\phi\|_{2^*}\\
&\leq C_n\left(\|b\|_p\|v\|_{2^*}+\|c\|_p\|\nabla v\|_2+\|d\|_p\|v\|_{2^*}\right)\|\nabla \phi\|_2,
\end{align*}
from Sobolev's inequality and the fact that $\diam(\Omega)<\frac{1}{8}$. Using Proposition 6.14 in \cite{KimSak} and the Sobolev embedding $\|v\|_{2^*}\leq C\|v\|_2+C\|\nabla v\|_2$, we obtain that
\begin{equation}\label{eq:wj}
\|w\|_{W_0^{1,2}(\Omega)}\leq C\|F\|_{W^{-1,2}(\Omega)}\leq C_n\left(\|b\|_p\|v\|_{2^*}+\|c\|_p\|\|\nabla v\|_2+\|d\|_p\|v\|_{2^*}\right)\leq C\|v\|_2+C\|\nabla v\|_2,
\end{equation}
where $C$ depends on $n,p,\lambda,\|A\|_{\infty},\|b\|_p,\|c\|_p,\|d\|_p$ and the Lipschitz character of $\partial\Omega$. Since now $u=v+w$ in $\Omega$, we obtain that
\begin{equation}\label{eq:Main}
\|\nabla u\|_2\leq \|\nabla v\|_2+\|\nabla w\|_2\leq C\|v\|_2+C\|\nabla v\|_2.
\end{equation}
We first bound the last term: using (5.2) in \cite{KenigShen} we estimate
\begin{equation}\label{eq:lambdaI_2^2}
\lambda\int_{\Omega}|\nabla v|^2\leq\int_{\Omega}A\nabla v\nabla v=\int_{\partial\Omega}\partial_{\nu}v\cdot v\,d\sigma\leq \int_{\partial\Omega}|\partial_{\nu}v|^2\,d\sigma+\int_{\partial\Omega}|u|^2\,d\sigma,
\end{equation}
since $v=u$ on $\partial\Omega$. Now, from right before (4.13) and right after (4.15) in \cite{KenigShen}, we obtain that
\[
\|\partial_{\nu}v\|_{L^2(\partial\Omega)}\leq C\|\tilde{\mathcal{S}}^{-1}u\|_{L^2(\partial\Omega)}\leq C\|u\|_{W^{1,2}(\partial\Omega)},
\]
where $C$ depends on $n,\lambda,\alpha,\tau$,$\|A\|_{\infty}$ and the Lipschitz character of $\Omega$. Therefore, plugging in \eqref{eq:lambdaI_2^2},
\begin{equation}\label{eq:I_2a}
\lambda\int_{\Omega}|\nabla v|^2\leq C \int_{\partial\Omega}|\nabla_Tu|^2\,d\sigma+\int_{\partial\Omega}|u|^2\,d\sigma.
\end{equation}
Set $s$ to be the average of $v$ in $\Omega$. Using \eqref{eq:FormulaForv} and estimate (2.5) in \cite{KenigShen}, we compute
\begin{align*}
|s|&=\left|\fint_{\Omega}v\right|\leq\frac{C}{|\Omega|}\int_{\partial\Omega}\left(\int_{\Omega}|x-q|^{2-n}\,dx\right)|\tilde{\mathcal{S}}^{-1}u(q)|\,d\sigma(q)\leq\frac{C}{|\Omega|}\int_{\partial\Omega}|\tilde{\mathcal{S}}^{-1}u(q)|\,d\sigma(q)\\
&\leq\frac{C\sigma(\partial\Omega)^{1/2}}{|\Omega|}\|\tilde{\mathcal{S}}^{-1}u\|_{L^2(\partial\Omega)}\leq C\|u\|_{W^{1,2}(\partial\Omega)}
\end{align*}
where $C$ depends on $n,\lambda,\alpha,\tau$,$\|A\|_{\infty}$ and the Lipschitz character of $\Omega$, and where we used the bound before \eqref{eq:FormulaForv} for the last estimate. Therefore, the last estimate for $|s|$ and Poincare's inequality in $\Omega$ show that
\begin{align}\nonumber
\int_{\Omega}|v|^2&\leq C\int_{\Omega}|v-s|^2+C\int_{\Omega}|s|^2\leq C\int_{\Omega}|\nabla v|^2+C|s|^2\\
\label{eq:vj}
&\leq C\int_{\Omega}|\nabla v|^2+C\|u\|_{W^{1,2}(\partial\Omega)}^2\leq C\int_{\partial\Omega}|\nabla_Tu|^2\,d\sigma+\int_{\partial\Omega}|u|^2\,d\sigma,
\end{align}
where we used \eqref{eq:I_2a} for the last estimate. We then plug \eqref{eq:vj} and \eqref{eq:I_2a} to \eqref{eq:Main}, and we obtain
\begin{equation}\label{eq:nablabound}
\int_{\Omega}|\nabla u|^2\leq C\int_{\partial\Omega}|\nabla_Tu|^2\,d\sigma+\int_{\partial\Omega}|u|^2\,d\sigma.
\end{equation}
Finally, we use $u=v+w$, \eqref{eq:wj}, \eqref{eq:I_2a} and \eqref{eq:vj} to obtain
\[
\int_{\Omega}|u|^2\leq C\|v\|_2^2+C\|w\|_2^2\leq C\|v\|_2^2+C\|\nabla v\|_2^2\leq C\int_{\partial\Omega}|u|^2\,d\sigma+C\int_{\partial\Omega}|\nabla_Tu|^2\,d\sigma.
\]
Then, adding the last estimate to \eqref{eq:nablabound} completes the proof.
\end{proof}

\subsection{The Rellich estimate}
We now turn our attention to the Rellich estimate, which we will use to show invertibility of the single layer potential operator in a special case. For this purpose, we consider a strengthening of (6.1) in \cite{KenigShen}: that is, we consider functions satisfying
\begin{equation}\label{eq:ACond}
f\in C^2(\Omega),\quad|\nabla f(x)|\leq C_1\delta(x)^{\alpha_0-1},\quad |\nabla^2f(x)|\leq C_1\delta(x)^{\alpha_0-2},
\end{equation}
for some $C_1>0$ and $a_0\in(0,1)$. Although the bound on the second derivatives will not be used in order to deduce the Rellich estimate, we will need it in order to show Lemma~\ref{ImprovedIntegrability}.

A crucial ingredient in the proof of the Rellich estimate is that the coefficient matrix $A$ is symmetric. However, using an integration by parts argument from \cite{PipherDrifts}, we can extend the Rellich estimate for solutions to equations with $A$ not necessarily symmetric.

We remark that a similar estimate to the one we will show can be found in Section 4.1 of \cite{XuZhaoZhou}. However, we carry out the proof for the sake of completeness, and in order to show how to extend this estimate in the case of non-symmetric matrices.

We first turn to the next lemma, which is a modification of Lemma 11.1 in \cite{thesis} (see also Remark 2.12 in \cite{PipherDrifts}).
\begin{lemma}\label{ReduceToSymmetric}
Let $\Omega\subseteq\mathbb R^n$ be a bounded domain, and let $A\in C^1_{\loc}(\Omega)$ be uniformly elliptic, and $b\in L^{\infty}_{\loc}(\Omega)$, $c\in L^p_{\loc}(\Omega)$ for some $p>n$. Let $A^s=\frac{A+A^t}{2}$, and define the vector function $\tilde{b}\in L^{\infty}_{\loc}(\Omega)$ by
\[
\tilde{b}_i=\frac{1}{2}\sum_{j=1}^n\partial_j(a_{ij}-a_{ji})\in L^{\infty}_{\loc}(\Omega),\quad i=1,\dots n.
\]
Then $\dive\tilde{b}=0$ in the sense of distributions, and if $u\in W^{1,2}_{\loc}(\Omega)$ is a solution to the equation $-\dive(A\nabla u+bu)+c\nabla u=0$ in $\Omega$, then $u$ also solves the equation $-\dive(A^s\nabla u+bu)+(c+\tilde{b})\nabla u=0$ in $\Omega$. 
\end{lemma}
	
We now show the local Rellich estimate. We will use the notation $\partial_{\nu}u=\left<A\nabla u,\nu\right>$.
	
\begin{lemma}\label{Rellich}
Let $\Omega\subseteq\mathbb R^n$ be a Lipschitz domain with Lipschitz constant $M$. Let also $A\in\M_{\Omega}(\lambda,\alpha,\tau)$ and $b\in C_{\Omega}(\alpha,\tau)$, where $A$ and $b$ satisfy Condition \eqref{eq:ACond}, and $c\in L^p(\Omega)$ for some $p>n$. Suppose that $u\in W^{1,2}_{\loc}(\Omega)$ is a solution to the equation $-\dive(A\nabla u+bu)+c\nabla u=0$ in $\Omega$, with $(\nabla u)^*\in L^2(\partial\Omega)$ and $\nabla u$ converging nontangentially, almost everywhere on $\partial\Omega$. Then, for any $q\in\partial\Omega$ and $\rho\in(0,r_{\Omega})$,
\begin{multline}\label{eq:LocalRellich}
\int_{\Delta_{\rho}(q)}|\partial_{\nu}u|^2\,d\sigma\leq C\int_{\Delta_{2\rho}(q)}|\nabla_Tu|^2\,d\sigma+C\int_{T_{2\rho}(q)}|\nabla A||\nabla u|^2\\
+C\int_{T_{2\rho}(q)}|\dive b||u||\nabla u|+C\int_{T_{2\rho}(q)}|c||\nabla u|^2+\frac{C}{\rho}\int_{T_{2\rho}(q)}|\nabla u|^2,
\end{multline}
where $C$ depends on $n,\lambda$, $\|A\|_{\infty}$, $\|b\|_{\infty}$, $M$ and $r_{\Omega}$, and $\Delta_{\rho}(q)$, $T_{\rho}(q)$ are defined in \eqref{eq:DT}.
\end{lemma}
\begin{proof}
As right before \eqref{eq:DT}, we will assume that $q\in B_{r_{\Omega}}(q_i)$, and we will consider the coordinate system for $B_{r_{\Omega}}(q_i)$. Then, for $\e>0$ and $r\in(0,2r_{\Omega})$, consider the sets
\[
T_r^{\e}(q)=T_r(q)+\e e_n,\qquad \Delta_r^{\e}(q)=\Delta_r(q)+\e e_n.
\]
Let $\sigma_{\e}$ be the surface measure on $\Delta_{3\rho/2}^{\e}(q)$, and denote by $\partial_{\nu_{\e}}$, $\nabla_{T_{\e}}$ the normal derivative and the tangential component of the derivative on $\Delta_{3\rho/2}^{\e}(q)$, respectively. To show the estimate, we claim that it is enough to show that
\begin{multline}\label{eq:2Der}
\int_{\Delta_{\rho}^{\e}(q)}|\partial_{\nu_{\e}}u|^2\,d\sigma_{\e}\leq C\int_{\Delta_{3\rho/2}^{\e}(q)}|\nabla_{T_{\e}}u|^2\,d\sigma_{\e}+C\int_{T_{3\rho/2}^{\e}(q)}|\nabla A||\nabla u|^2\\
+C\int_{T_{3\rho/2}^{\e}(q)}|\dive b||u||\nabla u|+C\int_{T_{3\rho/2}^{\e}(q)}|c||\nabla u|^2+\frac{C}{\rho}\int_{T_{3\rho/2}^{\e}(q)}|\nabla u|^2,
\end{multline}
where $C$ depends on $n,\lambda$, $\|A\|_{\infty},\|b\|_{\infty}$, $M$ and $r_{\Omega}$. Indeed, if this is the case, then by nontangential convergence and the fact that $(\nabla u)^*\in L^2(\partial\Omega)$ we obtain that
\begin{equation*}
\lim_{\e\to 0}\int_{\Delta_{\rho}^{\e}(q)}|\partial_{\nu_{\e}}u|^2\,d\sigma_{\e}=\int_{\Delta_{\rho}(q)}|\partial_{\nu}u|^2\,d\sigma,\qquad \lim_{\e\to 0}\int_{\Delta_{3\rho/2}^{\e}(q)}|\nabla_Tu|^2\,d\sigma_{\e}=\int_{\Delta_{3\rho/2}(q)}|\nabla_{T_{\e}}u|^2\,d\sigma.
\end{equation*}
Then, to show the estimate, we consider the $\limsup$ as $\e\to 0$ in \eqref{eq:2Der} and we note that, for $\e>0$ sufficiently small, $T_{3\rho/2}^{\e}(q)\subseteq T_{2\rho}(q)$. Therefore, it is enough to show \eqref{eq:2Der}.
		
To show \eqref{eq:2Der}, we denote the sets $\Delta_{\rho}^{\e}(q)$, $\Delta_{3\rho/2}^{\e}(q)$ and $T_{3\rho/2}^{\e}(q)$ by $\Delta$, $\Delta'$ and $T'$, respectively, and set $T_{\rho}^{\e}(q)=T$. First, from Lemma~\ref{ReduceToSymmetric}, $u$ is a solution to the equation
\begin{equation}\label{eq:EqForu}
-\dive(A^s\nabla u+bu)+(c+\tilde{b})\nabla u=0
\end{equation}
in $\Omega$, where $\tilde{b}$ is defined in Lemma~\ref{ReduceToSymmetric}. Moreover, for some constant $C=C_n$,
\begin{equation}\label{tildeb}
|\tilde{b}_i(x)|\leq\frac{1}{2}\sum_{j=1}^n|\partial_ja_{ij}(x)-\partial_ja_{ji}(x)|\leq C|\nabla A(x)|,\quad i=1,\dots n.
\end{equation}
Let now $T_0$ be such that $T'\subseteq T_0\subseteq\Omega$, where all inclusions are compact. Since $A,b\in C^1_{\loc}(\Omega)$ and $T_0$ is compactly supported in $\Omega$, we obtain from Proposition~\ref{CAlphaReg} that $u\in C^1(\overline{T_0})$. Therefore $|c||\nabla u|^2\in L^1(T')$, and also $u$ is a solution to the equation
\[
-\dive(A\nabla u+bu)=-c\nabla u\in L^2(T_0)
\]
in $T_0$, hence Theorem 8.8 in \cite{Gilbarg} shows that $u\in W^{2,2}_{\loc}(T_0)$. Therefore $u\in W^{2,2}(T')$.
		
From the definitions of $T$ and $T'$, we can construct a smooth cutoff $\theta_0$ in $\mathbb R^n$, with $\theta_0\equiv 1$ in $T$, $\theta_0$ vanishing in $T'\setminus T_{5\rho/4}^{\e}(q)$, $0\leq\theta_0\leq 1$ in $T'$ and $|\nabla\theta_0|\leq C/\rho$, where $C$ depends on $n$ and $M$. Then, if $\partial_nu$ denotes $\left<\nabla u,e_n\right>$, using \eqref{eq:EqForu} and the fact that $u\in W^{2,2}(T')$, we compute in $T'$,
\begin{align*}
\dive(\left<A^s\nabla u,\nabla u\right>e_n)-2\dive(\partial_nuA^s\nabla u)&=\partial_n(\left<A^s\nabla u,\nabla u\right>)-2\partial_nu\dive(A^s\nabla u)-2\left<\nabla\partial_n u,A^s\nabla u\right>\\
&=\left<\partial_nA^s\cdot\nabla u,\nabla u\right>+2\partial_nu \dive(bu)-2\partial_nu\cdot(c+\tilde{b})\nabla u\\
&=\left<\partial_nA^s\cdot\nabla u,\nabla u\right>+2\dive b\cdot u\partial_n u+2\overline{b}\nabla u\cdot \partial_nu,
\end{align*}
since $A^s$ is symmetric, where we define $\overline{b}=b-\tilde{b}-c$. Therefore, after multiplying with $\theta_0$, we obtain
\begin{multline*}
\dive(\theta_0\left<A^s\nabla u,\nabla u\right>e_n)-2\dive(\theta_0\partial_nuA^s\nabla u)=\\
\theta_0\left<\partial_nA^s\cdot\nabla u,\nabla u\right>+2\theta_0\dive b\cdot u\partial_n u+2\theta_0\overline{b}\nabla u\cdot\partial_nu+\partial_n\theta_0\left<A^s\nabla u,\nabla u\right>-2\partial_nu\left<A^s\nabla u,\nabla\theta_0\right>.
\end{multline*}
For simplicity, denote $\sigma_{\e}$ on $\Delta'$ by $\sigma$, and $\nabla_{T_{\e}}$ by $\nabla_T$. Since $T'$ is a Lipschitz domain, using the fact that $u\in W^{2,2}(T')$, the divergence theorem in $T'$ and the support properties of $\theta_0$, we obtain
\begin{multline*}
\int_{\Delta'}\theta_0\left(\left<A^s\nabla u,\nabla u\right>\left<e_n,\nu\right>-2\left<\nabla u,e_n\right>\left<A^s\nabla u,\nu\right>\right)\,d\sigma=\int_{T'}\theta_0\left<\partial_nA^s\cdot\nabla u,\nabla u\right>\\
+\int_{T'}\left(2\theta_0\dive b\cdot u\partial_n u+2\theta_0\overline{b}\nabla u\cdot\partial_nu+\partial_n\theta_0\left<A^s\nabla u,\nabla u\right>-2\partial_nu\left<A^s\nabla u,\nabla\theta_0\right>\right),
\end{multline*}
therefore, changing signs in the left hand side,
\begin{multline*}
\int_{\Delta'}\theta_0\left(2\left<\nabla u,e_n\right>\left<A^s\nabla u,\nu\right>-\left<A^s\nabla u,\nabla u\right>\left<e_n,\nu\right>\right)\,d\sigma\leq\\
C\int_{T'}\left(|\nabla A^s|+|\overline{b}|+|\nabla\theta_0||A^s|\right)|\nabla u|^2+|\dive b||u||\nabla u|,
\end{multline*}
for $C=C_n$. Note now that, from $\overline{b}=b-\tilde{b}-c$ and \eqref{tildeb},
\[
|\overline{b}|\leq |b|+|\tilde{b}|+|c|\leq |b|+C_n|\nabla A|+|c|,
\]
therefore using this bound, $A^s=\frac{A+A^t}{2}$ and $|\nabla\theta_0|\leq C/\rho$, we obtain that
\begin{multline}\label{eq:FirstRel}
\int_{\Delta'}\theta_0\left(2\left<\nabla u,e_n\right>\left<A^s\nabla u,\nu\right>-\left<A^s\nabla u,\nabla u\right>\left<e_n,\nu\right>\right)\,d\sigma\leq\\
C\int_{T'}\left(|\nabla A|+|b|+|c|\right)|\nabla u|^2+\frac{C}{\rho}\int_{T'}|A||\nabla u|^2+C\int_{T'}|\dive b||u||\nabla u|=\mathcal{Q},
\end{multline}
where $C$ depends on $n$ and $M$.
		
We now treat the left hand side. For simplicity, we denote the outer normal $\nu_{\e}$ on $\Delta'$ by $\nu$, and we write $\nabla u=\nabla_Tu+\partial_{\nu}^0u\cdot\nu$, where $\partial_{\nu}^0$ denotes differentiation with respect to $\nu$. We then compute
\begin{multline*}
2\left<\partial_{\nu}^0u\cdot\nu,e_n\right>\left<A^s\nabla u,\nu\right>-\left<A^s\nabla u,\nabla u\right>\left<e_n,\nu\right>
=\left<A^s\nabla u,2\partial_{\nu}^0u\cdot\nu-\nabla u\right>\left<e_n,\nu\right>\\
=\left<A^s\nabla u,\partial_{\nu}^0u\cdot\nu-\nabla_Tu\right>\left<e_n,\nu\right>=\left(\left<A^s\nu,\nu\right>|\partial_{\nu}^0u|^2-\left<A^s\nabla_Tu,\nabla_Tu\right>\right)\left<e_n,\nu\right>,
\end{multline*}
where we used that $A^s$ is symmetric for the last equality. Adding the term $2\left<\nabla_Tu,e_n\right>\left<A^s\nabla u,\nu\right>$ to the first and last terms of the previous equality, we have
\begin{multline*}
2\left<\nabla u,e_n\right>\left<A^s\nabla u,\nu\right>-\left<A^s\nabla u,\nabla u\right>\left<e_n,\nu\right>\\
=2\left<\nabla_Tu,e_n\right>\left<A^s\nabla u,\nu\right>+\left<A^s\nu,\nu\right>|\partial_{\nu}^0u|^2\left<e_n,\nu\right>-\left<A^s\nabla_Tu,\nabla_Tu\right>\left<e_n,\nu\right>,
\end{multline*}
therefore, plugging in \eqref{eq:FirstRel}, we obtain that
\[
\int_{\Delta'}\theta_0\left(2\left<\nabla_Tu,e_n\right>\left<A^s\nabla u,\nu\right>+\left<A^s\nu,\nu\right>|\partial_{\nu}^0u|^2\left<e_n,\nu\right>-\left<A^s\nabla_Tu,\nabla_Tu\right>\left<e_n,\nu\right>\right)\,d\sigma\leq\mathcal{Q},
\]
and, after rearranging,
\[
\int_{\Delta'}\theta_0\left<A^s\nu,\nu\right>|\partial_{\nu}^0u|^2\left<e_n,\nu\right>\,d\sigma\leq \int_{\Delta'}\theta_0\left(\left<A^s\nabla_Tu,\nabla_Tu\right>\left<e_n,\nu\right>-2\left<\nabla_Tu,e_n\right>\left<A^s\nabla u,\nu\right>\right)\,d\sigma+\mathcal{Q}.
\]
Note that $\left<e_n,\nu\right>\leq 1$, and if $\psi_i$ is the coordinate map for $\partial\Omega$ in $B_{r_{\Omega}}(q_i)$, then
\[
\left<e_n,\nu\right>=\left<e_n,\frac{(-\nabla\psi_i,1)}{|(-\nabla\psi_i,1)|}\right>=\frac{1}{\sqrt{|\nabla\psi_i|^2+1}}\geq\frac{1}{\sqrt{nM^2+1}}.
\]
Moreover, $\left<A^s\nu,\nu\right>=\left<A\nu,\nu\right>\geq\lambda|\nu|^2=\lambda$, therefore
\begin{align*}
\int_{\Delta'}\theta_0|\partial_{\nu}^0u|^2\,d\sigma&\leq C \int_{\Delta'}\theta_0|\left<A^s\nabla_Tu,\nabla_Tu\right>|\,d\sigma+C\int_{\Delta'}\theta_0\left|\left<\nabla_Tu,e_n\right>\left<A^s\nabla u,\nu\right>\right|\,d\sigma+C\mathcal{Q}\\
&\leq C\int_{\Delta'}\theta_0|\nabla_Tu|^2\,d\sigma+C\int_{\Delta'}\theta_0|\nabla_Tu||\nabla u|\,d\sigma+C\mathcal{Q},
\end{align*}
where $C$ depends on $n,\lambda$, $M$ and $\|A\|_{\infty}$. We now add the term $\int_{\Delta'}\theta_0|\nabla_Tu|^2$ to both sides, to obtain that
\begin{align*}
\int_{\Delta'}\theta_0|\nabla u|^2\,d\sigma&\leq C \int_{\Delta'}\theta_0|\nabla_Tu|^2\,d\sigma+C\int_{\Delta'}\theta_0|\nabla_Tu||\nabla u|\,d\sigma+C\mathcal{Q}\\
&\leq C\int_{\Delta'}\theta_0|\nabla_Tu|^2\,d\sigma+\frac{C}{4\delta}\int_{\Delta'}\theta_0|\nabla_Tu|^2\,d\sigma+C\delta\int_{\Delta'}\theta_0|\nabla u|^2\,d\sigma+C\mathcal{Q},
\end{align*}
from the Cauchy-Schwartz inequality, and the Cauchy inequality with $\delta$, where $C$ depends on $n,\lambda$, $M$ and $\|A\|_{\infty}$. Choosing $\delta$ such that $C\delta<1/2$, and using $|\left<A\nabla u,\nu\right>|^2\leq C|\nabla u|^2$ and the support properties of $\theta_0$, we obtain 
\begin{multline*}
\int_{\Delta}|\left<A\nabla u,\nu\right>|^2\,d\sigma\leq  C\int_{\Delta'}|\nabla_Tu|^2\,d\sigma+C\mathcal{Q}\\
=C\int_{\Delta'}|\nabla_Tu|^2\,d\sigma+C\int_{T'}\left(|\nabla A|+|b|+|c|\right)|\nabla u|^2+\frac{C}{\rho}\int_{T'}|A||\nabla u|^2+C\int_{T'}|\dive b||u||\nabla u|,
\end{multline*}
where $C$ depends on $n,\lambda$,$M$ and $\|A\|_{\infty}$. Since $\rho\in(0,r_{\Omega})$, we obtain that $|b|\leq\frac{|b|r_{\Omega}}{\rho}$, hence we obtain \eqref{eq:2Der} with a constant $C$ that depends on $n,\lambda$, $\|A\|_{\infty}$, $\|b\|_{\infty}$, $M$ and $r_{\Omega}$. This completes the proof.
\end{proof}

We now turn to the global analog of the Rellich estimate, in which the nontangential maximal function $(\nabla u)^*$ will appear.
	
\begin{lemma}\label{GlobalRellich}
Under the same assumptions as in Lemma~\ref{Rellich}, and assuming also that $\diam(\Omega)<\frac{1}{8}$, $\dive c\leq 0$ and $u$ converges nontangentially, almost everywhere to $u|_{\partial\Omega}\in W^{1,2}(\partial\Omega)$, then for any $\rho\in(0,r_{\Omega})$,
\begin{equation*}
\int_{\partial\Omega}|\partial_{\nu}u|^2\,d\sigma\leq C\rho^{-1}\left(\int_{\partial\Omega}|u|^2\,d\sigma+\int_{\partial\Omega}|\nabla_Tu|^2\,d\sigma\right)+C\int_{\Omega_{C_0\rho}}|c||\nabla u|^2+C\rho^{\alpha_0}\int_{\partial\Omega}|(\nabla u)^*|^2\,d\sigma,
\end{equation*}
where $C_0$ depends on $n$ and $M$, and $C$ depends on $n,p,\lambda,\|A\|_{\infty},\|b\|_{\infty},\|c\|_p$, the constants $C_1$ and $\alpha_0$ that appear in Condition~\eqref{eq:ACond} and the Lipschitz character of $\Omega$.
\end{lemma}
\begin{proof}
Consider the coordinate system for $B_{r_{\Omega}}(q_i)$ from Definition~\ref{LipDom} and set $\Delta_i=\Delta_{r_{\Omega}}(q_i)$, $2\Delta_i=\Delta_{2r_{\Omega}}(q_i)$. Then, from Fubini's theorem,
\begin{align}\nonumber
\int_{\Delta_i}|\partial_{\nu}u|^2\,d\sigma&\leq C\rho^{1-n}\int_{\Delta_i}|\partial_{\nu}u(q')|^2\sigma(\Delta_{\rho}(q'))\,d\sigma(q')=C\rho^{1-n}\int_{\Delta_i}\int_{\Delta_{\rho}(q')}|\partial_{\nu}u(q')|^2\,d\sigma(q)d\sigma(q')\\
\label{eq:DoubleSurface}
&\leq C\rho^{1-n}\int_{2\Delta_i}\int_{\Delta_{\rho}(q)}|\partial_{\nu}u(q')|^2\,d\sigma(q')d\sigma(q),
\end{align}
where $C$ depends on $n$ and $M$. In a similar way,
\begin{equation}\label{eq:DoubleSurface2}
\int_{2\Delta_i}\int_{\Delta_{2\rho}(q)}|\nabla_Tu(q')|^2\,d\sigma(q')d\sigma(q)\leq C\rho^{n-1}\int_{\partial\Omega}|\nabla_Tu|^2\,d\sigma.
\end{equation}
Moreover, if $q\in 2\Delta_i$ and $x\in T_{2\rho}(q)$, then $x\in\Omega_{C_0\rho}$ and $q\in B_{C_0\rho}(x)\cap\partial\Omega$ for some $C_0$ that depends on $n$ and $M$. Therefore, for any measurable $f\geq 0$ in $\Omega$, using Fubini's theorem we obtain
\[
\int_{2\Delta_i}\int_{T_{2\rho}(q)}f(x)\,dxd\sigma(q)\leq \int_{\Omega_{C_0\rho}}\int_{B_{C_0\rho}(x)\cap\partial\Omega}f(x)\,d\sigma(q)dx\leq C\rho^{n-1}\int_{\Omega_{C_0\rho}}f(x)\,dx.
\]
Hence, integrating \eqref{eq:LocalRellich} for $q\in 2\Delta_i$ and using \eqref{eq:DoubleSurface} and \eqref{eq:DoubleSurface2} we obtain
\begin{multline}\label{eq:nTu}
\int_{\Delta_i}|\partial_{\nu}u|^2\,d\sigma\leq C\int_{\partial\Omega}|\nabla_Tu|^2\,d\sigma+C\int_{\Omega_{C_0\rho}}|\nabla A||\nabla u|^2\\+C\int_{\Omega_{C_0\rho}}|\dive b||u||\nabla u|+C\int_{\Omega_{C_0\rho}}|c||\nabla u|^2+\frac{C}{\rho}\int_{\Omega_{C_0\rho}}|\nabla u|^2.
\end{multline}
We now proceed as in the proof of Lemma 6.6 in \cite{KenigShen}: using Condition \eqref{eq:ACond}, we obtain
\begin{equation}\label{eqeq:GradABound}
\int_{\Omega_{C_0\rho}}|\nabla A||\nabla u|^2\leq C_1\int_{\Omega_{C_0\rho}}\delta(x)^{\alpha_0-1}|\nabla u(x)|^2\,dx\leq C\rho^{\alpha_0}\int_{\partial\Omega}|(\nabla u)^*|^2\,d\sigma,
\end{equation}
where $C$ depends on $C_1$, $\alpha_0$ and the Lipschitz character of $\Omega$. Moreover, using Condition~\eqref{eq:ACond},
\begin{multline}\label{eqeq:DivbBound}
\int_{\Omega_{C_0\rho}}|\dive b||u||\nabla u|\leq C\left(\int_{\Omega_{C_0\rho}}|u|^2\right)^{1/2}\left(\int_{\Omega_{C_0\rho}}\delta(x)^{2\alpha_0-2}|\nabla u(x)|^2\,dx\right)^{1/2}\\
\hspace{10mm}\leq C\|u\|_{L^2(\Omega)}\left(\rho^{2\alpha_0-1}\int_{\partial\Omega}|(\nabla u)^*|^2\,d\sigma\right)^{1/2}\leq C\rho^{-1}\|u\|_{L^2(\Omega)}^2+C\rho^{2\alpha_0}\int_{\partial\Omega}|(\nabla u)^*|^2\,d\sigma,
\end{multline}
where $C$ depends on $n,C_1$, $\alpha_0$ and the Lipschitz character of $\Omega$, and where we also used Lemma~\ref{rOmegaIsGood}. To estimate the last term in \eqref{eq:nTu} we use \eqref{eq:SolidByBoundary}. Then, plugging \eqref{eqeq:GradABound} and \eqref{eqeq:DivbBound} in \eqref{eq:nTu}, adding for $i=1,\dots N$ and using Lemma~\ref{rOmegaIsGood} completes the proof.
\end{proof}

\subsection{A global estimate for the derivative}
We now turn to the following integrability result for the second derivatives of solutions to the equation $-\dive(A\nabla u)=0$.
\begin{lemma}\label{ImprovedDerivativeRegularity}
Let $\Omega\subseteq\bR^n$ be a Lipschitz domain with $\diam(\Omega)<\frac{1}{8}$, and $A\in \M_{\bR^n}(\lambda,\alpha,\tau)$ satisfying Condition \eqref{eq:ACond} in $\Omega$ with $\alpha_0\leq\alpha$. If $u\in W^{1,2}(\Omega)$ is the solution to the $R_2$ Regularity problem for $\mathcal{L}u=-\dive(A\nabla u)=0$ in $\Omega$ (in the sense of Definition 5.2 in \cite{KenigShen}) with $u=f\in W^{1,2}(\partial\Omega)$, then for any $\beta\in\left(0,\frac{1}{2}\right)$,
\[
\int_{\Omega}\left(\delta(z)^{1-\beta}|\nabla ^2u(z)|\right)^2\,dz\leq C\int_{\partial\Omega}|\nabla_T f|^2\,d\sigma,
\]
where $C$ depends on $n,\lambda,\alpha,\beta,\tau,\|A\|_{\infty},C_0,\alpha_0$ and the Lipschitz character of $\Omega$.
\end{lemma}
\begin{proof}
Let $A=(a_{ij})$. Let also $z\in\Omega$, set $r=\delta(z)/8$, $B_z=B_r(z)$, and let $s_z$ be the average of $u$ in $4B_z$. Set $v=u-s_z$ and fix $k=1,\dots n$. Since $v$ solves the equation $-\dive(A\nabla v)=0$, we obtain
\begin{equation*}
\sum_{i,j}a_{ij}\partial_{ij}v=-\sum_{i,j}\partial_ia_{ij}\cdot\partial_jv=g.
\end{equation*}
From Proposition~\ref{CAlphaReg}, $v\in C^{1,\alpha}_{\loc}(\Omega)$, and from Condition~\eqref{eq:ACond}, $A\in C^2_{\loc}(\Omega)$, therefore, $g\in C^{\alpha}_{\loc}(\Omega)$. Hence, from Theorem 6.13 in \cite{Gilbarg} we obtain that $v\in C^{2,\alpha}(4B_z)$.

We now differentiate the equation $-\dive(A\nabla v)=0$ in $4B_z$ with respect to $e_k$. Setting $v_k=\partial_kv$, we obtain
\[
-\dive(A\nabla v_k)=\sum_{i,j}\partial_i(\partial_ka_{ij}\cdot \partial_jv)=\dive f,
\]
where $f_i=\sum_j\partial_ka_{ij}\cdot\partial_jv$. Then, from Proposition~\ref{CAlphaReg},
\begin{equation}\label{eq:NablaVk}
|\nabla(\partial_ku)(z)|\leq\|\nabla v_k\|_{L^{\infty}(B_z)}\leq\frac{C}{r}\left(\fint_{2B_z}|v_k|^2\right)^{1/2}+C\|f\|_{L^{\infty}(2B_z)}+Cr^{\alpha_0}\|f\|_{C^{0,\alpha_0}(2B_z)}.
\end{equation}
Since $\delta(x)>r$ for any $x\in 2B_z$, we use \eqref{eq:ACond} and Proposition~\ref{CAlphaReg} to estimate
\begin{equation}\label{eq:NablaANablaU}
\|f\|_{L^{\infty}(2B_z)}\leq C\|\nabla A\|_{L^{\infty}(2B_z)}\|\nabla v\|_{L^{\infty}(2B_z)}\leq Cr^{\alpha_0-2}\left(\fint_{4B_z}|v|^2\right)^{1/2}\leq Cr^{\alpha_0-1}\left(\fint_{4B_z}|\nabla u|^2\right)^{1/2},
\end{equation}
where we also used Poincar{\'e}'s in the last estimate. Also, from Proposition~\ref{CAlphaReg} and Condition \eqref{eq:ACond},
\begin{align}\nonumber
\|f\|_{C^{0,\alpha_0}(2B_z)}&\leq C\|\nabla A\|_{L^{\infty}(2B_z)}\|\nabla v\|_{C^{0,\alpha_0}(2B_z)}+C\|\nabla A\|_{C^{0,\alpha_0}(2B_z)}\|\nabla v\|_{L^{\infty}(2B_z)}\\
\nonumber
&\leq\left(Cr^{-1-\alpha_0}\|\nabla A\|_{L^{\infty}(2B_z)}+Cr^{-1}\|\nabla A\|_{C^{0,\alpha_0}(2B_z)}\right)\left(\fint_{4B_z}|v|^2\right)^{1/2}\\
\label{eq:NablaANablaU2}
&\leq Cr^{-2}\left(\fint_{4B_z}|v|^2\right)^{1/2}\leq Cr^{-1}\left(\fint_{4B_z}|\nabla u|^2\right)^{1/2}.
\end{align}
Plugging \eqref{eq:NablaANablaU} and \eqref{eq:NablaANablaU2} in \eqref{eq:NablaVk} and considering $k=1,\dots n$, we obtain
\begin{equation}\label{eq:nabla^2Bound}
|\nabla^2u(z)|\leq\left(\frac{C}{r}+Cr^{\alpha_0-1}\right)\left(\fint_{4B_z}|\nabla u|^2\right)^{1/2}\leq\frac{C}{r}\left(\fint_{4B_z}|\nabla u|^2\right)^{1/2},
\end{equation}
where $C$ depends on $n,\lambda,\alpha,\tau,\|A\|_{\infty}, C_0$ and $\alpha_0$.

We now consider two cases: $z\in\Omega^{r_{\Omega}}$, and $z\in\Omega_{r_{\Omega}}$ (from \eqref{eq:CloseFar}). If $z\in\Omega^{r_{\Omega}}$, then $r>r_{\Omega}/8$, so \eqref{eq:nabla^2Bound} shows that
\begin{equation}\label{eq:nabla^2BoundInside1}
|\nabla^2u(z)|\leq C\left(\fint_{4B_z}|\nabla u|^2\right)^{1/2}\leq C\left(\int_{\Omega}|\nabla u|^2\right)^{1/2},
\end{equation}
where $C$ depends on $n,\lambda,\alpha,\tau,\|A\|_{\infty},C_0,\alpha_0$ and $r_{\Omega}$. Since $u$ solves the $R_2$ Regularity problem in $\Omega$ with boundary values $f$, then if $t$ is the average of $u$ on $\partial\Omega$,
\begin{align}\label{eq:GradientLater}
\lambda\int_{\Omega}|\nabla u|^2&\leq\int_{\Omega}A\nabla u\nabla u=\int_{\partial\Omega}(u-t)\partial_{\nu}u\leq \|\nabla u\|_{L^2(\partial\Omega)}\|u-t\|_{L^2(\partial\Omega)}\leq C\|(\nabla u)^*\|_{L^2(\partial\Omega)}^2,
\end{align}
from Poincar{\'e}'s inequality on $\partial\Omega$. Combining \eqref{eq:GradientLater} with \eqref{eq:nabla^2BoundInside1}, we then obtain that
\begin{equation}\label{eq:nabla^2BoundInside}
|\nabla^2u(z)|\leq C\left(\fint_{2B_z}|\nabla u|^2\right)^{1/2}\leq C\int_{\partial\Omega}((\nabla u)^*)^2\,d\sigma,
\end{equation}
for any $z\in\Omega^{r_{\Omega}}$. If now $z\in\Omega_{r_{\Omega}}$, then
\begin{equation}\label{eq:nabla^2BoundOutside}
|\nabla^2u(z)|\leq\frac{C}{\delta(z)}\left(\fint_{4B_z}|\nabla u|^2\right)^{1/2}.
\end{equation}
Now, for $\beta\in\left(0,\frac{1}{2}\right)$, we compute
\begin{align}\nonumber
\int_{\Omega}\left(\delta(z)^{1-\beta}|\nabla^2u(z)|\right)^2\,dz&=\int_{\Omega^{r_{\Omega}}}\delta(z)^{2-2\beta}|\nabla^2u(z)|^2\,dz+\int_{\Omega_{r_{\Omega}}}\delta(z)^{2-2\beta}|\nabla^2u(z)|^2\,dz\\
\label{eq:delta1betabound}
&\leq Cr_{\Omega}^{2-2\beta}|\Omega|\int_{\partial\Omega}((\nabla u)^*)^2+\int_{\Omega_{r_{\Omega}}}\delta(z)^{2-2\beta}|\nabla^2u(z)|^2\,dz,
\end{align}
where we used \eqref{eq:nabla^2BoundInside} in the last inequality. To estimate the last term in \eqref{eq:delta1betabound}, suppose that $z\in B_{2r_{\Omega}}(q_i)\cap\Omega$ from Definition~\ref{LipDom} and also $z=(z',z_n)$ in the coordinate system for $B_{C_Mr_{\Omega}}(q_i)$. Setting $q_z=(z',\psi_i(z'))\in\partial\Omega$, Lemma~\ref{LiesAbove} shows that $B_{\delta(z)/2}(z)\subseteq\Gamma(q_z)$, therefore \eqref{eq:nabla^2BoundOutside} shows that $|\nabla^2u(z)|\leq\frac{C}{\delta(z)}(\nabla u)^*(q_z)$. Hence, in the coordinate system for $B_{C_Mr_{\Omega}}(q_i)$,
\[
\int_{B_{2r_{\Omega}}(q_i)\cap\Omega_{r_{\Omega}}}\delta(z)^{2-2\beta}|\nabla^2u(z)|^2\,dz\leq \frac{Cr_{\Omega}^{1-2\beta}}{1-2\beta}\int_{\Delta_{2r_{\Omega}}(q_i)}((\nabla u)^*)^2\,d\sigma.
\]
Adding the previous estimates for $i=1,\dots N$, plugging in \eqref{eq:delta1betabound} and using also the estimate $\|(\nabla u)^*\|_{L^2(\partial\Omega)}\leq C\|\nabla_Tf\|_{L^2(\partial\Omega)}$ completes the proof.
\end{proof}

As a corollary, we obtain the next estimate on the derivative of a solution to the Regularity problem.

\begin{lemma}\label{ImprovedIntegrability}
Under the same assumptions as in Lemma~\ref{ImprovedDerivativeRegularity}, then for any $p_1\in\left(1,\frac{2n}{n-1}\right)$, $\|\nabla u\|_{L^{p_1}(\Omega)}\leq C\|\nabla_Tf\|_{L^2(\partial\Omega)}$, where $C$ depends on $n,p_1,\lambda,\alpha,\tau,\|A\|_{\infty},C_0,\alpha_0$ and the Lipschitz character of $\Omega$.
\end{lemma}
\begin{proof}
Using Lemma~\ref{ImprovedDerivativeRegularity} and \eqref{eq:GradientLater}, we obtain that
\begin{equation}\label{eq:ForBesov}
\left\|\delta^{1-\beta}|\nabla^2u|+|\nabla u|\right\|_{L^2(\Omega)}\leq C\|\nabla_T f\|_{L^2(\partial\Omega)}.
\end{equation}
for any $\beta\in\left(0,\frac{1}{2}\right)$. We now use the implication (b)$\Rightarrow$(a) in Theorem 4.1 of \cite{JerisonKenigInhomogeneous} for the partials $\partial_iu$ and $k=0$ (this theorem is stated for harmonic functions, but the proof of this implication does not use this fact). Then, combining with \eqref{eq:ForBesov}, we obtain that
\[
\|\nabla u\|_{B_{\beta}^2(\Omega)}\leq C\|\nabla_Tf\|_{L^2(\partial\Omega)},
\]
where $B_{\beta}^2(\Omega)$ denotes the space of restrictions of functions of $B_{\beta}^2$ in $\Omega$ (page 173 in \cite{JerisonKenigInhomogeneous}), and $B_{\beta}^2$ is defined in page 172 of \cite{JerisonKenigInhomogeneous}.

From Proposition 2.17 (b) in \cite{JerisonKenigInhomogeneous}, for any $i=1,\dots n$, there exists $g_i\in B_{\beta}^2$ such that $\partial_iu=g_i$ in $\Omega$ and $\|g_i\|_{B_{\beta}^2}\leq C\|\partial_iu\|_{B_{\beta}^2(\Omega)}$. Note now that $B_{\beta}^2$ coincides with $\Lambda_{\beta}^{2,2}$, where the latter space is defined on page 7 in \cite{JonssonWallin}. From the last theorem on page 8 in \cite{JonssonWallin}, $\Lambda_{\beta}^{2,2}=L_{\beta}^2(\bR^n)$, where $L_{\beta}^2(\bR^n)$ is defined on page 6 in \cite{JonssonWallin}. Then, from the theorem on the same page (for $\alpha=\beta$, $\alpha_1=0$, $p=2$ and $p_1=\frac{2n}{n-2\beta}$) we obtain that $L_{\beta}^2(\bR^n)\subseteq L_0^{p_1}(\bR^n)=L^{p_1}(\bR^n)$. So, for any $i=1,\dots n$,
\[
\|\partial_iu\|_{L^{p_1}(\Omega)}\leq \|g_i\|_{L^{p_1}(\bR^n)}\leq C\|g_i\|_{L_{\beta}^2(\bR^n)}\leq C\|g_i\|_{B_{\beta}^2}\leq C\|\nabla u\|_{B_{\beta}^2(\Omega)}\leq C\|\nabla_T f\|_{L^2(\partial\Omega)},
\]
which completes the proof.

\end{proof}

\section{Estimates on Green's function}

\subsection{Main properties}
We now develop the main properties of Green's function that we will need in the following, where Green's function is defined in Definition 5.1 in \cite{KimSak}. We begin with the following proposition, in which we have the pointwise bounds for Green's function and its derivative.
\begin{prop}\label{GreenBounds}
Let $\Omega\subseteq\mathbb R^n$ be a domain with $|\Omega|<\infty$. Assume that $A$ is bounded and elliptic with ellipticity $\lambda$, and $b,c,d\in L^p(\Omega)$ for some $p>n$, with either $d\geq\dive b$, or $d\geq\dive c$ in the sense of distributions. Then Green's function $G(x,y)$ for the operator $\mathcal{L}u=-\dive(A\nabla u+bu)+c\nabla u+du$ in $\Omega$ exists and it is nonnegative. Moreover, if $G_y(x)=G(x,y)$,
\begin{equation}\label{eq:GreenMain}
G_y(x)\leq C|x-y|^{2-n},\quad \|\nabla G_y\|_{L^{\frac{n}{n-1},\infty}(\Omega)}\leq C,
\end{equation}
for all $x,y\in\Omega$, where $C$ depends on $n,p$, $\lambda$, $\|A\|_{\infty}$, $\|b-c\|_p$ and $|\Omega|$. In particular, $\nabla G_y\in L^q$ for any $q\in\left[1,\frac{n}{n-1}\right)$. If, in addition, $\Omega$ is a ball $B_{\rho}$ of radius $\rho$, $A\in\M_{B_{\rho}}(\lambda,\alpha,\tau)$ and $b\in\C_{B_{\rho}}(\alpha,\tau)$, then for all $x,y\in B_{\rho}$ with $x\neq y$, we have that
\begin{equation}\label{eq:GreenDerivative}
|\nabla_xG(x,y)|\leq C|x-y|^{1-n},
\end{equation}
where $C$ depends on $n,p,\lambda,\alpha,\tau$, $\|A\|_{\infty},\|b\|_{\infty}$,$\|c\|_p,\|d\|_p$ and $\rho$.
\end{prop}
\begin{proof}
The first and second estimates are a combination of Theorems 6.10, 6.12 and 7.2 in \cite{KimSak}. To show the third estimate, we follow a procedure similar to the proof of Theorem 8.1 in \cite{KimSak}.
\end{proof}

We will also need the following representation formula.
\begin{lemma}\label{Representation}
Let $\Omega$ be a domain with $|\Omega|<\infty$, and let $A$ be bounded and elliptic, $b,c,d\in L^p(\Omega)$ for some $p>n$. Assume also that $d\geq\dive b$, or $d\geq\dive c$. If $G(\cdot,y)=G_y(\cdot)$ is Green's function for the equation $-\dive(A\nabla u+bu)+c\nabla u+du=0$ in $\Omega$, then for any $\phi\in C_c^{\infty}(\Omega)$,
\[
\int_{\Omega}A\nabla G_y\nabla\phi+b\nabla\phi\cdot G_y+c\nabla G_y\cdot\phi+dG_y\phi=\phi(y).
\]
\end{lemma}
\begin{proof}
Assume first that $d\geq\dive b$. Suppose also that $A,b,c,d$ are smooth in $\overline{\Omega}$. Fix $y\in\Omega$ and set $q\in\left(1,\frac{n}{n-1}\right)$ to be the conjugate exponent to $p$. Then, from Theorem 6.12 in \cite{KimSak}, $G_y\in W_0^{1,q}(\Omega)$. Let $\phi\in C_c^{\infty}(\Omega)$, and set $\psi=\mathcal{L}^t\phi=-\dive(A^t\nabla\phi+c\phi)+b\nabla\phi+d\phi\in C_c^{\infty}(\Omega)$. Then, from Theorem 6.12 in \cite{KimSak}, we have that
\begin{align*}
\phi(y)&=\int_{\Omega}G(x,y)\psi(x)\,dx=\int_{\Omega}G_y\left(-\dive(A^t\nabla\phi+c\phi)+b\nabla\phi+d\phi\right)\\
&=\int_{\Omega}A\nabla G_y\nabla\phi+b\nabla\phi\cdot G_y+c\nabla G_y\cdot\phi+dG_y\phi,
\end{align*}
after integrating by parts. So the identity holds when $A,b,c,d$ are smooth in $\overline{\Omega}$.

In the general case, consider mollifications $A_m,b_m,c_m,d_m$ as in in Lemma 6.9 in \cite{KimSak}, and set $\Omega_m=\{x\in\Omega\big{|}\delta(x)>1/m\}$. Set $\mathcal{L}_m=-\dive(A_m\nabla u+b_mu)+c_m\nabla u+d_mu$, and let $G_m$ be Green's function for $\mathcal{L}_m$ in $\Omega_m$. From Lemma 6.9 in \cite{KimSak}, if $d\geq\dive b$, then $d_m\geq\dive b_m$ in $\Omega_m$. Also, from the previous proof, for $m$ large such that $y\in\Omega_m$,
\begin{equation}\label{eq:Limmphi}
\phi(y)=\int_{\Omega}A_m\nabla G_y^m\nabla\phi+b_m\nabla\phi\cdot G_y^m+c_m\nabla G_y^m\cdot\phi+d_mG_y^m\phi.
\end{equation}
Following the proof of Theorem 6.10 in \cite{KimSak}, we can find a subsequence $(G_y^{k_m})$ of $(G_y^m)$ that converges to $G_y$ weakly in $W_0^{1,q}(\Omega)$.
Since $A_m\to A$, $b_m\to b$, $c_m\to c$, $d_m\to d$ strongly in $L^p(\Omega)$, taking the limit in \eqref{eq:Limmphi} as $m\to\infty$ along the subsequence $k_m$ completes the proof in the case $d\geq\dive b$.

The case $d\geq\dive c$ is treated similarly, using Theorems 6.10 and 7.2 in \cite{KimSak}.
\end{proof}

\subsection{Estimates on differences}
In this section we will show pointwise estimates for differences of Green's functions when we perturb the coefficients of the operators. The first lemma that we will need is the following.

\begin{lemma}\label{LorentzEstimate}
Let $r>0$ and $p>n$. Let also $p'$ be the conjugate exponent to $p$, and $B_r$ be the ball with radius $r$, centered at $0$. Set also $f_1(x)=|x|^{1-n}$, and $f_2(x)=|x|^{2-n}$. Then $f_1\in L^{n,1}(\mathbb R^n\setminus B_r)$ and $f_2\in L^{\frac{pn}{p-n},p'}(\mathbb R^n\setminus B_r)$, with
\[
\|f_1\|_{L^{n,1}(\mathbb R^n\setminus B_r)}\leq C_nr^{2-n},\quad\|f_2\|_{L^{\frac{pn}{p-n},p'}(\mathbb R^n\setminus B_r)}\leq C_{n,p}r^{3-n-\frac{n}{p}}.
\]
\end{lemma}
\begin{proof}
Let $\lambda_i(s)$ be the distribution function of $f_i$ for $i=1,2$. For $f_1$, note that $f_1(x)\leq r^{1-n}$, so $\lambda_1(s)=0$ for $s\geq r^{1-n}$. Moreover, for $s<r^{1-n}$, $|f_1(x)|>s$ if and only if $|x|<s^{\frac{1}{1-n}}$, therefore $\lambda_1(s)\leq c_ns^{\frac{n}{1-n}}$. Hence, using Proposition 1.4.9 in \cite{Grafakos}, we obtain that
\begin{align*}
\|f_1\|_{L^{n,1}(\mathbb R^n\setminus B_r)}&=C_n\int_0^{r^{1-n}}\lambda_1(s)^{1/n}\,ds\leq C_n\int_0^{r^{1-n}}s^{\frac{1}{1-n}}\,ds=C_nr^{2-n}.
\end{align*}
For $f_2$, note that $f_2(x)\leq r^{2-n}$, so $\lambda(s)=0$ for $s\geq r^{2-n}$. Moreover, for $s<r^{2-n}$, $|f_2(x)|>s$ if and only if $|x|<s^{\frac{1}{2-n}}$, therefore $\lambda_2(s)\leq c_ns^{\frac{n}{2-n}}$. Then, from Proposition 1.4.9 in \cite{Grafakos} we obtain
\begin{align*}
\|f_2\|_{L^{\frac{pn}{p-n},p'}(\mathbb R^n\setminus B_r)}&=C_{p,n}\left(\int_0^{r^{2-n}}\left(\lambda_2(s)^{\frac{p-n}{pn}}s\right)^{p'}\,\frac{ds}{s}\right)^{1/p'}\leq C_{p,n}\left(\int_0^{r^{2-n}}\left(s^{\frac{p-n}{(2-n)p}+1}\right)^{p'}\,\frac{ds}{s}\right)^{1/p'}\\
&=C_{p,n}r^{(2-n)\left(\frac{p-n}{(2-n)p}+1\right)}=C_{p,n}r^{\frac{p-n}{p}+2-n},
\end{align*}
which completes the proof.
\end{proof}

We will also need the following lemma, which is Lemma 5.18 in \cite{thesis}.

\begin{lemma}\label{ToBoundDif}
Let $q_0\geq 1$, and consider $p_1,p_2$ with $p_1,p_2>\frac{n(q_0-1)}{q_0}$ and $p_1+p_2+\frac{n}{q_0}<2n$.	Then, for every $x,y\in\bR^n$ with $x\neq y$,
\[
\int_{\bR^n}|x-z|^{q_0(p_1-n)}|y-z|^{q_0(p_2-n)}\,dz\leq C_n|x-y|^{q_0(p_1+p_2-2n)+n}.
\]
\end{lemma}

We now show pointwise estimates on differences of Green's functions.

\begin{lemma}\label{PointwiseGreen}
Let $B=B_{10\rho}\subseteq\bR^n$ be a ball of radius $10\rho$ for $\rho<\frac{1}{16}$, and let $A_i\in\M_B(\lambda,\alpha,\tau)$, $b_i\in\C_B(\alpha,\tau),c_i\in L^p(B)$ for some $p>n$, and $d_i\in L^p(B)$, for $i=1,2$. Assume that $d_i\geq \dive c_i$ for $i=1,2$ or $d_i\geq \dive b_i$ for $i=1,2$, in the sense of distributions, and set $\mathcal{L}_iu=-\dive(A_i\nabla u+b_iu)+c_i\nabla u+d_iu$.  If $G_i$ is Green's function for $\mathcal{L}_i$ in $B$, then for all $x,y\in B_{9\rho}$,
\begin{multline*}
|G_2(x,y)-G_1(x,y)|\leq C\|A_1-A_2\|_{\infty}|x-y|^{2-n}\\
+ C\left(\|b_1-b_2\|_{\infty}+\|c_1-c_2\|_p+\|d_1-d_2\|_p\right)|x-y|^{2-n+\delta_{n,p}},
\end{multline*}
where $\delta_{n,p}=1-\frac{n}{p}>0$, and $C$ depends on $n,p,\lambda,\alpha,\tau, \|A_i\|_{\infty},\|b_i\|_{\infty},\|c_i\|_p$, and $\|d_i\|_p$ for $i=1,2$.
\end{lemma}
\begin{proof}
Fix $x,y\in B_{9\rho}$, and set $G_1(\cdot)=G_1(\cdot, y)$, and $g_2(\cdot)=G_2^t(\cdot,x)$, where $G_2^t$ is Green's function for the adjoint operator $\mathcal{L}_2^t$. Then, from Lemma~\ref{Representation}, using Green's functions as test functions we obtain that
\begin{gather*}
\int_BA_1\nabla G_1\nabla g_2+b_1\nabla g_2\cdot G_1+c_1\nabla G_1\cdot g_2+d_1G_1g_2=g_2(y),\\
\int_BA^t_2\nabla g_2\nabla G_1+c_2\nabla G_1\cdot g_2+b_2\nabla g_2\cdot G_1+d_2g_2G_1=G_1(x).
\end{gather*}
Hence, after subtracting, we obtain that
\begin{equation}\label{eq:FormulaforDifference}
g_2(y)-G_1(x)=\int_B\tilde{A}\nabla G_1\nabla g_2+\int_B\tilde{b}\nabla g_2\cdot G_1+\int_B\tilde{c}\nabla G_1\cdot g_2+\int_B\tilde{d}G_1g_2=I_1+I_2+I_3+I_4,
\end{equation}
where $\tilde{A}=A_1-A_2$, $\tilde{b}=b_1-b_2$, $\tilde{c}=c_1-c_2$, and $\tilde{d}=d_1-d_2$.

Let $r=|x-y|/36$. Then $2r\leq(|x|+|y|)/18<\rho$, therefore $B_{2r}(y)$ is compactly supported in $B$. Moreover, $g_2$ solves the equation $-\dive(A^t\nabla u+cu)+b\nabla u+du=0$ in $B_{2r}(y)$, hence Proposition~\ref{LpReg} shows that
\begin{equation}\label{eq:CloseBoundNabla}
\int_{B_r(y)}|\nabla g_2|^p\leq\frac{C}{r^p}\int_{B_{2r}(y)}|g_2|^p\leq \frac{C}{r^p}\int_{B_{2r}(y)}|z-x|^{p(2-n)}\,dz\leq Cr^{p-pn+n},
\end{equation}
since $|z-x|>34r$ for every $z\in B_{2r}(y)$. Now, to bound $I_1$, we estimate
\begin{equation}\label{eq:I_1}
|I_1|\leq\|\tilde{A}\|_{\infty}\left(\int_{B_r(y)}|\nabla G_1||\nabla g_2|+\int_{B\setminus B_r(y)}|\nabla G_1||\nabla g_2|\right)=\|\tilde{A}\|_{\infty}\left(I_{11}+I_{12}\right).
\end{equation}
To bound $I_{11}$, let $p'$ be the conjugate exponent to $p$. We then use \eqref{eq:GreenDerivative} and \eqref{eq:CloseBoundNabla} to obtain that
\begin{align}\nonumber
I_{11}&\leq C\int_{B_r(y)}|z-y|^{1-n}|\nabla g_2(z)|\,dz\leq C\left(\int_{B_r(y)}|z-y|^{p'(1-n)}\,dz\right)^{1/p'}\left(\int_{B_r(y)}|\nabla g_2|^p\right)^{1/p}\\
\label{eq:I_5}
&\leq C\left(\int_{B_r(y)}|z-y|^{p'(1-n)}\,dz\right)^{1/p'}r^{1-n+\frac{n}{p}}\leq Cr^{2-n},
\end{align}
where the last integral is finite, since $p'<\frac{n}{n-1}$. For $I_{12}$, we set $f_1(z)=|z|^{1-n}$ and use H{\"o}lder's inequality for Lorentz norms (from \cite{Grafakos}, Theorem 1.4.17 (v)) to estimate
\begin{equation}\label{eq:I_6}
I_{12}\leq C\|\nabla G_1\|_{L^{n,1}(B\setminus B_r(y))}\|\nabla g_2\|_{L^{\frac{n}{n-1},\infty}(B\setminus B_r(y))}\leq C\|f_1\|_{L^{n,1}(\mathbb R^n\setminus B_r(0))}\leq Cr^{2-n},
\end{equation}
from \eqref{eq:GreenMain} and Lemma~\ref{LorentzEstimate}. Adding \eqref{eq:I_5} with \eqref{eq:I_6} and substituting in \eqref{eq:I_1}, we obtain
\begin{equation}\label{eq:I_1Est}
|I_1|\leq C\|\tilde{A}\|_{\infty}r^{2-n}.
\end{equation}
To estimate $I_2$, we write
\begin{equation}\label{eq:I_2}
|I_2|\leq\int_{B_r(y)}|\tilde{b}\nabla g_2\cdot G_1|+\int_{B\setminus B_r(y)}|\tilde{b}\nabla g_2\cdot G_1|=I_{21}+I_{22}.
\end{equation}
For $I_{21}$, let $q'>1$ be such that $\frac{2}{p}+\frac{1}{q'}=1$. Then $q'(2-n)>-n$, and using H{\"o}lder's inequality, \eqref{eq:GreenDerivative} and \eqref{eq:CloseBoundNabla}, we estimate
\begin{align}\nonumber
I_{21}&\leq\|\tilde{b}\|_p\left(\int_{B_r(y)}|\nabla g_2|^p\right)^{1/p}\left(\int_{B_r(y)}|G_1|^{q'}\right)^{1/q'}\leq C\|\tilde{b}\|_pr^{1-n+\frac{n}{p}}\left(\int_{B_r(y)}|z-y|^{q'(2-n)}\right)^{1/q'}\\
\label{eq:I_7}
&\leq C\|\tilde{b}\|_pr^{1-n+\frac{n}{p}}r^{2-n+\frac{n}{q'}}=C\|\tilde{b}\|_pr^{2-n+\delta_{n,p}}.
\end{align}
For $I_{22}$, set $s=\frac{n}{(n-1)p'}>1$. If $s'$ is the conjugate exponent to $s$, then we use H{\"o}lder's inequality and H{\"o}lder's inequality for Lorentz norms (from \cite{Grafakos}, Theorem 1.4.17 (v)) to estimate
\begin{align*}
I_{22}&\leq\|\tilde{b}\|_p\left(\int_{B\setminus B_r(y)}|\nabla g_2|^{p'}|G_1|^{p'}\right)^{1/p'}\leq C\|\tilde{b}\|_p\left(\left\||\nabla g_2|^{p'}\right\|_{L^{s,\infty}(B\setminus B_r(y))}\left\||G_1|^{p'}\right\|_{L^{s',1}(B\setminus B_r(y))}\right)^{1/p'}\\
&=C\|\tilde{b}\|_p\|\nabla g_2\|_{L^{p's,\infty}(B\setminus B_r(y))}\|G_1\|_{L^{p's',p'}(B\setminus B_r(y))}\leq C\|\tilde{b}\|_p\|G_1\|_{L^{p's',p'}(B\setminus B_r(y))},
\end{align*}
where we used Remark 1.4.7 in \cite{Grafakos} for the first equality, $p's=\frac{n}{n-1}$ and \eqref{eq:GreenMain}. Note also that $p's'=\frac{pn}{p-n}$. Therefore, setting $f_2(z)=|z|^{2-n}$ and using \eqref{eq:GreenMain} and Lemma~\ref{LorentzEstimate}, we estimate
\begin{equation}\label{eq:I_8}
I_{22}\leq C\|\tilde{b}\|_p\|f_2\|_{L^{p's',p'}(\mathbb R^n\setminus B_r(0))}\leq C\|\tilde{b}\|_pr^{2-n+\delta_{n,p}}.
\end{equation}
Plugging \eqref{eq:I_7} and \eqref{eq:I_8} in \eqref{eq:I_2}, we obtain that
\begin{equation}\label{eq:I_2Est}
|I_2|\leq C\|\tilde{b}\|_pr^{2-n+\delta_{n,p}}
\end{equation}
To bound $I_3$, we follow a procedure identical to the bound for $I_2$ by interchanging the roles of $G_1$ and $g_2$. This will show that
\begin{equation}\label{eq:I_3Est}
|I_3|\leq C\|\tilde{c}\|_pr^{2-n+\delta_{n,p}}.
\end{equation}
To bound $I_4$, let $q_0$ be the conjugate exponent to $\frac{pn}{p+n}$, and note that
\[
4-2n+\frac{n}{q_0}=4-n\left(1-\frac{p+n}{pn}\right)=3-n-\frac{n}{p}<0,
\]
therefore Lemma~\ref{ToBoundDif} for $p_1=p_2=2$ and $q_0$ is applicable. Hence, from \eqref{eq:GreenMain},
\begin{equation}\label{eq:I_4Est}
|I_4|\leq C\|\tilde{d}\|_{\frac{pn}{p+n}}\left(\int_{\mathbb R^n}|z-y|^{q_0(2-n)}|z-x|^{q_0(2-n)}\,dz\right)^{1/q_0}\leq C\|\tilde{d}\|_{\frac{pn}{p+n}}|x-y|^{3-n-\frac{n}{p}}.
\end{equation}
Adding \eqref{eq:I_1Est}, \eqref{eq:I_2Est}, \eqref{eq:I_3Est}, \eqref{eq:I_4Est}, substituting in \eqref{eq:FormulaforDifference}, and using that $\|\tilde{b}\|_p\leq C\|\tilde{b}\|_{\infty}$ and $\|\tilde{d}\|_{\frac{pn}{p+n}}\leq C\|\tilde{d}\|_p$ completes the proof.
\end{proof}

Under the setting of Lemma~\ref{PointwiseGreen}, we can show estimates on differences of derivatives of Green's functions.

\begin{lemma}\label{PointwiseGreenGradient}
Under the same assumptions as in Lemma~\ref{PointwiseGreen}, for any $x,y\in B_{8\rho}$,
\begin{multline*}
|\nabla_xG_2(x,y)-\nabla_xG_1(x,y)|\leq C\|A_1-A_2\|_{C^{\alpha}}|x-y|^{1-n}\\
+ C\left(\|b_1-b_2\|_{C^{\alpha}}+\|c_1-c_2\|_p+\|d_1-d_2\|_p\right)|x-y|^{1-n+\delta_{n,p}},
\end{multline*}
where $\delta_{n,p}=1-\frac{n}{p}>0$ and $C$ depends on the same quantities as in Lemma~\ref{PointwiseGreen}.
\end{lemma}
\begin{proof}
Set $\tilde{A}=A_1-A_2$, and similarly for $b,c,d$. Fix $x,y\in B_{8\rho}$, and set $r=|x-y|/32$ and $u(z)=G_2(z,y)-G_1(z,y)$. Since $4r\leq(|x|+|y|)/8<2\rho$, $B_{4r}(x)$ is compactly supported in $B$. Set also $u_0(z)=G_2(z,y)$ in $B_{4r}(x)$. Then $u$ is a solution to the equation $-\dive(A_1\nabla u+b_1u)+c_1\nabla u+d_1u=-\dive f+g$ in $B_{4r}(x)$, where 
\[
f=\tilde{A}\nabla u_0+\tilde{b} u_0,\qquad g=\tilde{c}\nabla u_0+\tilde{d}u_0.
\]
Then, by Proposition~\ref{CAlphaReg} we obtain that
\begin{equation}\label{eq:eq:ToBoundNablaG}
\|\nabla u\|_{L^{\infty}(B_r(x))}\leq\frac{C}{r}\|u\|_{L^{\infty}(B_{2r}(x))}\\
+C\|f\|_{L^{\infty}(B_{2r}(x))}+Cr^{\beta}\|f\|_{C^{0,\beta}(B_{2r}(x))}+Cr\left(\fint_{B_{2r}(x)}|g|^p\right)^{1/p}. 
\end{equation}
Note now that $B_{2r}(x)\subseteq B_{9\rho}$, therefore we can apply Lemma~\ref{PointwiseGreen} and obtain that
\begin{equation*}
\|u\|_{L^{\infty}(B_{2r}(x))}\leq C\|\tilde{A}\|_{\infty}r^{2-n}\\
+ C\left(\|\tilde{b}\|_{\infty}+\|\tilde{c}\|_p+\|\tilde{d}\|_p\right)r^{2-n+\delta_{n,p}}.
\end{equation*}
Also, by Proposition~\ref{CAlphaReg} we obtain
\begin{equation}\label{eq:CBeta}
\frac{1}{r}\|u_0\|_{L^{\infty}(B_{2r}(x))}+\|\nabla u_0\|_{L^{\infty}(B_{2r}(x))}+r^{\beta}\|\nabla u_0\|_{C^{0,\beta}(B_{2r}(x))}\leq\frac{C}{r}\left(\fint_{B_{4r}(x)}|G(z,y)|^2\,dz\right)^{1/2}\leq Cr^{1-n},
\end{equation}
where we also used the pointwise bound in \eqref{eq:GreenMain}. Then, using \eqref{eq:CBeta},
\begin{align*}
\|\tilde{A}\nabla u_0\|_{C^{0,\beta}(B_{2r}(x))}\leq \|\tilde{A}\|_{C^{\beta}}\|\nabla u_0\|_{C^{\beta}(B_{2r}(x))}\leq \|\tilde{A}\|_{C^{\alpha}}\|\nabla u_0\|_{C^{\beta}(B_{2r}(x))}\leq C\|\tilde{A}\|_{C^{\alpha}}r^{1-n-\beta},
\end{align*}
since $\beta\leq\alpha$, and also $\|\tilde{A}\nabla u_0\|_{L^{\infty}(B_{2r}(x))}\leq C\|\tilde{A}\|_{C^{\alpha}}r^{1-n}$. Similarly, using \eqref{eq:CBeta},
\[
\|\tilde{b}u_0\|_{C^{0,\beta}(B_{2r}(x))}\leq C\|\tilde{b}\|_{C^{\alpha}}r^{2-n-\beta},\qquad \|\tilde{b}u_0\|_{L^{\infty}(B_{2r}(x))}\leq C\|\tilde{b}\|_{C^{\alpha}}r^{2-n}.
\]
Finally,
\begin{align*}
\left(\fint_{B_{2r}(x)}|g|^p\right)^{1/p}&\leq Cr^{-n/p}\left(\|\tilde{c}\|_{L^p(B_{2r})}+\|\tilde{d}\|_{L^p(B_{2r})}\right)\left(\|\nabla u_0\|_{L^{\infty}(B_{2r}(x))}+\|u_0\|_{L^{\infty}(B_{2r}(x))}\right)\\
&\leq C\left(\|\tilde{c}\|_p+\|\tilde{d}\|_p\right)r^{1-n-n/p.}
\end{align*}
We then complete the proof by plugging all the estimates above in \eqref{eq:eq:ToBoundNablaG}.
\end{proof}

\subsection{Comparing with the fundamental solution}
We will now compare differences of Green's functions for the full equation and the fundamental solution when the lower order coefficients vanish. We remark that similar estimates appear in \cite{XuZhaoZhou}, but the authors compare with the fundamental solution for fixed coefficients (Lemma 4.6).

Assume that $B=B_{10\rho}$ is a ball of radius $10\rho$ for $\rho<\frac{1}{16}$ and let $A\in\M_B(\lambda,\alpha,\tau)$. Since $\diam(B_{8\rho})<\frac{1}{2}$, we can mimic the proofs of Lemmas 2.1 and 2.2 in \cite{thesis} to construct $\tilde{A}$ that is $1$-periodic in $\bR^n$ (that is, $\tilde{A}$ satisfies (1.3) in \cite{KenigShen})such that
\begin{equation}\label{eq:ATilde}
\tilde{A}\in\M_{\bR^n}(\lambda,\alpha,C_n(\tau+\|A\|_{\infty})),\quad \tilde{A}=A\,\,\text{in}\,\,B_{8\rho}.
\end{equation}
From the argument after Lemma 2.1 in \cite{KenigShen}, we can construct the fundamental solution $\Gamma_{\tilde{A}}$ for the operator $-\dive(\tilde{A}\nabla u)$ in $\bR^n$.

Suppose now that $A\in\M_B(\lambda,\alpha,\tau)$, $b\in\C_B(\alpha,\tau)$ and $c,d\in L^p(B)$ for some $p>n$, with either $d\geq\dive b$ or $d\geq\dive c$, and set $G$ to be Green's function for the operator $\mathcal{L}u=-\dive(A\nabla u+bu)+c\nabla u+du$ in $B$. Then, for the operator $\mathcal{L}$, we set
\begin{equation}\label{eq:PiDef2}
\pi_{\mathcal{L}}(x,y)=G(x,y)-\Gamma_{\tilde{A}}(x,y),
\end{equation}
for $x,y\in B_{8\rho}$.

Under the same setting as above, let $g_A$ be Green's function for the operator $-\dive(A\nabla u)$ in $B_{7\rho}$. Note then that, for any $A_1,A_2\in\M_B(\lambda,\alpha,\tau)$, from \eqref{eq:FormulaforDifference},
\[
g_{A_2}(x,y)-g_{A_1}(x,y)=\int_{B_{7\rho}}(A_1(z)-A_2(z))\nabla_zg_{A_1}(z,y)\nabla_zg_{A_2}(x,z)\,dz.
\]
Since $g_{A_2}(x,z)=g_{A_2^t}(z,x)$, \eqref{eq:GreenDerivative} and Lemma~\ref{ToBoundDif} show that for all $x,y\in B_{7\rho}$,
\begin{equation}\label{eq:EDifference2}
|g_{A_2}(x,y)-g_{A_1}(x,y)|\leq C\|A_1-A_2\|_{\infty}|x-y|^{2-n},
\end{equation}
where $C$ depends on $n,\lambda,\alpha,\tau$,$\|A_1\|_{\infty}$ and $\|A_2\|_{\infty}$. Fix now $x\in B_{7\rho}$ and $y\in B_{6\rho}$, and define $u(z)=\nabla_xg_{A_2}(x,z)-\nabla_xg_{A_1}(x,z)$ and $u_0(z)=\nabla_xg_{A_2}(x,z)$. If $r=|x-y|/26$, then $2r=|x-y|/13<\rho$, therefore $B_{2r}(y)\subseteq B_{7\rho}$. Since $u_0$ solves the equation $-\dive(A_2^t\nabla u_0)=0$ in $B_{2r}(y)$, from \eqref{eq:GreenMain}, \eqref{eq:GreenDerivative} and Proposition~\ref{CAlphaReg} we obtain
\[
\|\nabla u_0\|_{L^{\infty}(B_r(y))}+r^{\alpha}\|\nabla u_0\|_{C^{0,\alpha}(B_r(y))}\leq Cr^{1-n},
\]
where $C$ depends on $n,\lambda,\alpha,\tau$ and $\|A_2\|_{\infty}$. Note now that $-\dive(A_1^t\nabla u)=-\dive((A_1^t-A_2^t)\nabla u_0)=-\dive f$ in $B_r(y)$, hence, from Proposition~\ref{CAlphaReg}, we obtain
\[
\|\nabla u\|_{L^{\infty}(B_{r/2}(y))}\leq\frac{C}{r}\|u\|_{L^{\infty}(B_r(y))}+C\|f\|_{L^{\infty}(B_r(y))}+Cr^{\alpha}\|f\|_{C^{0,\alpha}(B_r(y))}\leq C\|A_1-A_2\|_{C^{\alpha}}r^{1-n},
\]
where we also used \eqref{eq:EDifference2} for the second estimate. Therefore, for all $x\in B_{7\rho}$ and $y\in B_{6\rho}$,
\begin{equation}\label{eq:NablaEDifference2}
|\nabla_xg_2(x,y)-\nabla_xg_1(x,y)|\leq C\|A_1-A_2\|_{C^{\alpha}}|x-y|^{1-n}.
\end{equation}

We then have the next estimate.

\begin{lemma}\label{PointwiseDifDif}
Under the same assumptions as in Lemma~\ref{PointwiseGreen}, if $\pi_{\mathcal{L}_1},\pi_{\mathcal{L}_2}$ denote the differences in \eqref{eq:PiDef2} for $x,y\in B_{8\rho}$, then for any $x,y\in B_{6\rho}$,
\begin{equation*}
|\pi_{\mathcal{L}_1}(x,y)-\pi_{\mathcal{L}_2}(x,y)|\leq C\left(\|A_1-A_2\|_{C^{\alpha}}+\|b_1-b_2\|_{C^{\alpha}}+\|c_1-c_2\|_p+\|d_1-d_2\|_p\right)|x-y|^{2-n+\delta_{n,p}},
\end{equation*}
where $C$ depends on the same quantities as in Lemma~\ref{PointwiseGreen}.
\end{lemma}
\begin{proof}
Fix $x,y\in B_{6\rho}$. Let also $g_i=g_{A_i}$ for $i=1,2$ in $B_{7\rho}$, where $g_A$ is defined after \eqref{eq:PiDef2}, and define $\pi_i'(x,y)=G_i(x,y)-g_i(x,y)$, $\pi_i''(x,y)=g_i(x,y)-\Gamma_{\tilde{A}_i}(x,y)$; then,
\begin{equation*}
\pi_{\mathcal{L}_i}(x,y)=\pi_i'(x,y)+\pi_i''(x,y).
\end{equation*}
Set $g_i^x(z)=g_i(x,z)$ and $G_i^y(z)=G_i(y,z)$. Let $v_i$ be the solution to $-\dive(A_i\nabla v_i)=0$ in $B_{7\rho}$, with boundary values $G_i^y$ on $\partial(B_{7\rho})$. Then, an integration by parts argument shows that
\begin{equation}\label{eq:eq:v_i}
v_i(y)=-\int_{\partial(B_{7\rho})}\left<A_i\nabla g_i^x,\nu\right>G_i^y\,d\sigma.
\end{equation}
Using $w_i=G_i^y-v_i$ as a test function for $g_i^t$ (Green's function for $-\dive(A_i^t\nabla u)$ in $B_{7\rho}$) in Lemma~\ref{Representation}, we obtain that
\begin{equation}\label{eq:eq:w_i}
w_i(x)=G_i(x,y)-v_i(x)=\int_{B_{7\rho}}A_i^t\nabla g_i^x\nabla w_i=\int_{B_{7\rho}}A_i^t\nabla g_i^x\nabla G_i^y,
\end{equation}
where we also used that $v_i$ is a solution of $-\dive(A_i\nabla v_i)=0$ in $B_{7\rho}$. Moreover, extending $g_i^x$ by $0$ outside $B_{7\rho}$ and using it as a test function for $G_i$, we obtain that
\begin{equation}\label{eq:eq:g_i}
g_i(x,y)=\int_{B_{7\rho}}A_i\nabla G_i^y\nabla g_i^x+b_i\nabla g_i^x\cdot G_i^y+c_i\nabla G_i^y\cdot g_i^x+d_iG_i^yg_i^x.
\end{equation}
Subtracting \eqref{eq:eq:g_i} from \eqref{eq:eq:w_i} and using \eqref{eq:eq:v_i}, we then obtain that
\begin{equation*}
\pi_i'(x,y)=-\int_{B_{7\rho}}b_i\nabla g_i^x\cdot G_i^y+c_i\nabla G_i^y\cdot g_i^x+d_iG_i^yg_i^x-\int_{\partial(B_{7\rho})}\left<A_i\nabla g_i^x,\nu\right>G_i^y\,d\sigma.
\end{equation*}
Then,  $\pi_1'(x,y)-\pi_2'(x,y)=\sum_{i=1}^{12}I_i$, where
\begin{gather*}
I_1=\int_{B_{7\rho}}(b_2-b_1)\nabla g_2^x\cdot G_2^y,\quad I_2=\int_{B_{7\rho}}b_1(\nabla g_2^x-\nabla g_1^x)\cdot G_2^y,\quad I_3=\int_{B_{7\rho}}b_1\nabla g_1^x\cdot(G_2^y-G_1^y),\\
I_4=\int_{B_{7\rho}}(c_2-c_1)\nabla G_2^y\cdot g_2^x,\quad I_5=\int_{B_{7\rho}}c_1(\nabla G_2^y-\nabla G_1^y)\cdot g_2^x,\quad I_6=\int_{B_{7\rho}}c_1\nabla G_1^y\cdot(g_2^x-g_1^x),\\
I_7=\int_{B_{7\rho}}(d_2-d_1)G_2^yg_2^x,\quad I_8=\int_{B_{7\rho}}d_1(G_2^y-G_1^y)g_2^x,\quad I_9=\int_{B_{7\rho}}d_1G_1^y(g_2^x-g_1^x),
\end{gather*}
are the solid integral differences, and
\begin{equation*}
I_{10}=\int_{\partial(B_{7\rho})}\left<(A_2-A_1)\nabla g_2^x,\nu\right>G_2^y\,d\sigma,\quad I_{11}=\int_{\partial(B_{7\rho})}\left<A_1(\nabla g_2^x-\nabla g_1^x),\nu\right>G_2^y\,d\sigma,
\end{equation*}
and $\displaystyle I_{12}=\int_{\partial(B_{7\rho})}\left<A_1\nabla g_1^x,\nu\right>(G_2^y-G_1^y)\,d\sigma$ are the surface integral differences.

To treat $I_1$, note that $\nabla g_2^x(z)=\nabla_zg_2^t(z,x)$, where $g_2^t$ is Green's function for the operator $-\dive(A_2^t\nabla u)$ in $B_{7\rho}$, therefore \eqref{eq:GreenDerivative} shows that $|\nabla g_2^x(z)|\leq C|x-z|^{1-n}$. Therefore, using \eqref{eq:GreenMain} and Lemma~\ref{ToBoundDif}, and also $|z-y|^{2-n}=|z-y|^{1-n+\delta_{n,p}}|z-y|^{-n/p}\leq C_{\rho}|z-y|^{1-n+\delta_{n,p}}$,
\[
|I_1|\leq C\|b_2-b_1\|_{\infty}\int_{B_{7\rho}}|x-z|^{1-n}|z-y|^{1-n+\delta_{n,p}}\,dz\leq C\|b_2-b_1\|_{\infty}|x-y|^{2-n+\delta_{n,p}}.
\]
To bound $I_2$, using $g_2^x(z)=g_2^t(z,x)$ and \eqref{eq:NablaEDifference2} for $g_2^t$ (for $z\in B_{7\rho}$ and $x\in B_{6\rho}$) we obtain
\[
|I_2|\leq C\|b_1\|_{\infty}\|A_1-A_2\|_{C^{\alpha}}\int_{B_{7\rho}}|x-z|^{1-n}|z-y|^{1-n+\delta_{n,p}}\,dz\leq C\|A_1-A_2\|_{C^{\alpha}}|x-y|^{2-n+\delta_{n,p}},
\]
and we bound $I_3$ similarly, using Lemma~\ref{PointwiseGreen} instead of \eqref{eq:NablaEDifference2} to bound $G_2^y-G_1^y$.

To bound $I_4$ let $p'$ be the conjugate exponent to $p$. We then use H{\"o}lder's inequality, \eqref{eq:GreenDerivative} and \eqref{eq:GreenMain}, to obtain that
\[
|I_4|\leq C\|c_2-c_1\|_p\left(\int_{B_{7\rho}}|z-y|^{p'(2-n)}|x-z|^{p'(1-n)}\,dz\right)^{1/p'}\leq C\|c_2-c_1\|_p|x-y|^{3-n-\frac{n}{p}},
\]
and we similarly bound $I_7$. To bound $I_5$ and $I_8$ we use Lemmas~\ref{PointwiseGreenGradient} and \ref{PointwiseGreen} respectively and a similar procedure as in $I_4$, and for $I_6$ and $I_9$ we use \eqref{eq:EDifference2}.

For $I_{10}$, note that for $z\in\partial(B_{7\rho})$ and $x,y\in B_{6\rho}$, $|x-z|>\rho$ and $|y-z|>\rho$, therefore
\[
|I_{10}|\leq C\|A_1-A_2\|_{\infty}\int_{\partial(B_{7\rho})}|x-z|^{1-n}|z-y|^{2-n}\,d\sigma(z)\leq C\|A_1-A_2\|_{\infty}.
\]
For $I_{11}$ we use \eqref{eq:NablaEDifference2}, and for $I_{12}$ we use Lemma~\ref{PointwiseGreen}, to obtain that $I_{11}\leq C\|A_1-A_2\|_{C^{\alpha}}$ and $I_{12}\leq C\|A_1-A_2\|_{C^{\alpha}}$. This shows that
\begin{equation*}
|\pi_1'(x,y)-\pi_2'(x,y)|\leq C(\|A_1-A_2\|_{C^{\alpha}}+\|b_1-b_2\|_{C^{\alpha}}+\|c_1-c_2\|_p+\|d_1-d_2\|_p)|x-y|^{2-n+\delta_{n,p}}.
\end{equation*}
To bound $\pi_1''-\pi_2''$, let $B_x$ be a small ball centered at $x$. Note that from \eqref{eq:GreenDerivative} and (2.5) in \cite{KenigShen}, $w_i(z)=\pi_i''(z,y)$ is a $W^{1,p_0}(B_{7\rho})\cap C(\overline{B_{7\rho}}\setminus B_x))$ solution to the equation $-\dive(A_i\nabla w_i)=0$ in $B_{7\rho}$, where $p_0\in\left(1,\frac{n}{n-1}\right)$ is fixed, with boundary values $-\Gamma_{\tilde{A}_i}(z,y)$. Then, Theorem A1.1 in \cite{AnconaElliptic} shows that $w_i\in W_0^{1,2}(B_x)$, hence combining with Proposition~\ref{CAlphaReg}, we obtain that $w_i\in W^{1,2}(B_{7\rho})\cap C(\overline{B_{7\rho}})$. Therefore, an integration by parts argument shows that
\begin{equation*}
w_i(x)=\int_{\partial(B_{7\rho})}\left<A_i(z)\nabla_zg_i(x,z),\nu(z)\right>\Gamma_{\tilde{A}_i}(z,y)\,d\sigma(z).
\end{equation*}
We then bound $w_1(x)-w_2(x)$ as we bounded $I_{10},I_{11}$ and $I_{12}$, where instead for the estimates for $G_2^y$ and $G_2^y-G_1^y$ we use (2.5) and (2.20) in \cite{KenigShen}, and this completes the proof.
\end{proof}

Set $A_2=2A_1=2A$, $b_2=2b_1=2b$, $c_2=2c_1=2c$ and $d_2=2d_1=2d$ in Lemma~\ref{PointwiseDifDif}. Then $\mathcal{L}_2u=2\mathcal{L}_1u$, therefore $G_2(x,y)=2G_1(x,y)$. Considering $\tilde{A}_2=2\tilde{A}_1$ in \eqref{eq:ATilde} we obtain that $\Gamma_{\tilde{A}_2}(x,y)=2\Gamma_{\tilde{A}_1}(x,y)$, hence from \eqref{eq:PiDef2}, $\pi_2(x,y)=2\pi_1(x,y)$. Therefore, Lemma~\ref{PointwiseDifDif} shows that
\begin{equation}\label{eq:piPointwise}
|\pi_2(x,y)|=2|\pi_2(x,y)-\pi_1(x,y)|\leq C|x-y|^{2-n+\delta_{n,p}},
\end{equation}
for any $x,y\in B_{6\rho}$, where $C$ depends on $n,p,\lambda,\alpha,\tau$,$\|A_1\|_{\infty}$,$\|b_1\|_{\infty}$,$\|c_1\|_p$ and $\|d_1\|_p$.

Let now $x,y\in B_{5\rho}$ and let $r=|x-y|/20$. Then $2r\leq(|x|+|y|)/10<\rho$, therefore $B_{2r}(x)\subseteq B_{6\rho}$. Setting $\Gamma_2^y(z)=\Gamma_2(z,y)$, note that $\pi_2(z)=\pi_2(z,y)$ solves the equation
\[
-\dive(A_2\nabla\pi_2+b_2\pi_2)+c_2\nabla\pi_2+d_2\pi_2=\dive(b_2\Gamma_2^y)-c_2\nabla\Gamma_2^y-d_2\Gamma_2^y=-\dive f+g
\]
in $B_{2r}(x)$. Hence, from Proposition~\ref{CAlphaReg} and \eqref{eq:piPointwise}, for $\beta=\min\left\{\alpha,1-\frac{n}{p}\right\}$,
\begin{multline*}
\|\nabla\pi_2\|_{L^{\infty}(B_r(x))}+r^{\beta}\|\nabla\pi_2\|_{C^{0,\beta}(B_r(x))}\leq Cr^{1-n+\delta_{n,p}}\\
+C\|f\|_{L^{\infty}(B_{2r}(x))}+Cr^{\alpha}\|f\|_{C^{0,\alpha}(B_{2r}(x))}+Cr^{\delta_{n,p}}\|g\|_{L^p(B_{2r}(x))}.
\end{multline*}
Then, using (2.5) in \cite{KenigShen}, we obtain that
\begin{equation}\label{eq:piGradPointwise}
\|\nabla\pi_2\|_{L^{\infty}(B_r(x))}+r^{\beta}\|\nabla\pi_2\|_{C^{0,\beta}(B_r(x))}\leq C|x-y|^{1-n+\delta_{n,p}}.
\end{equation}
In particular, for any $x,y\in B_{5\rho}$,
\begin{equation}\label{eq:DifGamma}
|\nabla_xG(x,y)-\nabla_x\Gamma_{\tilde{A}}(x,y)|\leq C|x-y|^{1-n+\delta_{n,p}},
\end{equation}
where $G$ is Green's function for $-\dive(A\nabla u+bu)+c\nabla u+du$ in $B_{10\rho}$, and $\Gamma_{\tilde{A}}$ is the fundamental solution for $-\dive(\tilde{A}\nabla u)$ in $\bR^n$.

We now show the next bound on the gradient of differences.

\begin{lemma}\label{PointwiseDifDifOnGrad}
Under the same assumptions as in Lemma~\ref{PointwiseGreen}, if $\pi_{\mathcal{L}_1},\pi_{\mathcal{L}_2}$ denote the differences in \eqref{eq:PiDef2} for $x,y\in B_{8\rho}$, then for any $x,y\in B_{4\rho}$,
\begin{equation*}
|\nabla_x\pi_{\mathcal{L}_1}(x,y)-\nabla_x\pi_{\mathcal{L}_2}(x,y)|\leq C(\|A_1-A_2\|_{C^{\alpha}}+\|b_1-b_2\|_{C^{\alpha}}+\|c_1-c_2\|_p+\|d_1-d_2\|_p)|x-y|^{1-n+\delta_{n,p}},
\end{equation*}
where $C$ depends on $n,p,\lambda,\alpha,\tau$,$\|A_i\|_{\infty}$,$\|b_i\|_{\infty}$,$\|c_i\|_p$ and $\|d_i\|_p$ for $i=1,2$.
\end{lemma}
\begin{proof}
Fix $x,y\in B_{4\rho}$ and set $r=|x-y|/16$. Then $2r\leq(|x|+|y|)/8<\rho$, so $B_{2r}(x)\subseteq B_{5\rho}$. Let also $\Gamma_i(z)=\Gamma_{\tilde{A}_i}(z,y)$ for $i=1,2$. Then, $v(z)=\pi_1(z,y)-\pi_2(z,y)$ solves the equation $-\dive(A_1\nabla v+b_1v)+c_1\nabla v+d_1v=-\dive f+g$ in $B_{2r}(x)$, where
\begin{gather*}
f=(b_2-b_1)\Gamma_2+b_1(\Gamma_2-\Gamma_1)+(A_2-A_1)\nabla\pi_2+(b_2-b_1)\pi_2,\\
g=(c_2-c_1)\nabla\Gamma_2+c_1(\nabla\Gamma_2-\nabla \Gamma_1)+(d_2-d_1)\Gamma_2+d_1(\Gamma_2-\Gamma_1)+(c_2-c_1)\nabla\pi_2+(d_2-d_1)\pi_2
\end{gather*}
in $B_{2r}(x)$. Hence, from Proposition~\ref{CAlphaReg},  for $\beta=\min\left\{\alpha,1-\frac{n}{p}\right\}$, we obtain that
\begin{equation}\label{eq:eq:v}
\|\nabla v\|_{L^{\infty}(B_r(x))}\leq\frac{C}{r}\|v\|_{L^{\infty}(B_{2r}(x))}+C\|f\|_{L^{\infty}(B_{2r}(x)}+Cr^{\beta}\|f\|_{C^{0,\beta}(B_{2r}(x))}+Cr\left(\fint_{B_{2r}(x)}|g|^p\right)^{1/p}.
\end{equation}
To bound $\|v\|_{L^{\infty}(B_r(x))}$, we use Lemma~\ref{PointwiseDifDif}. Moreover, to bound the other norms in \eqref{eq:eq:v} we use \eqref{eq:piPointwise}, \eqref{eq:piGradPointwise} and also (2.5), (2.20) and (2.21) in \cite{KenigShen}, which completes the proof.
\end{proof}

\section{Layer Potentials}

\subsection{Singular Integrals}
Let $\Omega\subseteq\bR^n$ be a bounded Lipschitz domain. Without loss of generality assume that $0\in\Omega$. Set $\rho=\diam(\Omega)$, and let $B_{\Omega}=B_{10\rho}$ to be the ball centered at $0$, with radius $10\rho$. We will assume that $A\in\M_{B_{\Omega}}(\lambda,\alpha,\tau)$, $b\in\C_{B_{\Omega}}(\alpha,\tau)$, and $c,d\in L^p(B_{\Omega})$ with either $d\geq\dive b$ or $d\geq\dive c$ in the sense of distributions. We then set $\mathcal{L}$ to be the operator
\[
\mathcal{L}u=-\dive(A\nabla u+bu)+c\nabla u+du
\]
in $B_{\Omega}$. Set also $G$ to be Green's function for $\mathcal{L}$ in $B_{\Omega}$. Then, for $f\in L^2(\partial\Omega)$, we define
\[
\mathcal{S}_+f(x)=\int_{\partial\Omega}G(x,q')f(q')\,d\sigma(q'),
\]
for $x\in\Omega$. When $x\in B_{\Omega}\setminus\overline{\Omega}$, we define $\mathcal{S}_-f$ with the same formula. We also consider the single layer potential operator
\[
\mathcal{S}:L^2(\partial\Omega)\to L^2(\partial\Omega),\quad \mathcal{S}f(q)=\int_{\partial\Omega}G(q,q')f(q')\,d\sigma(q').
\]
The fact that $\mathcal{S}$ maps $L^2(\partial\Omega)$ to $L^2(\partial\Omega)$ follows from the pointwise bounds on $G$. We also consider the maximal truncation operators
\[
\mathcal{T}_1^*f(q)=\sup_{\delta>0}\left|\int_{|q-q'|>\delta}\nabla_T^qG(q,q')f(q')\,d\sigma(q')\right|,\quad \mathcal{T}_2^*f(q)=\sup_{\delta>0}\left|\int_{|q-q'|>\delta}\nabla_T^{q'}G(q',q)f(q')\,d\sigma(q')\right|,
\]
where $\nabla_T^q$ denotes the tangential derivative with respect to $q$.

To show that $\mathcal{S}$ maps $L^2(\partial\Omega)$ to $W^{1,2}(\partial\Omega)$ and $\mathcal{T}_1^*$, $\mathcal{T}_2^*$ map $L^2(\partial\Omega)$ to $L^2(\partial\Omega)$ we will reduce to the cases considered in \cite{KenigShen}. For this reason, suppose that $\rho<\frac{1}{16}$.
\begin{prop}\label{LayerBoundedness}
Let $\Omega\subseteq\bR^n$ be a Lipschitz domain with $\diam(\Omega)<\frac{1}{16}$. Let $A\in\M_{B_{\Omega}}(\lambda,\alpha,\tau)$, $b\in\C_{B_{\Omega}}(\alpha,\tau)$, and $c,d\in L^p(B_{\Omega})$, for some $p>n$, with either $d\geq\dive b$ or $d\geq\dive c$. Then, for any $f\in L^2(\partial\Omega)$,
\begin{equation}\label{eq:SingBound}
\|\mathcal{S}f\|_{W^{1,2}(\partial\Omega)}+\|\mathcal{T}_1^*f\|_{L^2(\partial\Omega)}+\|\mathcal{T}_2^*f\|_{L^2(\partial\Omega)}\leq C\|f\|_{L^2(\partial\Omega)},
\end{equation}
with $C$ depending on $n,p,\lambda,\alpha,\tau$,$\|A\|_{\infty},\|b\|_{\infty}$,$\|c\|_p$, $\|d\|_p$ and the Lipschitz character of $\Omega$, and also
\begin{equation}\label{eq:FormTangDer}
\nabla_T\mathcal{S}f(q)=\lim_{\e\to 0}\int_{|q-q'|>\e}\nabla_T^qG(q,q')f(q')\,d\sigma(q'),
\end{equation}
where the limit exists both in the $L^2(\partial\Omega)$ sense and for almost every $q\in\partial\Omega$.
\end{prop}
\begin{proof}
Let $\tilde{A}$ be as in \eqref{eq:ATilde} such that $A=\tilde{A}$ in $B_{8\rho}$, where $\rho=\diam(\Omega)$, and set $\Gamma_{\tilde{A}}$ to be the fundamental solution for the operator $-\dive(\tilde{A}\nabla u)$ in $\bR^n$. Then, from \eqref{eq:DifGamma}, we obtain that
\begin{equation}\label{eq:DifGammaG}
|\nabla_xG(x,y)-\nabla_x\Gamma(x,y)|\leq C|x-y|^{1-n+\delta_{n,p}},
\end{equation}
for all $x,y\in B_{5\rho}$, where $\delta_{n,p}=1-\frac{n}{p}>0$. Now, from (4.17) and Theorem 3.1 in \cite{KenigShen}, the operator $\tilde{\mathcal{S}}f(q)=\int_{\partial\Omega}\Gamma_{\tilde{A}}(q,q')f(q')\,d\sigma(q')$ is bounded from $L^2(\partial\Omega)$ to $W^{1,2}(\partial\Omega)$, and the operators
\[
\tilde{\mathcal{T}}_1^*f(q)=\sup_{\delta>0}\left|\int_{|q-q'|>\delta}\nabla_T^q\Gamma_{\tilde{A}}(q,q')f(q')\,d\sigma(q')\right|,\quad \tilde{\mathcal{T}}_2^*f(q)=\sup_{\delta>0}\left|\int_{|q-q'|>\delta}\nabla_T^{q'}\Gamma_{\tilde{A}}(q',q)f(q')\,d\sigma(q')\right|
\]
are bounded from $L^2(\partial\Omega)$ to $L^2(\partial\Omega)$, with norms that depend on $n,\lambda,\alpha,\tau$,$\|A\|_{\infty}$ and the Lipschitz character of $\Omega$. Since $\partial\Omega\subseteq B_{5\rho},$ combining with \eqref{eq:DifGammaG} shows \eqref{eq:SingBound}.

To show \eqref{eq:FormTangDer}, we use an integration by parts argument on $\partial\Omega$ and the $L^2$ bound for $\mathcal{T}_1^*$ from \eqref{eq:SingBound}, which completes the proof.
\end{proof}

We now show a perturbation result for the norms of the single layer potentials.

\begin{lemma}\label{GeneralPerturbation}
Let $\Omega\subseteq\bR^n$ be a Lipschitz domain with $\diam(\Omega)<\frac{1}{16}$, $A_i\in\M_{B_{\Omega}}(\lambda,\alpha,\tau)$, $b_i\in\C_{B_{\Omega}}(\alpha,\tau)$, and $c_i,d_i\in L^p(B_{\Omega})$, with either $d_i\geq\dive b_i$ for $i=1,2$ or $d\geq\dive c_i$ for $i=1,2$. If $\mathcal{S}_i$ is the single layer potential on $\partial\Omega$ for the operator $\mathcal{L}_iu=-\dive(A_i\nabla u+b_iu)+c_i\nabla u+d_iu$, then
\[
\|\mathcal{S}_2-\mathcal{S}_1\|_{L^2(\partial\Omega)\to W^{1,2}(\partial\Omega)}\leq C\left(\|A_1-A_2\|_{C^{\alpha}}+\|b_1-b_2\|_{C^{\alpha}}+\|c_1-c_2\|_p+\|d_1-d_2\|_p\right),
\]
where $C$ depends on $n,p,\lambda,\alpha,\tau$,$\|A_i\|_{\infty},\|b_i\|_{\infty}$,$\|c_i\|_p$ and $\|d_i\|_p$, for $i=1,2$, and the Lipschitz character of $\Omega$.
\end{lemma}
\begin{proof}
To bound $\|\mathcal{S}_2f-\mathcal{S}_1f\|_{L^2(\partial\Omega)}$ for $f\in L^2(\partial\Omega)$, we use Lemma~\ref{PointwiseGreenGradient}. We now let $\tilde{A}_i$ be extensions of $A_i$ as in \eqref{eq:ATilde}, and let $\Gamma_{\tilde{A}_i}$ be the fundamental solutions for the operators $-\dive(\tilde{A}_i\nabla u)$ in $\bR^n$. Then, from \eqref{eq:PiDef2} and Lemma~\ref{PointwiseDifDifOnGrad}, we obtain that for any $x,y\in B_{4\rho}$,
\begin{align*}
|\nabla_xG_1(x,y)-\nabla_xG_2(x,y)-(\nabla_x\Gamma_1(x,y)-\nabla_x\Gamma_2(x,y))|&=|\nabla_x\pi_{\mathcal{L}_1}(x,y)-\nabla_x\pi_{\mathcal{L}_2}(x,y)|\\
&\leq C\theta|x-y|^{1-n+\delta_{n,p}},
\end{align*}
where $\theta=\|A_1-A_2\|_{C^{\alpha}}+\|b_1-b_2\|_{C^{\alpha}}+\|c_1-c_2\|_p+\|d_1-d_2\|_p$. Since $\partial\Omega\subseteq B_{4\rho}$, integrating over $\partial\Omega$ we obtain that the integral operator with kernel
\[
\nabla_xG_1(x,y)-\nabla_xG_2(x,y)-(\nabla_x\Gamma_1(x,y)-\nabla_x\Gamma_2(x,y))
\]
is bounded from $L^2(\partial\Omega)$ to $L^2(\partial\Omega)$, with norm bounded above by $C\theta$. From Theorem 3.4 in \cite{KenigShen}, the integral operator with kernel $\nabla_x\Gamma_1(x,y)-\nabla_x\Gamma_2(x,y)$ is bounded from $L^2(\partial\Omega)$ to $L^2(\partial\Omega)$, with norm bounded above by $C\|A_1-A_2\|_{C^{\alpha}}$, where $C$ depends on $n,\lambda,\alpha,\tau$,$\|A_1\|_{\infty}$,$\|A_2\|_{\infty}$ and the Lipschitz character of $\Omega$, and this completes the proof.
\end{proof}

In order to treat the Dirichlet problem, we will consider the adjoint of the single layer potential
\[
\mathcal{S}^*:W^{-1,2}(\partial\Omega)\to L^2(\partial\Omega).
\]
\begin{lemma}\label{SAdjointFormula}
Under the same assumptions as in Proposition~\ref{LayerBoundedness}, if $F=Rf\in W^{-1,2}(\partial\Omega)$ for $f\in W^{1,2}(\partial\Omega)$, where $R$ is defined in \eqref{eq:Emb1}, then
\[
\mathcal{S}^*F(q)=\int_{\partial\Omega}G(q',q)f(q')\,d\sigma(q')+\lim_{\e\to 0}\int_{|q-q'|>\e}\nabla_T^{q'}G(q',q)\nabla_Tf(q')\,d\sigma(q'),
\]
where the limit exists in the $L^2(\partial\Omega)$ sense, and almost everywhere on $\partial\Omega$. Also, if $F=Ef$ for $f\in W^{1,2}(\partial\Omega)$, where $E$ is defined in \eqref{eq:Emb2}, then
\[
\mathcal{S}^*F(q)=\int_{\partial\Omega}G(q',q)f(q')\,d\sigma(q').
\]
\end{lemma}
\begin{proof}
Let $h\in L^2(\partial\Omega)$. We use a procedure as in the proof of Proposition 9.6 in \cite{thesis}: from the definition of $R$, we compute
\[
\left<\mathcal{S}^*F,h\right>_{L^2,L^2}=\left<F,\mathcal{S}h\right>_{W^{-1,2},W^{1,2}}=\int_{\partial\Omega}f\cdot\mathcal{S}h\,d\sigma+\int_{\partial\Omega}\nabla_Tf\cdot\nabla_T\mathcal{S}h\,d\sigma.
\]
For the first integral, we compute
\begin{align*}
\int_{\partial\Omega}f\cdot\mathcal{S}h\,d\sigma&=\int_{\partial\Omega}\left(\int_{\partial\Omega}G(q',q)h(q)\,d\sigma(q)\right)f(q')\,d\sigma(q')\\
&=\int_{\partial\Omega}\left(\int_{\partial\Omega}G(q',q)f(q')\,d\sigma(q')\right)h(q)\,d\sigma(q),
\end{align*}
since the double integral is absolutely convergent, from \eqref{eq:GreenMain}. For the second integral, using \eqref{eq:FormTangDer},
\[
\int_{\partial\Omega}\nabla_Tf\cdot\nabla_T\mathcal{S}h\,d\sigma=\int_{\partial\Omega}\left(\lim_{\e\to 0}\int_{|q-q'|>\e}\nabla_T^{q'}G(q',q)\nabla_Tf(q')\,d\sigma(q')\right)h(q)\,d\sigma(q),
\]
using the dominated convergence theorem and \eqref{eq:SingBound} for $\mathcal{T}_2^*$, which completes the proof of the first identity. The proof of the second identity is similar.
\end{proof}

For any $x\in\Omega$, using \eqref{eq:q:CAlphaReg2} we obtain that Green's function $G_x(\cdot)=G(\cdot,x)$ is $C^1$ in a neighborhood of $\partial\Omega$, hence $G_x\in W^{1,2}(\partial\Omega)$. Therefore, for $F\in W^{-1,2}(\partial\Omega)$ and $x\in\Omega$, we can define
\[
\mathcal{S}^*_+F(x)=F(G_x).
\]
Note now that a procedure similar to the proof of Lemma~\ref{SAdjointFormula} shows that, if $F=Rf$,
\begin{equation}\label{eq:FormulaForR}
\mathcal{S}_+^*F(x)=\int_{\partial\Omega}G(q',x)f(q')\,d\sigma(q')+\int_{\partial\Omega}\nabla_T^{q'}G(q',x)\nabla_Tf(q')\,d\sigma(q'),
\end{equation}
and if $F=Ef$, then
\begin{equation*}
\mathcal{S}_+^*F(x)=\int_{\partial\Omega}G(q',x)f(q')\,d\sigma(q').
\end{equation*}

\subsection{Properties of Layer Potentials}
We now show that the layer potentials we have defined are solutions to the equations we are interested at. We first treat the single layer potential.

\begin{prop}\label{SingleLayerBounds}
Under the same assumptions as in Proposition~\ref{LayerBoundedness}, for every $f\in L^2(\partial\Omega)$, the function $\mathcal{S}_+f$ is a $W^{1,2}_{\loc}(\Omega)$ solution to $\mathcal{L}u=-\dive(A\nabla u+bu)+c\nabla u+du=0$ in $\Omega$, which converges nontangentially, almost everywhere to $\mathcal{S}f$ on $\partial\Omega$. Similarly, $\mathcal{S}_-f$ is a $W^{1,2}_{\loc}(B_{\Omega}\setminus \Omega)$ solution to $\mathcal{L}u=0$ in $B_{\Omega}\setminus \Omega$, which converges nontangentially, almost everywhere to $f\cdot\chi_{\partial\Omega}$ on $\partial(B_{\Omega}\setminus \Omega)$. Moreover,
\[
\|(\nabla\mathcal{S}_{\pm}f)^*\|_{L^2(\partial\Omega)}\leq C\|f\|_{L^2(\partial\Omega)},
\]
where $C$ depends on $n,p,\lambda,\alpha,\tau$,$\|A\|_{\infty},\|b\|_{\infty}$,$\|c\|_p,\|d\|_p$ and the Lipschitz character of $\Omega$.
\end{prop}
\begin{proof}
The fact that $\mathcal{S}_+f\in W^{1,2}_{{\rm loc}}(\Omega)$ follows from the pointwise bounds on $G$ and its derivative, from Proposition~\ref{GreenBounds}. Moreover, since $G(\cdot,q)$ is a solution of $\mathcal{L}u=0$ in $\Omega$, for every fixed $q\in\partial\Omega$, it follows that $\mathcal{S}_+f$ is a solution to $\mathcal{L}u=0$ in $\Omega$.

For the boundary values of $\mathcal{S}_+f$, we use the pointwise bounds on $G$ and a procedure as in Proposition 8.8 in \cite{thesis} (where instead of Lipschitz continuity of Green's function, we use that Green's function is H{\"o}lder continuous in our case). To show the bound on the nontangential maximal function of the gradient, we let $\Gamma$ be Green's function for the operator $-\dive(\tilde{A}\nabla u)$ in $\bR^n$, where $\tilde{A}$ is as in \eqref{eq:ATilde}, and let $\mathcal{S}_+^0$ be the single layer potential for the same operator in $\Omega$. Then, for $q\in\partial\Omega$ and $x\in\Gamma(q)$,
\begin{align*}
|\nabla\mathcal{S}_+f(x)|&\leq\int_{\partial\Omega}\left|\nabla_xG(x,q')-\nabla_x\Gamma(x,q')\right||f(q')|\,d\sigma(q')+|\nabla\mathcal{S}_+^0f(x)|\\
&\leq C\int_{\partial\Omega}|x-q'|^{1-n+\delta_{n,p}}|f(q')|\,d\sigma(q')+|\nabla\mathcal{S}_+^0f(x)|\\
&\leq C\int_{\partial\Omega}|q-q'|^{1-n+\delta_{n,p}}|f(q')|\,d\sigma(q')+|\nabla\mathcal{S}_+^0f(x)|,
\end{align*}
where we used \eqref{eq:DifGamma} and the fact that $x\in\Gamma(q)$. We then obtain that
\[
\left(\nabla\mathcal{S}_+f\right)^*(q)\leq C\int_{\partial\Omega}|q-q'|^{1-n+\delta_{n,p}}|f(q')|\,d\sigma(q')+\left(\nabla\mathcal{S}_+^0f\right)^*(q),
\]
for all $q\in\partial\Omega$. From Theorem 4.3 in \cite{KenigShen}, the operator $(\nabla\mathcal{S}_+^0f)^*$ is bounded from $L^2(\partial\Omega)$ to $L^2(\partial\Omega)$. Hence, integrating over $\partial\Omega$, we obtain the estimate for $(\nabla\mathcal{S}_+f)^*$. The results for $\mathcal{S}_-f$ follow in a similar manner.
\end{proof}

We now turn to the adjoint of the single layer potential.

\begin{prop}\label{SingleAdjointLayerBounds}
Under the same assumptions as in Proposition~\ref{LayerBoundedness}, for any $F\in W^{-1,2}(\Omega)$, the function $\mathcal{S}_+^*F$ is a $W^{1,2}_{\loc}(\Omega)$ solution to $\mathcal{L}^tu=-\dive(A^t\nabla u+cu)+b\nabla u+du=0$ in $\Omega$, which converges nontangentially, almost everywhere to $\mathcal{S}^*F$ on $\partial\Omega$. Moreover,
\begin{equation*}
\|(\mathcal{S}_+^*F)^*\|_{L^2(\partial\Omega)}\leq C\|F\|_{W^{-1,2}(\partial\Omega)},
\end{equation*}
where $C$ depends on $n,p,\lambda,\alpha,\tau$,$\|A\|_{\infty}$, $\|b\|_{\infty}$,$\|c\|_p,\|d\|_p$ and the Lipschitz character of $\partial\Omega$.
\end{prop}
\begin{proof}
Since $R:W^{1,2}(\partial\Omega)\to W^{-1,2}(\partial\Omega)$ from \eqref{eq:Emb1} is invertible, $F=Rf$ for some $f\in W^{1,2}(\partial\Omega)$. Hence, since $x\mapsto G(p,x)$ and $x\mapsto\nabla_T^pG(p,x)$ are solutions to $\mathcal{L}^tu=0$ in $\Omega$ for any fixed $p\in\partial\Omega$, using the formula in \eqref{eq:FormulaForR} we obtain that $\mathcal{S}_+^*F$ solves $\mathcal{L}^tu=0$ in $\Omega$.

To show the bound on the maximal function, let $q\in\partial\Omega$ and $x\in\Gamma(q)$. Using \eqref{eq:FormulaForR}, we write
\[
\mathcal{S}_+^*F(x)=\int_{\partial\Omega}G(q',x)f(q')\,d\sigma(q')+\int_{\partial\Omega}\nabla_T^{q'}G(q',x)\nabla_Tf(q')\,d\sigma(q')=I_1+I_2.
\]
For $I_1$, note that
\begin{equation*}
|I_1|\leq \int_{\partial\Omega}|q'-x|^{2-n}|f(q')|\,d\sigma(q')\leq C\int_{\partial\Omega}|q'-q|^{2-n}|f(q')|\,d\sigma(q'),
\end{equation*}
since $|x-q'|\geq C|q-q'|$. For $I_2$, if $\tilde{A}$ and $\Gamma$ are as in Proposition~\ref{LayerBoundedness}, we use \eqref{eq:DifGamma} to estimate
\begin{align*}
|I_2|&\leq \int_{\partial\Omega}|\nabla_T^{q'}G(q',x)-\nabla_T^{q'}\Gamma(q',x)||\nabla_Tf(q')|\,d\sigma(q')+\left|\int_{\partial\Omega}\nabla_T^{q'}\Gamma(q',x)\nabla_Tf(q')\,d\sigma(q')\right|\\
&\leq C\int_{\partial\Omega}|q'-q|^{1-n+\delta_{n,p}}|\nabla_Tf(q')|\,d\sigma(q')+|I_3|.
\end{align*}
To bound $I_3$, we write $\nabla^{q'}_T\Gamma(q',x)=\nabla^{q'}\Gamma(q',x)-\left<\nabla^{q'}\Gamma(q',x),\nu(q)\right>\nu(q)$. Then, considering the supremum for $x\in\Gamma(q)$, integrating for $q\in\partial\Omega$ and using Theorem 3.5 in \cite{KenigShen}, we obtain that
\begin{equation*}
\|(\mathcal{S}_+^*F)^*\|_{L^2(\partial\Omega)}\leq C\|f\|_{W^{1,2}(\partial\Omega)}\leq C\|F\|_{W^{-1,2}(\partial\Omega)}.
\end{equation*}
Finally, to show nontangential, almost everywhere convergence to $\mathcal{S}^*F$ we follow the proof of Proposition 9.10 in \cite{thesis}: we first show that this holds in the case $F=Ef$, where $f\in L^2(\partial\Omega)$ and $E$ is defined in \eqref{eq:Emb2}, and the general case follows by density, since $E(L^2(\partial\Omega))$ is dense in $W^{-1,2}(\partial\Omega)$ from the argument right after \eqref{eq:Emb2}. This completes the proof.
\end{proof}

We now turn to the nontangential behavior of $\mathcal{S}_{\pm}f$ on $\partial\Omega$.

\begin{prop}\label{NonTangConv}
Let $u_{\pm}=\mathcal{S}_{\pm}f$ for some $f\in L^2(\partial\Omega)$.
Under the same assumptions as in Proposition~\ref{LayerBoundedness}, for almost every $q\in\partial\Omega$,
\[
\nabla u_{\pm}(x)\to\pm\frac{1}{2}f(q)\left<A(q)\nu(q),\nu(q)\right>^{-1}\nu(q)+\int_{\partial\Omega}\nabla^qG(q,q')f(q')\,d\sigma(q'),
\]
as $x\to q$ nontangentially. In particular, we obtain that $\nabla_Tu_+=\nabla_Tu_-$, and the jump relation $f=(\partial_{\nu}u)_+-(\partial_{\nu}u)_-$ holds.
\end{prop}
\begin{proof}
Let $\tilde{A}$ be an extension as in \eqref{eq:ATilde}, $\tilde{\mathcal{L}}u=-\dive(\tilde{A}\nabla u)$ in $\bR^n$, $\Gamma$ be the fundamental solution for $\tilde{\mathcal{L}}$, and $\tilde{\mathcal{S}}$ be the corresponding layer potential. We then follow the lines of the proof of Theorem 4.4 in \cite{KenigShen}: note first that the formula above holds for $\tilde{\mathcal{S}}$, from Theorem 4.4 in \cite{KenigShen}, and combining with \eqref{eq:DifGamma}, we obtain the analog of (4.10) in \cite{KenigShen} for $G,\Gamma$ in the place of $\Gamma_A,\Theta$, respectively. We then finish the proof continuing as right after (4.10).
\end{proof}

From the bound on the maximal function in Proposition~\ref{SingleLayerBounds}, we have that $\nabla\mathcal{S}_+f\in L^2(\Omega)$, whenever $f\in L^2(\partial\Omega)$. In the next proposition we show that in the special case that $A$ satisfies Condition~\eqref{eq:ACond}, $d=0$ and $\dive c\leq 0$, then a better integrability result holds for $\nabla\mathcal{S}_+f$.

\begin{lemma}\label{qImproved}
Let $\Omega\subseteq\mathbb R^n$ be a Lipschitz domain with $\diam(\Omega)<\frac{1}{16}$, and $A\in\M_{B_{\Omega}}(\lambda,\alpha,\tau)$, $b\in\C_{B_{\Omega}}(\alpha,\tau)$ and $c\in L^p(B_{\Omega})$ for some $p>n$ with $\dive c\leq 0$. Assume also that $A$ satisfies Condition~\eqref{eq:ACond} for some $\alpha_0\leq\alpha$, and let $f\in L^2(\partial\Omega)$. If $\mathcal{L}u=-\dive(A\nabla u+bu)+c\nabla u$, $G$ is Green's function for $\mathcal{L}$ in $B_{\Omega}$ and $\mathcal{S}$ is the corresponding single layer potential, then
\begin{equation*}
\|\nabla\mathcal{S}_{\pm}f\|_{L^{p_1}(\Omega)}\leq C\|f\|_{L^2(\partial\Omega)},
\end{equation*}
for any $p_1\in\left(2,\frac{2n}{n-1}\right)$, where $C$ depends on $n,p,p_1,\lambda,\alpha,\tau$,$\|A\|_{\infty},\|b\|_{\infty},\|c\|_p$, the constants $C_1$ and $\alpha_0$ in Condition~\ref{eq:ACond} and the Lipschitz character of $\Omega$.
\end{lemma}
\begin{proof}
We will treat $\mathcal{S}_+$, as the reasoning for $\mathcal{S}_-$ is similar. Let $\tilde{A}$ be an extension of $A\in\M_{B_{8\rho}}(\lambda,\alpha,\tau)$ as in \eqref{eq:ATilde}, and $\Gamma$ be the fundamental solution for $\tilde{\mathcal{L}}u=-\dive(\tilde{A}\nabla u)$. Define also $v(x)=\int_{\partial\Omega}\Gamma(x,q)f(q)\,d\sigma(q)$. Then, from \cite{KenigShen}, $v$ is the solution to the $R_2$ regularity problem
\[
\left\{\begin{array}{c l}-\dive(\tilde{A}\nabla v)=0, & {\rm in}\,\,\Omega \\ v=\tilde{\mathcal{S}}f, & {\rm on}\,\,\partial\Omega\end{array}\right.,
\]
in the sense of Definition 5.2 in \cite{KenigShen}, where $\tilde{\mathcal{S}}$ is the single layer potential operator for $\tilde{\mathcal{L}}$ in $\Omega$. Hence, from Lemma~\ref{ImprovedIntegrability} and Theorem 4.7 in \cite{KenigShen} we obtain that, for any $p_1\in\left(2,\frac{2n}{n-1}\right)$,
\begin{equation}\label{eq:vFromImprovedBefore}
\|\nabla v\|_{L^{p_1}(\Omega)}\leq C\|\nabla_T\tilde{\mathcal{S}}f\|_{L^2(\partial\Omega)}\leq C\|f\|_{L^2(\partial\Omega)}.
\end{equation}
Let now $w=\mathcal{S}_+f-v$. Then, from \eqref{eq:DifGamma} we obtain, for all $x\in\Omega$,
\[
|\nabla w(x)|\leq\int_{\partial\Omega}|G(x,q)-\Gamma(x,q)||f(q)|\,d\sigma(q)\leq C\int_{\partial\Omega}|x-q|^{1-n+\delta}|f(q)|\,d\sigma(q),
\]
where $\delta=1-\frac{n}{p}>0$. Hence, from the Cauchy-Schwartz inequality and Lemma~\ref{BoundxInsideDelta},
\begin{align*}
|\nabla w(x)|^2&\leq\left(\int_{\partial\Omega}|x-q|^{1-n+\delta}\,d\sigma(q)\right)\cdot\left(\int_{\partial\Omega}|x-q|^{1-n+\delta}|f(q)|^2\,d\sigma(q)\right)\\
&\leq C\int_{\partial\Omega}|x-q|^{1-n+\delta}|f(q)|^2\,d\sigma(q).
\end{align*}
Therefore, from the Minkowski inequality,
\begin{align}\nonumber
\left(\int_{\Omega}|\nabla w(x)|^{2\cdot\frac{n}{n-1}}\right)^{\frac{n-1}{n}}&\leq C\left(\int_{\Omega}\left(\int_{\partial\Omega}|x-q|^{1-n+\delta}|f(q)|^2\,d\sigma(q)\right)^{\frac{n}{n-1}}\,dx\right)^{\frac{n-1}{n}}\\
\nonumber
&\leq C\int_{\partial\Omega}\left(\int_{\Omega}\left(|x-q|^{1-n+\delta}|f(q)|^2\right)^{\frac{n}{n-1}}\,dx\right)^{\frac{n-1}{n}}\,d\sigma(q)\\
\label{eq:SubstForInner}
&=C\int_{\partial\Omega}\left(\int_{\Omega}|x-q|^{-n+\frac{n\delta}{n-1}}\,dx\right)^{\frac{n-1}{n}}|f(q)|^2\,d\sigma(q)\leq C\int_{\partial\Omega}|f|^2\,d\sigma,
\end{align}
where we used the calculation right after (7.32) in \cite{Gilbarg}. Using that $\mathcal{S}_+f=v+w$, \eqref{eq:vFromImprovedBefore} and \eqref{eq:SubstForInner} show that
\[
\|\nabla\mathcal{S}_+f\|_{L^{p_1}(\Omega)}\leq \|\nabla v\|_{L^{p_1}(\Omega)}+\|\nabla w\|_{L^{p_1}(\Omega)}\leq \|\nabla v\|_{L^{p_1}(\Omega)}+|\Omega|^{\frac{1}{p_1}-\frac{n-1}{2n}}\|\nabla w\|_{L^{\frac{2n}{n-1}}(\Omega)}\leq C\|f\|_{L^2(\partial\Omega)},
\]
which completes the proof.
\end{proof}

\subsection{Invertibility of \texorpdfstring{$\mathcal{S}$}{S}: a special case}
We will now show that under the additional assumption \eqref{eq:ACond} for the coefficients $A,b$ and if also $d=0$ and $\dive c\leq 0$, the single layer potential $\mathcal{S}:L^2(\partial\Omega)\to W^{1,2}(\partial\Omega)$ is invertible.

\begin{lemma}\label{InvertibilityUnderCondition}
Let $\Omega$ be a Lipschitz domain with $\diam(\Omega)<\frac{1}{16}$, and let $A\in\M_{B_{\Omega}}(\lambda,\alpha,\tau)$, $b\in C_{B_{\Omega}}(\alpha,\tau)$. Assume, in addition, that $A$ and $b$ satisfy Condition \eqref{eq:ACond}, both in $\Omega$ and $B_{\Omega}\setminus\Omega$, with $\alpha_0\leq \alpha$. Consider also $c\in L^p(B_{\Omega})$ for $p>n$, with $\dive c\leq 0$ in $B_{\Omega}$. Then, the single layer potential $\mathcal{S}:L^2(\partial\Omega)\to W^{1,2}(\partial\Omega)$ for the operator $-\dive(A\nabla u+bu)+c\nabla u$ is invertible, with $\|\mathcal{S}^{-1}\|_{W^{1,2}(\partial\Omega)\to L^2(\partial\Omega)}\leq C$, where $C$ depends on $n,p,\lambda,\alpha,\tau$, $\|A\|_{\infty},\|b\|_{\infty},\|c\|_p$, the constants $C_1$ and $\alpha_0$ that appear in Condition~\eqref{eq:ACond}, and the Lipschitz character of $\Omega$.
\end{lemma}
\begin{proof}
Note that, from Proposition~\ref{SingleLayerBounds}, $u_+=\mathcal{S}_+f$ is a solution of $\mathcal{L}u=0$ in $\Omega$, with $(\nabla u_+)^*\in L^2(\partial\Omega)$. Moreover, $\nabla u_+$ converges nontangentially, almost everywhere on $\partial\Omega$, from Proposition~\ref{NonTangConv}. Therefore, the Rellich estimate (Lemma~\ref{GlobalRellich}) is applicable for $u_+$, so for $\rho\in(0,r_{\Omega})$,
\begin{align}
\nonumber
\int_{\partial\Omega}|\partial_{\nu}u_+|^2\,d\sigma&\leq C\rho^{-1}\int_{\partial\Omega}\left(|u_+|^2+|\nabla_Tu_+|^2\right)\,d\sigma+C\int_{\Omega_{C_0\rho}}|c||\nabla u_+|^2+C\rho^{\alpha_0}\int_{\partial\Omega}|(\nabla u_+)^*|^2\,d\sigma\\
\label{eq:u_+}
&\leq C\rho^{-1}\|\mathcal{S}f\|_{W^{1,2}(\partial\Omega)}^2+C\int_{\Omega_{C_0\rho}}|c||\nabla\mathcal{S}_+f|^2+C\rho^{\alpha_0}\int_{\partial\Omega}|f|^2\,d\sigma,
\end{align}
where we used the bound from Proposition~\ref{SingleLayerBounds} in the last step and the fact that $u_+$ converges to $\mathcal{S}f$ nontangentially, almost everywhere on $\partial\Omega$, and where $C$ depends on $n,p,\lambda,\alpha,\tau$, $\|A\|_{\infty},\|b\|_{\infty}$, $\|c\|_p$, the constants $C_1$ and $\alpha_0$ that appear in Condition~\eqref{eq:ACond}, and the Lipschitz character of $\Omega$. A similar reasoning applies to $u_-=\mathcal{S}_-f$, which is a solution of $\mathcal{L}u=0$ in the Lipschitz domain $B_{\Omega}\setminus\Omega$, and we obtain
\begin{equation}\label{eq:u_-}
\int_{\partial\Omega}|\partial_{\nu}u_-|^2\,d\sigma\leq C\rho^{-1}\|\mathcal{S}f\|_{W^{1,2}(\partial\Omega)}^2+C\int_{\Omega_{C_0\rho}}|c||\nabla\mathcal{S}_-f|^2+C\rho^{\alpha_0}\int_{\partial\Omega}|f|^2\,d\sigma.
\end{equation}
Note now that, from the jump relation (Proposition~\ref{NonTangConv}), we have that
\[
\|f\|_{L^2(\partial\Omega)}=\|\partial_{\nu}u_+-\partial_{\nu}u_-\|_{L^2(\partial\Omega)}\leq \|\partial_{\nu}u_+\|_{L^2(\partial\Omega)}+\|\partial_{\nu}u_-\|_{L^2(\partial\Omega)}.
\]
Let now $p'=\frac{p+n}{2}\in(n,p)$, and $q'\in\left(1,\frac{n}{n-1}\right)$ be the conjugate exponent to $p'$. Then, plugging \eqref{eq:u_+} and \eqref{eq:u_-} in the last estimate and using H{\"o}lder's inequality,
\begin{align*}
\|f\|_{L^2(\partial\Omega)}^2&\leq C\rho^{-1}\|\mathcal{S}f\|_{W^{1,2}(\partial\Omega)}^2+C\int_{\Omega_{C_0\rho}}|c|(|\nabla\mathcal{S}_+f|^2+|\nabla\mathcal{S}_-f|^2)+C\rho^{\alpha_0}\int_{\partial\Omega}|f|^2\,d\sigma\\
&\leq C\rho^{-1}\|\mathcal{S}f\|_{W^{1,2}(\partial\Omega)}^2+C\|c\|_{L^{\tilde{p}}(\Omega_{C_0\rho})}(\|\nabla\mathcal{S}_+f\|_{L^{2q'}(\Omega)}^2+\|\nabla\mathcal{S}_-f\|_{L^{2q'}(\Omega)}^2)+C\rho^{\alpha_0}\|f\|_{L^2(\partial\Omega)}^2\\
&\leq C\rho^{-1}\|\mathcal{S}f\|_{W^{1,2}(\partial\Omega)}^2+C\rho^{1/r}\|c\|_p\|f\|_{L^2(\partial\Omega)}^2+C\rho^{\alpha_0}\|f\|_{L^2(\partial\Omega)}^2,
\end{align*}
where $r=\frac{1}{p'}-\frac{1}{p}$, and where we used Lemma~\ref{SmallArea} and Lemma~\ref{qImproved} for $p_1=2q'$ in the last step.

We now choose $\rho>0$, depending only on $n,p,\lambda,\alpha,\tau$,$\|A\|_{\infty},\|b\|_{\infty}$,$\|c\|_p$, the constants $C_1$ and $\alpha_0$ that appear in Condition~\eqref{eq:ACond}, and the Lipschitz character of $\Omega$, such that $C\rho^{1/r}\|c\|_p+C\rho^{\alpha_0}<\frac{1}{2}$. Then, we obtain that
\begin{equation}\label{eq:InverseIneq}
\|f\|_{L^2(\partial\Omega)}\leq C\|\mathcal{S}f\|_{W^{1,2}(\partial\Omega)}.
\end{equation}

To show invertibility, let $\mathcal{S}_tf$ be the single layer potential for the operator
\[
\mathcal{L}_tu=-\dive((tA+(1-t)I)\nabla u+tbu)+tc\nabla u,
\]
where $t\in[0,1]$. Since $\mathcal{S}_0$ corresponds to the single layer potential for the Laplacian in $\Omega$, with the kernel being Green's function for the Laplacian in $B_{\Omega}$, $\mathcal{S}_0:L^2(\partial\Omega)\to W^{1,2}(\partial\Omega)$ is invertible. Moreover, if $G_t$ is Green's function for $\mathcal{L}_t$ in $B_{\Omega}$, then $\|\mathcal{S}_tf-\mathcal{S}_sf\|_{L^2(\partial\Omega)\to W^{1,2}(\partial\Omega)}\leq C|t-s|$, from Lemma~\ref{GeneralPerturbation}. Therefore, the bound in \eqref{eq:InverseIneq} for $\mathcal{S}_t$ and the continuity method show that $\mathcal{S}_1=\mathcal{S}:L^2(\partial\Omega)\to W^{1,2}(\partial\Omega)$ is invertible, and the bound in \eqref{eq:InverseIneq} completes the proof.
\end{proof}

\section{Invertibility of \texorpdfstring{$\mathcal{S}$}{S}}
\subsection{A Perturbation Lemma}
In order to reduce the general case considered to Lemma~\ref{InvertibilityUnderCondition}, and in order to treat $c$ as a perturbation of $0$, we will use the next lemma.
\begin{lemma}\label{Smalldelta0}
Let $\Omega\subseteq\bR^n$ be a Lipschitz domain with $\diam(\Omega)<\frac{1}{16}$. Let also $A_i\in\M_{B_{\Omega}}(\lambda,\alpha,\tau)$, $b_i\in\C_{B_{\Omega}}(\alpha,\tau)$, $c_i\in L^p(B_{\Omega})$ for some $p>n$ for $i=1,2$, with either $\dive b_1,\dive b_2\leq 0$, or $\dive c_1,\dive c_2\leq 0$. Set
\[
\mathcal{L}_iu=-\dive(A_i\nabla u+b_iu)+c_i\nabla u,
\]
and let $\mathcal{S}_i$ be the single layer potential operator $\displaystyle\mathcal{S}_if(p)=\int_{\partial\Omega}G_i(p,q)f(q)\,d\sigma(q)$, where $G_i$ is Green's function for $\mathcal{L}_i$ in $B_{\Omega}$. Assume also that $\mathcal{S}_1:L^2(\partial\Omega)\to W^{1,2}(\partial\Omega)$ is invertible, with the norm of its inverse being bounded by $N_1>0$. There exists a constant $\delta_0$, depending on $n,p,\lambda,\alpha,\tau$,$\|A_i\|_{\infty}$,$\|b_i\|_{\infty}$,$\|c_i\|_p$ for $i=1,2$, the Lipschitz character of $\Omega$ and $N_1$, such that, if
\[
\|A_1-A_2\|_{C^{\alpha}}+\|b_1-b_2\|_{C^{\alpha}}+\|c_1-c_2\|_p<\delta_0,
\]
then $\mathcal{S}_2:L^2(\partial\Omega)\to W^{1,2}(\partial\Omega)$ is invertible, with $\|\mathcal{S}^{-1}\|_{L^2(\partial\Omega)\to W^{1,2}(\partial\Omega)}\leq 2N_1$.
\end{lemma}
\begin{proof}
Using Lemma~\ref{GeneralPerturbation}, we have that $\|\mathcal{S}_2-\mathcal{S}_1\|_{L^2(\partial\Omega)\to W^{1,2}(\partial\Omega)}\leq C\delta_0$. So, if $\delta_0=\frac{1}{2CN_1}$,
\begin{align*}
\|\mathcal{S}_2\mathcal{S}_1^{-1}-I\|_{L^2(\partial\Omega)\to W^{1,2}(\partial\Omega)}&=\|\left(\mathcal{S}_2-\mathcal{S}_1\right)\mathcal{S}_1^{-1}\|_{L^2(\partial\Omega)\to W^{1,2}(\partial\Omega)}\leq C\delta_0N_1=\frac{1}{2}.
\end{align*}
If $\mathcal{A}=\mathcal{S}_2\mathcal{S}_1^{-1}$, then the previous estimate shows that $\mathcal{A}$ is invertible, and $\left\|\mathcal{A}^{-1}\right\|_{L^2(\partial\Omega)\to W^{1,2}(\partial\Omega)}\leq 2$. Therefore $\mathcal{S}_2$ is invertible, and $\|\mathcal{S}_2^{-1}\|_{L^2(\partial\Omega)\to W^{1,2}(\partial\Omega)}=\|\mathcal{S}_1^{-1}\mathcal{A}^{-1}\|_{L^2(\partial\Omega)\to W^{1,2}(\partial\Omega)}\leq 2N_1$.
\end{proof}

\subsection{The Rellich property}
We now turn to the Rellich property, which will be an equivalent condition for solvability of the $R_2$ Regularity problem. For the next definition we adapt Definition 5.1 in \cite{KenigShen} in our case, the main difference being that only the tangential gradient of $u$ appears on the right hand side.

\begin{dfn}\label{TRel}
Let $\mathcal{L}u=-\dive(A\nabla u+bu)+c\nabla u+du$ in a bounded Lipschitz domain $\Omega$. We say that $\mathcal{L}$ has the $T$-Rellich property in $\Omega$ with constant $C$ if for any $u\in W^{1,2}_{{\rm loc}}(\Omega)$ which solves $\mathcal{L}u=0$ such that $(\nabla u)^*\in L^2(\partial\Omega)$ and $u,$ $\nabla u$ exist nontangentially, almost everywhere on $\partial\Omega$, we have the estimate
\begin{equation}\label{eq:TRel}
\|\nabla u\|_{L^2(\partial\Omega)}\leq C\|u\|_{L^2(\partial\Omega)}+C\|\nabla_Tu\|_{L^2(\partial\Omega)}.
\end{equation}
\end{dfn}

The main result we will show is that, under our assumptions, the $T$-Rellich property is equivalent to invertibility of the single layer potential.

\begin{prop}\label{RelInvEquivalence}
Let $\Omega\subseteq\bR^n$ be a Lipschitz domain with $\diam(\Omega)<\frac{1}{16}$, let $A\in M_{B_{\Omega}}(\lambda,\alpha,\tau)$, $b\in\C_{B_{\Omega}}(\alpha,\tau)$ and $c\in L^p(B_{\Omega})$ for some $p>n$, with $\dive b\leq 0$, or $\dive c\leq 0$. Consider also $\mathcal{L}_tu=-\dive((tA+\lambda(1-t)I)\nabla u+tbu)+tc\nabla u$, for $t\in[0,1]$, and let $\mathcal{S}$ be the single layer potential for the operator $\mathcal{L}=\mathcal{L}_1$.
\begin{enumerate}[i)]
\item If $\mathcal{S}:L^2(\partial\Omega)\to W^{1,2}(\partial\Omega)$ is invertible with $\|\mathcal{S}^{-1}\|_{W^{1,2}(\partial\Omega)\to L^2(\partial\Omega)}\leq N_1$, then the $T$-Rellich property holds for $\mathcal{L}$ in $\Omega$, with a constant that depends on $n,p,\lambda,\alpha,\tau$,$\|A\|_{\infty}$,$\|b\|_{\infty}$,$\|c\|_p$, the Lipschitz character of $\Omega$ and $N_1$.
\item If the $T$-Rellich property holds for $\mathcal{L}_t$ for some constant $C$ both in $\Omega$ and $B_{\Omega}\setminus\Omega$, uniformly in $t\in[0,1]$, then the single layer potential $\mathcal{S}:L^2(\partial\Omega)\to W^{1,2}(\partial\Omega)$ is invertible, with $\|\mathcal{S}^{-1}\|$ being bounded above by a constant that depends on $C$.
\end{enumerate}
\end{prop}
\begin{proof}
Let $u$ be as in Definition~\ref{TRel} and denote by $u_{\partial}\in W^{1,2}(\partial\Omega)$ the nontangential limit of $u$ on $\partial\Omega$. From invertibility of $\mathcal{S}:L^2(\partial\Omega)\to W^{1,2}(\partial\Omega)$, there exists $f\in L^2(\partial\Omega)$ such that $\mathcal{S}f=u_{\partial}$. Then, the function $v=\mathcal{S}_+f-u\in W^{1,2}_{\loc}(\Omega)$ solves $\mathcal{L}v=0$ in $\Omega$ and converges nontangentially, almost everywhere to $0$ on $\partial\Omega$, hence \eqref{eq:SolidByBoundary} shows that $v\equiv 0$ in $\Omega$. Hence, from nontangential convergence,
\begin{align*}
\|\nabla u\|_{L^2(\partial\Omega)}&\leq\|(\nabla u)^*)\|_{L^2(\partial\Omega)}=\|(\nabla\mathcal{S}_+f)^*\|_{L^2(\partial\Omega)}\\
&\leq C\|f\|_{L^2(\partial\Omega)}=C\|\mathcal{S}^{-1}u_{\partial}\|_{L^2(\partial\Omega)}\leq CN_1\|u_{\partial}\|_{W^{1,2}(\partial\Omega)},
\end{align*}
where we used Proposition \ref{SingleLayerBounds} for the second inequality. This completes the proof of the first part.
	
For the second part, let $\mathcal{S}^t$ be the single layer potential for $\mathcal{L}_t$. Let $f\in L^2(\partial\Omega)$ and set $u_+^t=\mathcal{S}^t_+f$, $u_-=\mathcal{S}_-^tf$. Using Propositions~\ref{SingleLayerBounds} and \ref{NonTangConv}, we obtain that \eqref{eq:TRel} is applicable for $u_+^t$ and $u_-^t$. Then, using the jump relation from Proposition~\ref{NonTangConv} and \eqref{eq:TRel}, we obtain that
\[
\|f\|_{L^2(\partial\Omega)}\leq \|\nabla u_+^t\|_{L^2(\partial\Omega)}+\|\nabla u_-^t\|_{L^2(\partial\Omega)}\leq C\|\mathcal{S}^tf\|_{W^{1,2}(\partial\Omega)}.
\]
Then the continuity method shows that $\mathcal{S}:L^2(\partial\Omega)\to W^{1,2}(\partial\Omega)$ is invertible, with $\|\mathcal{S}^{-1}\|\leq C$. 
\end{proof}

The basic fact about the Rellich property that we will use is that it only depends on the behavior of the coefficients near the boundary. For this, we recall the definition of $\Omega_{\sigma}$ from \eqref{eq:CloseFar}.

\begin{lemma}\label{ExtendingRellich}
Let $\Omega\subseteq\bR^n$ be a Lipschitz domain with $\diam(\Omega)<\frac{1}{16}$, and let $A_i\in\M_{B_{\Omega}}(\lambda,\alpha,\tau)$, $b_i\in\C_{B_{\Omega}}(\alpha,\tau)$, $c_i\in L^p(B_{\Omega})$ for some $p>n$ for $i=1,2$, with either $\dive b_1,\dive b_2\leq 0$, or $\dive c_1,\dive c_2\leq 0$. Set 
\[
\mathcal{L}_iu=-\dive(A_i\nabla u+b_iu)+c_i\nabla u,
\]
and suppose that the $T$-Rellich property holds for $\mathcal{L}_1$ with constant $\tilde{C}$. If $A_1=A_2$, $b_1=b_2$ and $c_1=c_2$ in $\Omega_{\sigma}$ for some $\sigma>0$, then the Rellich property holds for $\mathcal{L}_2$, with constant that depends on $n,p,\lambda,\alpha,\tau$, $\|A_i\|_{\infty}$,$\|b_i\|_{\infty}$,$\|c_i\|_p$ for $i=1,2$, the Lipschitz character of $\Omega$, $\sigma$ and $\tilde{C}$.
\end{lemma}
\begin{proof}
Suppose that $u\in W^{1,2}_{{\rm loc}}(\Omega)$ is a solution to $\mathcal{L}_2u=0$ in $\Omega$, such that $(\nabla u)^*\in L^2(\partial\Omega)$ and $u$, $\nabla u$ exist nontangentially almost everywhere on $\partial\Omega$. Then, we compute
\begin{align*}
\mathcal{L}_1u=-\dive((A_1-A_2)\nabla u+(b_1-b_2)u)+(c_1-c_2)\nabla u=-\dive f+g.
\end{align*}
Since $u$ solves the equation $\mathcal{L}_2u=0$ in $\Omega$, Proposition~\ref{CAlphaReg} shows that $u,\nabla u\in L^{\infty}_{\loc}(\Omega)$. Since also $A_1=A_2$, $b_1=b_2$ and $c_1=c_2$ in $\Omega_{\sigma}$, we obtain that $f\in L^{\infty}(\Omega)$ and $g\in L^p(\Omega)$.
	
Extend $f$ and $g$ by $0$ in $B_{\Omega}\setminus\Omega$, and consider the solutions $v_1,v_2\in W_0^{1,2}(B_{\Omega})$ to the equations
\[
\left\{\begin{array}{c l}\mathcal{L}_1v_1=-\dive f, & {\rm in}\,\,B_{\Omega} \\ v_1=0 & {\rm on}\,\,\partial B_{\Omega}\end{array}\right.,\quad \left\{\begin{array}{c l}\mathcal{L}_1v_2=g, & {\rm in}\,\,B_{\Omega} \\ v_2=0 & {\rm on}\,\,\partial B_{\Omega}\end{array}\right.
\]
(the existence of these solutions can be justified be Lemmas 4.2 and 4.4 in \cite{KimSak}). Then, Proposition 6.14 in \cite{KimSak} and Proposition~\ref{SolidByBoundary} (which is applicable, since $\mathcal{L}_2u=0$ in $\Omega$) show that
\begin{align}
\nonumber
\|v_1\|_{W_0^{1,2}(B_{\Omega})}&\leq C\|f\|_{L^2(B_{\Omega})}=C\left(\int_{\Omega\setminus\Omega_{\sigma}}\left|(A_1-A_2)\nabla u+(b_1-b_2)u\right|^2\right)^{1/2}\\
\label{eq:v1Bound}
&\leq C\left(\int_{\Omega}|\nabla u|^2+|u|^2\right)^{1/2}\leq C\|u\|_{W^{1,2}(\partial\Omega)},
\end{align}
where $C$ depends on $n,p,\lambda,\alpha,\tau$,$\|A_i\|_{\infty}$,$\|b_i\|_{\infty}$,$\|c_i\|_p$ and the Lipschitz character of $\Omega$. Consider now $x\in\Omega_{\sigma/4}$. Since $f\equiv 0$ in $\Omega_{\sigma}$, $v_1$ solves the equation $\mathcal{L}_1v_1=0$ in $B_{\sigma/4}(x)$. So, from \eqref{eq:q:CAlphaReg2} and \eqref{eq:v1Bound},
\begin{equation}\label{eq:v_1Bound}
|v_1(x)|\leq C\left(\fint_{B_{\sigma/8}(x)}v_1^2\right)^{1/2}\leq  C\sigma^{-n/2}\|v_1\|_{L^2(B_{\Omega})}\leq C\|u\|_{W^{1,2}(\partial\Omega)},
\end{equation}
and also
\begin{equation}\label{eq:Nablav_1Bound}
|\nabla v_1(x)|\leq\frac{C}{\sigma}\left(\fint_{B_{\sigma/8}(x)}v_1^2\right)^{1/2}\leq  C\sigma^{-1-n/2}\|v_1\|_{L^2(B_{\Omega})}\leq C\|u\|_{W^{1,2}(\partial\Omega)}.
\end{equation}
To bound $|\nabla v_2|$ note that, from Definition 5.1 in \cite{KimSak}, for any $x\in\Omega_{\sigma/2}$,
\[
v_2(x)=\int_{B_{\Omega}}G(x,y)g(y)\,dy=\int_{\Omega\setminus\Omega_{\sigma}}G(x,y)g(y)\,dy.
\]
If now $y\in\Omega\setminus\Omega_{\sigma}$ and $x\in\Omega_{\sigma/2}$, then  $|x-y|>\frac{\sigma}{2}$, therefore, using \eqref{eq:GreenMain} and Proposition~\ref{SolidByBoundary},
\begin{equation}\label{eq:v_2Bound}
|v_2(x)|\leq C\int_{\Omega\setminus\Omega_{\sigma}}|x-y|^{2-n}|g(y)|\,dy\leq C\sigma^{2-n}\int_{\Omega\setminus\Omega_{\sigma}}|g|\leq C\|c_1-c_2\|_2\|\nabla u\|_{L^2(\Omega)}\leq C\|u\|_{W^{1,2}(\partial\Omega)},
\end{equation}
where $C$ depends on $n,p,\lambda,\alpha,\tau$,$\|A_1\|_{\infty},\|b_1\|_{\infty}$,$\|c_1\|_p$,$\|c_2\|_p$, the Lipschitz character of $\Omega$, and $\sigma$. Since $g$ vanishes in $\Omega_{\sigma}$, $v_2$ is a solution of the equation $\mathcal{L}_1v_2=0$ in $B_{\sigma/4}(x)$, for any $x\in\Omega_{\sigma/4}$. Therefore, using \eqref{eq:q:CAlphaReg2} and \eqref{eq:v_2Bound}, we obtain that
\begin{equation}\label{eq:Nablav_2Bound}
|\nabla v_2(x)|\leq\frac{C}{\sigma}\left(\fint_{B_{\sigma/8}(x)}v_2^2\right)^{1/2}\leq C\|u\|_{W^{1,2}(\partial\Omega)}.
\end{equation}
Hence, in the notation of \eqref{eq:MaxFun}, adding \eqref{eq:Nablav_1Bound} and \eqref{eq:Nablav_2Bound}, we obtain that, for almost all $q\in\partial\Omega$,
\begin{equation}\label{eq:CloseBound}
(\nabla(v_1+v_2))_{\sigma/4}^*(q)\leq C\|u\|_{W^{1,2}(\partial\Omega)}.
\end{equation}
Note now that $v_3=u-v_1-v_2\in W^{1,2}_{\loc}(\Omega)$ is a solution to the equation $\mathcal{L}_1v_3=0$ in $\Omega$. Moreover, from \eqref{eq:CloseBound}, we obtain that
\[
(\nabla v_3)_{\sigma/4}^*(q)\leq(\nabla(v_1+v_2))_{\sigma/4}^*(q)+(\nabla u)_{\sigma/4}^*(q)\leq C\|u\|_{W^{1,2}(\partial\Omega)}+(\nabla u)^*(q).
\]
Since $v_3\in W^{1,2}_{\loc}(\Omega)$ and $(\nabla u)^*\in L^2(\partial\Omega)$, the previous estimate shows that $(\nabla v_3)^*\in L^2(\partial\Omega)$. Moreover, $v_1$ and $v_2$ are solutions of $\mathcal{L}_1v=0$ in $B_{\Omega}\setminus(\Omega\setminus\Omega_{\sigma})$, therefore $v_1$, $v_2$, $\nabla v_1$ and $\nabla v_2$ are continuous in $B_{\Omega}\setminus(\Omega\setminus\Omega_{\sigma})$, from Proposition~\ref{CAlphaReg}. Since $u$, $\nabla u$ converge nontangentially almost everywhere on $\partial\Omega$, this implies that $v_3$, $\nabla v_3$ converge nontangentially, almost everywhere on $\partial\Omega$. Therefore, since the $T$-Rellich property holds for $\mathcal{L}_1$ in $\Omega$ with constant $\tilde{C}$, we obtain that
\[
\|\nabla v_3\|_{L^2(\partial\Omega)}\leq\tilde{C}\|u-v_1-v_2\|_{L^2(\partial\Omega)}+\tilde{C}\|\nabla_T(u-v_1-v_2)\|_{L^2(\partial\Omega)}.
\]
Using that $u=v_1+v_2+v_3$ and \eqref{eq:CloseBound}, we obtain that
\begin{align*}
\|\nabla u\|_{L^2(\partial\Omega)}&\leq\|\nabla(v_1+v_2)\|_{L^2(\partial\Omega)}+\|\nabla v_3\|_{L^2(\partial\Omega)}\\
&\leq C\|u\|_{W^{1,2}(\partial\Omega)}+\tilde{C}\|u-v_1-v_2\|_{L^2(\partial\Omega)}+\tilde{C}\|\nabla_T(u-v_1-v_2)\|_{L^2(\partial\Omega)}\\
&\leq (C+\tilde{C})\|u\|_{W^{1,2}(\partial\Omega)}+\tilde{C}\|v_1+v_2\|_{W^{1,2}(\partial\Omega)}\leq C\|u\|_{W^{1,2}(\partial\Omega)},
\end{align*}
where we used \eqref{eq:v_1Bound}, \eqref{eq:Nablav_1Bound}, \eqref{eq:v_2Bound} and \eqref{eq:Nablav_2Bound} in the last step. This completes the proof.
\end{proof}

\subsection{Invertibility of \texorpdfstring{$\mathcal{S}$}{S}: the case \texorpdfstring{$\dive c\leq 0$}{divc<=0}}
We will now show that Lemma~\ref{InvertibilityUnderCondition} holds without the assumption that $A$ and $b$ satisfy Condition~\eqref{eq:ACond}. To do this, we use the coefficient extensions of Section 7 in \cite{KenigShen}. More specifically, we have the next lemma.
\begin{lemma}\label{BarExt}
Let $\Omega$ be a Lipschitz domain with $\diam(\Omega)<\frac{1}{16}$, and let $A\in\M_{B_{\Omega}}(\lambda,\alpha,\tau)$, $b\in \C_{B_{\Omega}}(\alpha,\tau)$. Then there exist $\alpha_0\in(0,\alpha]$, $\tau_0>0$, $C_1>0$, and $\overline{A}\in\M_{B_{\Omega}}(\lambda,\alpha_0,\tau_0),\overline{b}\in\C_{B_{\Omega}}(\alpha_0,\tau_0)$, such that $\overline{A}=A$ and $\overline{b}=b$ on $\partial\Omega$. Moreover, $\overline{A}$ and $\overline{b}$ satisfy Condition~\ref{eq:ACond}: that is, for all $x\in B_{\Omega}$,
\[
|\nabla A(x)|\leq C_1\delta(x)^{\alpha_0-1},\quad |\nabla^2A(x)|\leq C_1\delta(x)^{\alpha_0-2},
\]
where $\delta(x)=\dist(x,\partial\Omega)$, and similarly for $b$. Here, $\alpha_0,\tau_0$ and $C_1$ depend on $n,\lambda,\alpha,\tau$ and the Lipschitz character of $\Omega$.
\end{lemma}
\begin{proof}
As in the proof of Lemma 7.1 in \cite{KenigShen}, we define $\overline{a}_{ij}$, $\overline{b}_i$ in $\Omega$ to be the Poisson extensions of $a_{ij}, b_i$ in $\Omega$, respectively, and in $B_{\Omega}\setminus\Omega$ to be harmonic functions, with boundary values $a_{ij}$, $b_i$ on $\partial\Omega$, respectively, and $\lambda\delta_{ij}$, $0$ on $\partial B_{\Omega}$, respectively.

To obtain the pointwise bounds on the first derivatives, we follow the proof of the same lemma in \cite{KenigShen}. For the bounds on the second derivatives, note that the functions $\partial_k\overline{a}_{ij}$ and $\partial_{kl}\overline{a}_{ij}$ are harmonic in $\Omega$ for any $i,j,k,l=1,\dots n$. Therefore, from the mean value property and Cacciopoli's inequality, for any $x\in\Omega$,
\begin{align*}
|\partial_{kl}\overline{a}_{ij}(x)|&\leq\fint_{B_{\delta(x)/2}(x)}|\partial_{kl}\overline{a}_{ij}(y)|\,dy\leq C\left(\fint_{B_{\delta(x)/2}(x)}|\nabla(\partial_k\overline{a}_{ij})(y)|^2\,dy\right)^{1/2}\\
&\leq\frac{C}{\delta(x)}\left(\fint_{B_{2\delta(x)/3}}|\partial_k\overline{a}_{ij}(y)|^2\,dy\right)^{1/2}\leq \frac{CC1}{\delta(x)}\left(\fint_{B_{2\delta(x)/3}}\delta(y)^{2\alpha_0-2}\,dy\right)^{1/2}\leq CC_1\delta(x)^{\alpha_0-2},
\end{align*}
which completes the proof for $x\in\Omega$. The case when $x\in B_{\Omega}\setminus\Omega$ is similar.
\end{proof}

We now let $\theta_0\in C_c^{\infty}\left(-\frac{1}{2},\frac{1}{2}\right)$ with $0\leq\theta\leq 1$ and $\theta=1$ in $\left(-\frac{1}{4},\frac{1}{4}\right)$. We also define, for $\rho\in\left(0,\frac{1}{8}\right)$ and $x\in B_{\Omega}$,
\begin{gather}\label{eq:A^rho}
\begin{aligned}
A^{\rho}(x)=\theta\left(\frac{\delta(x)}{\rho}\right)A(x)+\left(1-\theta\left(\frac{\delta(x)}{\rho}\right)\right)\overline{A}(x),\\
b^{\rho}(x)=\theta\left(\frac{\delta(x)}{\rho}\right)b(x)+\left(1-\theta\left(\frac{\delta(x)}{\rho}\right)\right)\overline{b}(x).
\end{aligned}
\end{gather}
Then, the proof of Lemma 7.2 in \cite{KenigShen} shows the next lemma.
\begin{lemma}\label{CloseToARho}
If $A$,$b$ are as in Lemma~\ref{BarExt}, and $A^{\rho}, b^{\rho}$ are as in \eqref{eq:A^rho}, then
\[
\|A^{\rho}-\overline{A}\|_{\infty}\leq C\rho^{\alpha_0},\quad\|A^{\rho}-\overline{A}\|_{C^{0,\alpha_0}}\leq C,
\]
and similarly for $b^{\rho}$ and $\overline{b}$, where $C$ depends on $n,\lambda,\alpha,\tau$ and the Lipschitz character of $\Omega$.
\end{lemma}

As a corollary, we obtain invertibility of the single layer potential operator in the case $\dive c\leq 0$.

\begin{thm}\label{InvForc}
Let $\Omega$ be a Lipschitz domain with $\diam(\Omega)<\frac{1}{16}$, $A\in\M_{B_{\Omega}}(\lambda,\alpha,\tau)$ and $b\in\C_{B_{\Omega}}(\alpha,\tau)$. Assume also that $c\in L^p(B_{\Omega})$ for some $p>n$, with $\dive c\leq 0$ in $B_{\Omega}$. Then, the single layer potential $\mathcal{S}:L^2(\partial\Omega)\to W^{1,2}(\partial\Omega)$ for the operator $\mathcal{L}u=-\dive(A\nabla u+bu)+c\nabla u$ in $\Omega$ is invertible, with $\|\mathcal{S}^{-1}\|$ being bounded above by a constant that depends on $n,p,\lambda,\alpha,\tau$,$\|A\|_{\infty}$,$\|b\|_{\infty}$,$\|c\|_p$ and the Lipschitz character of $\Omega$.
\end{thm}
\begin{proof}
Let $t\in[0,1]$. From Lemmas~\ref{BarExt} and \ref{InvertibilityUnderCondition}, the single layer potential $\overline{\mathcal{S}}_t:L^2(\partial\Omega)\to W^{1,2}(\partial\Omega)$ for the operator $\mathcal{L}_t-\dive((t\overline{A}+\lambda(1-t)I)\nabla u+t\overline{b}u)+tc\nabla u$ is invertible with $\|\overline{\mathcal{S}}_t^{-1}\|\leq N_1$, where $N_1$ depends on $n,p,\lambda,\alpha,\tau$,$\|A\|_{\infty}$, $\|b\|_{\infty}$,$\|c\|_p$ and the Lipschitz character of $\Omega$. From Lemma~\ref{Smalldelta0}, there exists $\delta_0>0$ depending on the same constants as $N_1$ such that, if
\begin{equation}\label{eq:Forrho_0}
\|A^{\rho}-\overline{A}\|_{C^{\alpha_0/2}}+\|b^{\rho}-\overline{b}\|_{C^{\alpha_0/2}}\leq\delta_0,
\end{equation}
then the single layer potential $\mathcal{S}^{\rho}_t$ for the operator $\mathcal{L}_t^{\rho}u=-\dive((tA^{\rho}+\lambda(1-t)I)\nabla u+tb^{\rho}u)+tc\nabla u$ is invertible, with $\|(\mathcal{S}_t^{\rho})^{-1}\|\leq 2N_1$.

Using Lemma~\ref{CloseToARho}, we can find $\rho_0\in(0,r_{\Omega})$, depending on $n,p,\lambda,\alpha,\tau$,$\|A\|_{\infty}$,$\|b\|_{\infty}$,$\|c\|_p$ and the Lipschitz character of $\Omega$, such that \eqref{eq:Forrho_0} holds for $\rho_0$; hence, $\|(\mathcal{S}_t^{\rho_0})^{-1}\|\leq 2N_1$. Then, from the first part of Proposition~\ref{RelInvEquivalence}, the $T$-Rellich property holds for $\mathcal{L}_t^{\rho_0}$ in $\Omega$. Since $A^{\rho_0}=A$ and $b^{\rho_0}=b$ in $\Omega_{\rho_0/4}$, Lemma~\ref{ExtendingRellich} shows that the $T$-Rellich property holds for $\mathcal{L}_t$ in $\Omega$, with a constant that depends on the same constants as above. Therefore, the second part of Proposition~\ref{RelInvEquivalence} shows that $\mathcal{S}=\mathcal{S}_1:L^2(\partial\Omega)\to W^{1,2}(\partial\Omega)$ is invertible, which completes the proof.
\end{proof}

\subsection{Invertibility of \texorpdfstring{$\mathcal{S}$}{S}: the case \texorpdfstring{$\dive b\leq 0$}{divb<=0}}
For the case $\dive b\leq 0$, we will consider $c$ as a perturbation of $0$ and we will reduce to the case considered in the previous subsection. This way we avoid passing through the construction of the functions $b^{\rho}$ in \eqref{eq:A^rho}, which do not necessarily satisfy $\dive b^{\rho}\leq 0$, even if we assume that $\dive b\leq 0$.

\begin{thm}\label{InvForb}
Let $\Omega\subseteq\bR^n$ be a Lipschitz domain with $\diam(\Omega)<\frac{1}{16}$, $A\in\M_{B_{\Omega}}(\lambda,\alpha,\tau)$ and $b\in\C_{B_{\Omega}}(\alpha,\tau)$ with $\dive b\leq 0$. Assume also that $c\in L^p(B_{\Omega})$ for some $p>n$. Then, the single layer potential $\mathcal{S}$ for the operator $\mathcal{L}u=-\dive(A\nabla u+bu)+c\nabla u$ in $\Omega$ is invertible, 
with $\|\mathcal{S}^{-1}\|$ being bounded above by a constant that depends on $n,p,\lambda,\alpha,\tau$,$\|A\|_{\infty}$,$\|b\|_{\infty}$,$\|c\|_p$ and the Lipschitz character of $\Omega$.
\end{thm}
\begin{proof}
Note that, from Theorem~\ref{InvForc} for the special case $c=0$, the single layer potential $\mathcal{S}_0:L^2(\partial\Omega)\to W^{1,2}(\partial\Omega)$ for the operator $\mathcal{L}_0u=-\dive(A\nabla u+bu)$ is invertible, with the norm of the inverse being bounded above by a constant that depends on $n,p,\lambda,\alpha,\tau$,$\|A\|_{\infty}$,$\|b\|_{\infty}$ and the Lipschitz character of $\Omega$.

We now set $c_{\rho}=c\cdot\chi_{\Omega_{\rho}}$. Then $c_{\rho}=c$ in $\Omega_{\rho}$, and for $p'=\frac{p+n}{2}$,
\[
\|c_{\rho}\|_{L^{p'}(B_{\Omega})}\leq\|c_{\rho}\|_{L^p(B_{\Omega})}|\Omega_{\rho}|^{\frac{1}{p'}-\frac{1}{p}}\leq \|c\|_p\left(C\rho\right)^{\frac{p-n}{p(p+n)}},
\]
where we also used Lemma~\ref{SmallArea}. Hence, there exists $\rho_0\in(0,r_{\Omega})$, depending on the same constants as above, such that $\|c_{\rho_0}\|_{L^{p'}(\Omega)}\leq\delta_0$, where $\delta_0$ is the constant that appears in Lemma~\ref{Smalldelta0}. Therefore, from the same lemma (applied for $p'$ instead of $p$), the single layer potential $\mathcal{S}_{\rho_0}$ for the operator
\[
\mathcal{L}_{\rho_0}u=-\dive(A\nabla u+bu)+c_{\rho_0}\nabla u
\] 
is invertible. We then continue the proof as in the proof of Theorem~\ref{InvForc}.
\end{proof}

\section{Solvability of the Dirichlet and Regularity problems}

\subsection{The \texorpdfstring{$R_2$}{R2} Regularity problem}
The formulation of the $R_2$ regularity problem now follows.

\begin{dfn}\label{RegProb}
We say that the $R_2$ Regularity problem for the operator $\mathcal{L}$ is solvable, if for every $f\in W^{1,2}(\partial\Omega)$ there exists $u\in W^{1,2}_{\loc}(\Omega)$ which solves the equation $\mathcal{L}u=0$ in $\Omega$, with $(\nabla u)^*\in L^2(\partial\Omega)$ and $u\to f$ nontangentially, almost everywhere on $\partial\Omega$.
\end{dfn}

We now show that, in the cases that we consider, the solution in Definition~\ref{RegProb} is unique.

\begin{prop}\label{UniquenessForR2}
Let $\Omega\subseteq\bR^n$ be a Lipschitz domain, $A\in\M_{B_{\Omega}}(\lambda,\alpha,\tau)$, and $b,c,d\in L^p(B_{\Omega})$ for some $p>n$, with either $d\geq\dive b$ or $d\geq\dive c$. If $u\in W^{1,2}_{\loc}(\Omega)$ solves $-\dive(A\nabla u+bu)+c\nabla u+du=0$ in $\Omega$ with $u\to 0$ nontangentially, almost everywhere on $\partial\Omega$, then $u\equiv 0$.
\end{prop}
\begin{proof}
After scaling, we can assume that $\diam(\Omega)<\frac{1}{8}$. Then, the right hand side of \eqref{SolidByBoundary} is equal to $0$, hence $u\equiv 0$.
\end{proof}

To show existence for $R_2$, we will use the single layer potential operator. We first show the next lemma, which will be used to reduce to the case $d=0$.

\begin{lemma}\label{DivOperator}
Let $B\subseteq\bR^n$ be a ball and let $d\in L^p(B)$ for some $p>n$. Then there exists $e\in W^{1,p}(B)$ such that
\[
\dive e=d,\quad{\rm and}\quad \|e\|_{W^{1,p}}\leq C\|d\|_p,
\]
where $C$ depends on $n,p$, and the radius of $B$.
\end{lemma}
\begin{proof}
Set $d_0=\int_Bd$, then $|d_0|\leq |B|^{1-1/p}\|d\|_p$. Set also $\tilde{d}=d-d_0$. Then $\int_B\tilde{d}=0$, therefore, from Theorem 4.1 in \cite{Acosta}, there exists $\tilde{e}\in W_0^{1,p}(B)$ such that
\[
\dive\tilde{e}=\tilde{d},\quad{\rm and}\quad\|\tilde{e}\|_{W_0^{1,p}(B)}\leq C\|\tilde{d}\|_{L^p(B)},
\]
where $C$ depends on $n,p$ and the radius of $B$. We now consider $x_0\in B$ and we set $e=\tilde{e}+\frac{d_0}{n}(x-x_0)$. Then, we obtain that $\dive e=\dive\tilde{e}+d_0=d$, and also
\[
|e|\leq|\tilde{e}|+\frac{|d_0|}{n}|x-x_0|\leq|\tilde{e}|+\frac{d_0}{n}\diam(B)\leq |\tilde{e}|+\frac{C|d_0|}{n},
\]
where $C$ depends on the radius of $B$, hence $\|e\|_{L^p(B)}\leq C\|d\|_{L^p(B)}$. Moreover,  $|\partial_ie_j|\leq|\partial_i\tilde{e}_j|+\frac{d_0}{n}$, therefore we finally obtain that $\|e\|_{W^{1,p}(B)}\leq C\|d\|_p$.
\end{proof}

We can now show Theorem~\ref{R_2Solvability}.

\begin{proof}[Proof of Theorem~\ref{R_2Solvability}]
After scaling and using Lemma~\ref{DiameterBound}, we can assume that $\diam(\Omega)<\frac{1}{16}$. Note then that, from Lemma~\ref{DivOperator}, we can write $d=\dive e$, for some vector function $e\in W^{1,p}(B_{\Omega})$ with $\|e\|_{W^{1,p}}\leq C\|d\|_p$. 
	
In the case $\dive c\leq d$, we set $\tilde c=c-e$. From Morrey's inequality and the Sobolev inequality that we have that $\|\tilde{c}\|_{L^p(\Omega)}\leq C$, where $C$ depends on $n,p,\alpha,\tau$,$\|d\|_p$ and the Lipschitz character of $\Omega$, and also $\dive\tilde{c}\leq 0$. Hence, from Theorem~\ref{InvForc}, the single layer potential $\tilde{\mathcal{S}}$ for the operator $\tilde{\mathcal{L}}u=-\dive(A\nabla u+bu)+\tilde{c}\nabla u$ is invertible, with $\|\tilde{\mathcal{S}}^{-1}\|_{W^{1,2}(\partial\Omega)\to L^2(\partial\Omega)}\leq C$, for some $C$ that depends on $n,p,\lambda,\alpha,\tau$,$\|A\|_{\infty}$,$\|b\|_{\infty}$,$\|c\|_p$,$\|d\|_p$ and the Lipschitz character of $\Omega$.
	
In the case $\dive b\leq d$, set $\tilde b=b-e$. From Morrey's inequality we have that $\tilde{b}\in\C_{B_{\Omega}}(\alpha_0,\tau_0)$, where $\alpha_0,\tau_0$ depend on $n,p,\alpha,\tau$,$\|d\|_p$ and the Lipschitz character of $\Omega$, and also $\dive\tilde{b}\leq 0$. Hence, from Theorem~\ref{InvForb}, the single layer potential $\tilde{\mathcal{S}}$ for the operator $\tilde{\mathcal{L}}u=-\dive(A\nabla u+\tilde{b}u)+c\nabla u$ is invertible, with $\|\tilde{\mathcal{S}}^{-1}\|_{W^{1,2}(\partial\Omega)\to L^2(\partial\Omega)}\leq C$.

Hence, in all cases, if $f\in W^{1,2}(\partial\Omega)$, then there exists $g\in L^2(\partial\Omega)$ such that $f=\mathcal{S}g$. Then, $u=\mathcal{S}_+g$ is a solution of $\tilde{\mathcal{L}}u=0$ in $\Omega$ from Proposition~\ref{SingleLayerBounds}. We also compute that
\[
\mathcal{L}u=-\dive(A\nabla u+bu)+c\nabla u+du=-\dive(A\nabla u+\tilde{b}u)+\tilde{c}\nabla u=\tilde{\mathcal{L}u}=0.
\]
Moreover, from the same proposition, $u\to f$ nontangentially, almost everywhere on $\partial\Omega$, and also
\[
\|(\nabla u)^*\|_{L^2(\partial\Omega)}\leq C\|g\|_{L^2(\partial\Omega)}\leq C\|\mathcal{S}^{-1}\|\|f\|_{W^{1,2}(\partial\Omega)}\leq C\|f\|_{W^{1,2}(\partial\Omega)}.
\]
Combining with Proposition~\ref{UniquenessForR2}, we obtain that $u$ is the unique solution to the $R_2$ Regularity problem for $\mathcal{L}$ with boundary values $f$, which completes the proof.
\end{proof}

\subsection{The \texorpdfstring{$D_2$}{D2} Dirichlet problem}

We now turn to the formulation of the $D_2$ Dirichlet problem.

\begin{dfn}
We say that the $D_2$ Dirichlet problem for the operator $\mathcal{L}$ is solvable, if for any $f\in L^2(\partial\Omega)$, there exists $u\in W^{1,2}_{\loc}(\Omega)$ which solves the equation $\mathcal{L}u=0$ in $\Omega$, with $u^*\in L^2(\partial\Omega)$ and $u\to f$ nontangentially, almost everywhere on $\partial\Omega$.
\end{dfn}

We will show existence and uniqueness for $D_2$ for the adjoint operators of the ones for which we have established existence and uniqueness for the $R_2$ Regularity problem.

To show uniqueness, we will need the following lemma.

\begin{lemma}\label{BoundForVanishing}
Let $B_r$ be a ball of radius $r$, and suppose that $g\in W^{1,2}(B_r)$. Assume that
\[
\left|\{x\in B_r\big{|}g(x)=0\}\right|\geq cr.
\]
Then, there exists $C>0$, depending only on $n$ and $c$, such that $\displaystyle\|g\|_{L^{2^*}(B_r)}\leq C\|\nabla g\|_{L^2(B_r)}$.
\end{lemma}
\begin{proof}
Since the inequality we want to show is scale invariant, we can assume that $r=1$, so $B_r=B$. Then, from the Sobolev inequality, $\|g\|_{L^{2^*}(B)}\leq C_n\|g\|_{L^2(B)}+C_n\|\nabla g\|_{L^2(B)}$. But, using Exercise 15 on page 291 in \cite{Evans}, we obtain that $\|g\|_{L^2(B)}\leq C\|\nabla g\|_{L^2(B)}$ for some $C$ that depends on $n$ and $c$, and this completes the proof.
\end{proof}

\begin{prop}\label{UniquenessForD_2}
Let $\Omega$ be a Lipschitz domain, $A\in\M_{B_{\Omega}}(\lambda,\alpha,\tau)$, $b\in\C_{B_{\Omega}}(\alpha,\tau)$ and $c,d\in L^p(B_{\Omega})$ for some $p>n$, with either $d\geq\dive b$ or $d\geq\dive c$. If a solution $u\in W^{1,2}_{\loc}(\Omega)$ to the $D_2$ Dirichlet problem for the operator
\[
\mathcal{L}u=-\dive(A\nabla u+cu)+b\nabla u+du
\]
in $\Omega$ exists, then it is unique.
\end{prop}
\begin{proof}
Let $y\in\Omega$ fixed, and let $B_0$ be a small ball centered at $y$ which is compactly contained in $\Omega$. Let also $g$ be Green's function for $\mathcal{L}^t$ in $\Omega$, and set $g_y(\cdot)=g(\cdot,y)$. From Proposition~\ref{CAlphaReg}, $g_y$ is continuous in $\Omega\setminus B_0$ and it is continuously differentiable in a neighborhood of $\partial B_0$. Moreover, $g_y$ vanishes continuously on $\partial\Omega$, hence, from solvability of the $R_2$ Regularity problem for $\mathcal{L}^t$ (Theorem~\ref{R_2Solvability}), we obtain that $\|(\nabla g_y)_{\rho}^*\|_{L^2(\partial\Omega)}\leq C$, for any $\rho\in(0,\delta(y)/4)$.

Consider now $\phi_{\rho}\in C_c^{\infty}(\Omega)$ with $\phi_{\rho}=1$ in $\Omega^{2\rho}$, $\phi_{\rho}=0$ in $\Omega_{\rho}$, and $|\nabla\phi_{\rho}|\leq C\rho^{-1}$. Set also $E_{\rho}=\Omega^{2\rho}\setminus\Omega_{\rho}$. Then, for $\rho<\delta(y)/4$, and using Lemma~\ref{Representation} and that $u$ is a solution of $\mathcal{L}u=0$ in $\Omega$, we obtain
\begin{align*}
u(y)&=u(y)\phi_{\rho}(y)=\int_{\Omega}A^t\nabla g_y\nabla(u\phi_{\rho})+b\nabla(u\phi_{\rho})\cdot g_y+c\nabla g_y\cdot u\phi_{\rho}+dg_yu\phi_{\rho}\\
&=\int_{\Omega}A\nabla\phi_{\rho}\nabla g_y\cdot u-A\nabla u\nabla\phi_{\rho}\cdot g_y+b\nabla\phi_{\rho}\cdot g_yu-c\nabla\phi_{\rho}\cdot g_yu,
\end{align*}
therefore
\begin{equation}\label{eq:|u|Bound}
|u(y)|\leq\frac{C\|A\|_{\infty}}{\rho}\int_{E_{\rho}}\left(|u||\nabla g_υ|+|g_y||\nabla u|\right)+\frac{C}{\rho}\int_{E_{\rho}}|b-c||g_yu|\leq C\int_{\partial\Omega}(\nabla g_y)^*_{3\rho}u^*_{3\rho}d\sigma+\frac{C}{\rho}I_2,
\end{equation}
where we used the argument in (5.18)-(5.19) in \cite{KenigShen} to obtain the second estimate. To bound $I_2$, we set $f=|b-c||g_yu|$, and let $x\in E_{\rho}$. Then $\sigma(B_{3\rho}(x)\cap\partial\Omega)\geq C\rho^{n-1}$, where $C$ depends on the Lipschitz constant for $\partial\Omega$. Therefore,
\begin{equation}\label{eq:I_2WithF}
I_2=\int_{E_{\rho}}f\leq C\rho^{1-n}\int_{E_{\rho}}\int_{B_{3\rho}(x)\cap\partial\Omega}f(x)\,d\sigma(q)dx\leq C\rho^{1-n}\int_{\partial\Omega}\int_{E_{\rho}\cap B_{3\rho}(q)}f(x)\,dxd\sigma(q),
\end{equation}
from Fubini's theorem. For the inner integral, for any $q\in\partial\Omega$, we extend $g_y$ by $0$ in $B_{3\rho}(q)\setminus\Omega$. Since $\Omega$ is a Lipschitz domain, there exists a ball lying in $B_{3\rho}\setminus\Omega$, with radius comparable to $\rho$. Using Lemma~\ref{BoundForVanishing}, we then obtain that $\|g_y\|_{L^{2^*}(B_{3\rho}(q))}\leq C\|\nabla g_y\|_{L^2(B_{3\rho}(q))}$, where $C$ depends on $n$ and the Lipschitz constant for $\Omega$. Hence, using H{\"o}lder's inequality, we compute
\begin{align*}
\int_{E_{\rho}\cap B_{3\rho}(q)}|f(x)|\,dx&\leq\|b-c\|_n\|g_y\|_{L^{2^*}(B_{3\rho}(q))}\|u\|_{L^2(E_{\rho}\cap B_{3\rho}(q))}\leq C\|\nabla g_y\|_{L^2(B_{3\rho}(q))}\|u\|_{L^2(E_{\rho}\cap B_{3\rho}(q))}\\
&\leq C\rho\left(\int_{\Delta_{3\rho}(q)}|(\nabla g_y)^*_{3\rho}|^2\,d\sigma\right)^{1/2}\left(\int_{\Delta_{3\rho}(q)}|u_{3\rho}^*|^2\,d\sigma\right)^{1/2}.
\end{align*}
Plugging in \eqref{eq:I_2WithF} and using the Cauchy-Schwartz inequality, we obtain that
\begin{align*}
I_2&\leq C\rho^{2-n}\left(\int_{\partial\Omega}\left(\int_{\Delta_{3\rho}(q)}\left|(\nabla g_y)^*_{3\rho}\right|^2\,d\sigma\right)\,d\sigma(q)\right)^{1/2}\cdot\left(\int_{\partial\Omega}\left(\int_{\Delta_{3\rho}(q)}\left|u_{3\rho}^*\right|^2\,d\sigma\right)\,d\sigma(q)\right)^{1/2}\\
&\leq C\rho\left\|(\nabla g_y)_{3\rho}^*\right\|_{L^2(\partial\Omega)}\|u_{3\rho}^*\|_{L^2(\partial\Omega)},
\end{align*}
from Fubini's theorem for the iterated integrals. Plugging in \eqref{eq:|u|Bound}, letting $\rho\to 0$ and using that $\|u_{3\rho}^*\|_{L^2(\partial\Omega)}\to 0$ as $\rho\to 0$ shows that $u(y)=0$, which completes the proof.
\end{proof}

We will now use invertibility of the single layer potential and Proposition~\ref{SingleAdjointLayerBounds} to obtain the proof of Theorem~\ref{D_2Solvability}.

\begin{proof}[Proof of Theorem~\ref{D_2Solvability}]
After scaling and using Lemma~\ref{DiameterBound}, we can assume that $\diam(\Omega)<\frac{1}{16}$.

Uniqueness follows from Proposition~\ref{UniquenessForD_2}. For existence, assume first that $\dive c\leq d$, and define $\tilde{\mathcal{L}}u=-\dive(A^t\nabla u+bu)+\tilde{c}\nabla u$, where $\tilde{c}$ is as in the proof of Theorem~\ref{R_2Solvability}. Then, from the same proof, the single layer potential  $\tilde{\mathcal{S}}:L^2(\partial\Omega)\to W^{1,2}(\partial\Omega)$ for $\tilde{\mathcal{L}}$ is invertible.  Hence, the adjoint $\tilde{\mathcal{S}}^*:W^{-1,2}(\partial\Omega)\to L^2(\partial\Omega)$ is invertible. Therefore, if $f\in L^2(\partial\Omega)$, $F=\left(\tilde{\mathcal{S}}^*\right)^{-1}f\in W^{-1,2}(\partial\Omega)$, and also $\|F\|_{W^{-1,2}(\partial\Omega)}\leq C\|f\|_{L^2(\partial\Omega)}$, where $C$ depends on $n,p,\lambda,\alpha,\tau$,$\|A\|_{\infty}$,$\|b\|_{\infty}$,$\|c\|_p$,$\|d\|_p$ and the Lipschitz character of $\Omega$. Setting $u=\tilde{\mathcal{S}}_+^*F$ in $\Omega$, Proposition~\ref{SingleAdjointLayerBounds} shows that $u$ converges to $\tilde{\mathcal{S}}^*F=f$ nontangentially, almost everywhere on $\partial\Omega$, and also
\[
\|u^*\|_{L^2(\partial\Omega)}=\|(\mathcal{S}^*_+F)^*\|_{L^2(\partial\Omega)}\leq C\|F\|_{W^{-1,2}(\partial\Omega)}\leq C\|f\|_{L^2(\partial\Omega)}.
\]
The case $\dive b\leq d$ is treated similarly, using the function $\tilde{b}$ from the proof of Theorem~\ref{R_2Solvability}, and this completes the proof.
\end{proof}

\bibliographystyle{amsalpha}
\bibliography{Bibliography}

\newcommand{\etalchar}[1]{$^{#1}$}
\providecommand{\bysame}{\leavevmode\hbox to3em{\hrulefill}\thinspace}
\providecommand{\MR}{\relax\ifhmode\unskip\space\fi MR }
\providecommand{\MRhref}[2]{%
  \href{http://www.ams.org/mathscinet-getitem?mr=#1}{#2}
}
\providecommand{\href}[2]{#2}
\begin{thebibliography}{HKMP15}

\bibitem[AAA{\etalchar{+}}11]{HofmannAnalyticity}
M.~Angeles Alfonseca, Pascal Auscher, Andreas Axelsson, Steve Hofmann, and
  Seick Kim, \emph{Analyticity of layer potentials and {$L^2$} solvability of
  boundary value problems for divergence form elliptic equations with complex
  {$L^\infty$} coefficients}, Adv. Math. \textbf{226} (2011), no.~5,
  4533--4606.

\bibitem[ADM06]{Acosta}
Gabriel Acosta, Ricardo~G. Dur\'an, and Mar\'\i a~A. Muschietti,
  \emph{Solutions of the divergence operator on {J}ohn domains}, Adv. Math.
  \textbf{206} (2006), no.~2, 373--401.

\bibitem[Anc09]{AnconaElliptic}
Alano Ancona, \emph{Elliptic operators, conormal derivatives and positive parts
  of functions}, J. Funct. Anal. \textbf{257} (2009), no.~7, 2124--2158, With
  an appendix by Ha\"\i m Brezis.

\bibitem[AQ02]{AuscherQafsaoui}
P.~Auscher and M.~Qafsaoui, \emph{Observations on {$W^{1,p}$} estimates for
  divergence elliptic equations with {VMO} coefficients}, Boll. Unione Mat.
  Ital. Sez. B Artic. Ric. Mat. (8) \textbf{5} (2002), no.~2, 487--509.

\bibitem[CMM82]{Coifman}
R.~R. Coifman, A.~McIntosh, and Y.~Meyer, \emph{L'int\'egrale de {C}auchy
  d\'efinit un op\'erateur born\'e sur {$L^{2}$} pour les courbes
  lipschitziennes}, Ann. of Math. (2) \textbf{116} (1982), no.~2, 361--387.

\bibitem[DEK18]{DongEscKim}
Hongjie Dong, Luis Escauriaza, and Seick Kim, \emph{On {$C^1$}, {$C^2$}, and
  weak type-{$(1,1)$} estimates for linear elliptic operators: part {II}},
  Math. Ann. \textbf{370} (2018), no.~1-2, 447--489.

\bibitem[Eva10]{Evans}
Lawrence~C. Evans, \emph{Partial differential equations}, second ed., Graduate
  Studies in Mathematics, vol.~19, American Mathematical Society, Providence,
  RI, 2010.

\bibitem[FJK84]{FabesJerisonKenigNecessary}
Eugene~B. Fabes, David~S. Jerison, and Carlos~E. Kenig, \emph{Necessary and
  sufficient conditions for absolute continuity of elliptic-harmonic measure},
  Ann. of Math. (2) \textbf{119} (1984), no.~1, 121--141.

\bibitem[Gra08]{Grafakos}
Loukas Grafakos, \emph{Classical {F}ourier analysis}, second ed., Graduate
  Texts in Mathematics, vol. 249, Springer, New York, 2008.

\bibitem[GT01]{Gilbarg}
David Gilbarg and Neil~S. Trudinger, \emph{Elliptic partial differential
  equations of second order}, Classics in Mathematics, Springer-Verlag, Berlin,
  2001, Reprint of the 1998 edition.

\bibitem[HKMP15]{Hofmannduality}
Steve Hofmann, Carlos Kenig, Svitlana Mayboroda, and Jill Pipher, \emph{The
  regularity problem for second order elliptic operators with complex-valued
  bounded measurable coefficients}, Math. Ann. \textbf{361} (2015), no.~3-4,
  863--907.

\bibitem[JK80]{JerisonKenigIdentity}
David~S. Jerison and Carlos~E. Kenig, \emph{An identity with applications to
  harmonic measure}, Bull. Amer. Math. Soc. (N.S.) \textbf{2} (1980), no.~3,
  447--451.

\bibitem[JK81a]{JerisonKenigNonsmooth}
\bysame, \emph{The {D}irichlet problem in nonsmooth domains}, Ann. of Math. (2)
  \textbf{113} (1981), no.~2, 367--382.

\bibitem[JK81b]{JerisonKenigN2R2}
\bysame, \emph{The {N}eumann problem on {L}ipschitz domains}, Bull. Amer. Math.
  Soc. (N.S.) \textbf{4} (1981), no.~2, 203--207.

\bibitem[JK82]{JerisonKenigBoundary}
\bysame, \emph{Boundary behavior of harmonic functions in nontangentially
  accessible domains}, Adv. in Math. \textbf{46} (1982), no.~1, 80--147.

\bibitem[JK95]{JerisonKenigInhomogeneous}
David Jerison and Carlos~E. Kenig, \emph{The inhomogeneous {D}irichlet problem
  in {L}ipschitz domains}, J. Funct. Anal. \textbf{130} (1995), no.~1,
  161--219.

\bibitem[JW84]{JonssonWallin}
Alf Jonsson and Hans Wallin, \emph{Function spaces on subsets of {${\bf
  R}^n$}}, Math. Rep. \textbf{2} (1984), no.~1, xiv+221.

\bibitem[KKPT00]{KenigKochPipherToro}
C.~Kenig, H.~Koch, J.~Pipher, and T.~Toro, \emph{A new approach to absolute
  continuity of elliptic measure, with applications to non-symmetric
  equations}, Adv. Math. \textbf{153} (2000), no.~2, 231--298.

\bibitem[KP93]{KenigPipherNeumann}
Carlos~E. Kenig and Jill Pipher, \emph{The {N}eumann problem for elliptic
  equations with nonsmooth coefficients}, Invent. Math. \textbf{113} (1993),
  no.~3, 447--509.

\bibitem[KP01]{PipherDrifts}
\bysame, \emph{The {D}irichlet problem for elliptic equations with drift
  terms}, Publ. Mat. \textbf{45} (2001), no.~1, 199--217.

\bibitem[KR09]{KenigRule}
Carlos~E. Kenig and David~J. Rule, \emph{The regularity and {N}eumann problem
  for non-symmetric elliptic operators}, Trans. Amer. Math. Soc. \textbf{361}
  (2009), no.~1, 125--160.

\bibitem[KS11a]{KenigShenHomogenization}
Carlos~E. Kenig and Zhongwei Shen, \emph{Homogenization of elliptic boundary
  value problems in {L}ipschitz domains}, Math. Ann. \textbf{350} (2011),
  no.~4, 867--917.

\bibitem[KS11b]{KenigShen}
\bysame, \emph{Layer potential methods for elliptic homogenization problems},
  Comm. Pure Appl. Math. \textbf{64} (2011), no.~1, 1--44.

\bibitem[KS17]{KimSak}
S.~{Kim} and G.~{Sakellaris}, \emph{{Green's function for second order elliptic
  equations with singular lower order coefficients}}, arXiv:1712.01188
  [math.AP] (2017).

\bibitem[Ngu16]{NguyenPaper}
Nguyen~T. Nguyen, \emph{The {D}irichlet and regularity problems for some second
  order linear elliptic systems on bounded {L}ipschitz domains}, Potential
  Anal. \textbf{45} (2016), no.~1, 167--186.

\bibitem[Sak17]{thesis}
Georgios Sakellaris, \emph{Boundary {V}alue {P}roblems in {L}ipschitz {D}omains
  for {E}quations with {D}rifts}, ProQuest LLC, Ann Arbor, MI, 2017, Thesis
  (Ph.D.)--The University of Chicago.

\bibitem[Ver84]{Verchota}
Gregory Verchota, \emph{Layer potentials and regularity for the {D}irichlet
  problem for {L}aplace's equation in {L}ipschitz domains}, J. Funct. Anal.
  \textbf{59} (1984), no.~3, 572--611.

\bibitem[Xu16]{XuUniformSystems}
Qiang Xu, \emph{Uniform regularity estimates in homogenization theory of
  elliptic system with lower order terms}, J. Math. Anal. Appl. \textbf{438}
  (2016), no.~2, 1066--1107.

\bibitem[XZZ18]{XuZhaoZhou}
Q.~{Xu}, P.~{Zhao}, and S.~{Zhou}, \emph{{The Methods of Layer Potentials for
  General Elliptic Homogenization Problems in Lipschitz Domains}},
  arXiv:1801.09220 [math.AP] (2018).

\end{thebibliography}

\end{document}